\documentclass[a4]{article}

\usepackage[pdftex]{graphicx}
\usepackage{xspace}
\usepackage{natbib}
\bibpunct[,]{(}{)}{,}{a}{}{,}
\usepackage{enumitem}
\usepackage{tikz}
\usepackage{hyperref}

\usepackage{authblk}
\usepackage{latexsym}
\usepackage{amsmath}
\usepackage{amssymb}
\usepackage{amsthm}
\usepackage{url}
\usepackage{color}

\usepackage{mathrsfs} % mathscr

\theoremstyle{definition}
\newtheorem{thm}{Theorem}[section]
\newtheorem{lem}[thm]{Lemma}
\newtheorem{defi}[thm]{Definition}
\newtheorem{prop}[thm]{Proposition}
\newtheorem{cor}[thm]{Corollary}

\newtheorem{claim}[thm]{Claim}
\newtheorem{rmk}[thm]{Remark}

\usepackage[ruled]{algorithm2e}
\SetKw{KwReturn}{return}
\SetKwInput{KwRequire}{Require}

\newcommand{\argmin}{\operatornamewithlimits{argmin}}
\newcommand{\argmax}{\operatornamewithlimits{argmax}}

\newcommand{\ee}{\mathrm{e}}

\newcommand{\EE}{\mathbb{E}}
\newcommand{\RR}{\mathbb{R}}

\newcommand{\NN}{\mathbb{N}}

\newcommand{\mcB}{\mathcal{B}}
\newcommand{\mcC}{\mathcal{C}}

\newcommand{\mcK}{\mathcal{K}}
\newcommand{\mcL}{\mathcal{L}}

\newcommand{\mcN}{\mathcal{N}}

\newcommand{\mcQ}{\mathcal{Q}}

\newcommand{\mcV}{\mathcal{V}}

\newcommand{\mbV}{\mathbf{V}}
\newcommand{\mbD}{\mathbf{D}}

\newcommand{\Norm}[1]{\left \lVert #1 \right \rVert}
\newcommand{\norm}[1]{\lVert #1 \rVert}

\newcommand{\dist}{\mathord{\mathrm{dist}}}
\newcommand{\cone}{\mathord{\mathrm{cone}}}
\newcommand{\aff}{\mathord{\mathrm{aff}}}
\newcommand{\Span}{\mathord{\mathrm{span}}}
\newcommand{\bigO}{\mathord{\mathrm{O}}}

\newcommand{\set}[1]{\{#1\}}

\title{Estimating Piecewise Monotone Signals}
\author[1,2]{Kentaro Minami}
\affil[1]{The University of Tokyo}
\affil[2]{Preferred Networks, Inc.}
\date{7 March 2020}

\begin{document}
\maketitle

%\subjclass[2010]{62G08}
\begin{abstract}
We study the problem of estimating piecewise monotone vectors.
This problem can be seen as a generalization of the isotonic regression that allows a small number of order-violating changepoints.
We focus mainly on the performance of the nearly-isotonic regression proposed by Tibshirani et al. (2011).
We derive risk bounds for the nearly-isotonic regression estimators that are adaptive to piecewise monotone signals.
The estimator achieves a near minimax convergence rate over certain classes of piecewise monotone signals under a weak assumption.
Furthermore, we present an algorithm that can be applied to the nearly-isotonic type estimators on general weighted graphs.
The simulation results suggest that the nearly-isotonic regression performs as well as the ideal estimator that knows the true positions of changepoints.

\noindent \textbf{keywords}: piecewise monotone function, isotonic regression, nearly-isotonic regression, adaptive risk bounds
\end{abstract}

%\keywords{piecewise monotone function, isotonic regression, nearly-isotonic regression, adaptive risk bounds}

\tableofcontents

\section{Introduction}

Isotonic regression is a popular statistical method based on partial order structures, which has a long history in statistics \citep{Ayer1955, Brunk1955, vanEeden1956}.
Suppose that $\theta^* \in \RR^n$ is a monotone vector satisfying $\theta^*_1 \leq \theta^*_2 \leq \cdots \leq \theta^*_n$, and $y$ is a noisy observation of $\theta^*$.
The goal of the isotonic regression is to find a least-square fit under the monotone constraint:
\begin{equation}\label{eq:isotonic_regression_1}
    \text{minimize} \ \norm{y - \theta}_2 \quad
    \text{subject to} \ \theta_1 \leq \theta_2 \leq \cdots \leq \theta_n.
\end{equation}
In other words, the isotonic regression is the least squares estimator $\hat{\theta} = \hat{\theta}_{K_n^\uparrow}$ over a closed convex cone $K^\uparrow_n := \set{\theta \in \RR^n: \theta_1 \leq \theta_2 \leq \cdots \leq \theta_n}$.
Broadly speaking, the isotonic regression is an example of \textit{shape restricted regression}.
For comprehensive reviews on this field, see \citet{Robertson88,Groeneboom,Chatterjee2015,Guntuboyina2017b} and references therein.

In this paper, we study the problem of estimating \textit{piecewise monotone} vectors, which can be regarded as a generalization of isotonic regression that allows order-violating changepoints.
We formulate the problem precisely as follows.
Let us consider the Gaussian sequence model
\begin{equation}\label{eq:gaussian_sequence}
    y_i = \theta^*_i + \xi_i, \quad i = 1, 2, \ldots, n,
\end{equation}
where $y = (y_1, y_2, \ldots, y_n)^\top \in \RR^n$ is the observed vector, $\theta^* = (\theta^*_1, \theta^*_2, \ldots, \theta^*_n)^\top \in \RR^n$ is the unknown parameter of interest, and $\xi = (\xi_1, \xi_2, \ldots, \xi_n)^\top$ is the unobserved noise distributed according to the Gaussian distribution $N(0, \sigma^2 I_n)$.
Given the noisy observation $y$, the problem is to find a good piecewise monotone approximation of $\theta^*$.
Here we define piecewise monotone vectors as follows.

\begin{defi}\label{def:piecewise_monotone}
Let $\Pi = (A_1, A_2, \ldots, A_m)$ be a connected partition of $[n] = \set{1, 2, \ldots, n}$, that is, there exists a sequence $1 = \tau_1 < \tau_2 < \cdots < \tau_{m} < \tau_{m+1} = n + 1$ such that $A_i = \set{\tau_{i}, \tau_{i} + 1, \ldots, \tau_{i+1} - 1}$ ($i = 1, 2, \ldots, m$).
We say that a vector $\theta \in \RR^n$ is \textit{piecewise monotone} on $\Pi$ if the restriction on each $A_i$ is monotone:
\[
    \theta_{\tau_i} \leq \theta_{\tau_i + 1} \leq \cdots \leq \theta_{\tau_{i+1} - 1}, \quad \text{for $i= 1, 2, \ldots, m$}.
\]
We also say that $\theta$ is $m$-piecewise monotone if $\theta$ is piecewise monotone on some partition $\Pi$ with $|\Pi| = m$.
\end{defi}

We are particularly interested in the case where the number of pieces $m$ is larger than two but much smaller than $n$ because it is reduced to simpler problems if otherwise.
From Definition \ref{def:piecewise_monotone}, a monotone vector in $K_n^\uparrow$ is $m$-piecewise monotone for any $m \geq 1$.
In particular, the least squares estimators over $1$-piecewise monotone vectors coincide with the isotonic regression.
Besides, since any vector in $\RR^n$ is $n$-piecewise monotone, the least squares estimator over $n$-piecewise monotone vectors is merely the identity function $\hat{\theta}_{\mathrm{id}} = y$.

\begin{figure}[tb]

\begin{minipage}[c]{\linewidth}
\centering
\includegraphics[width=0.95\linewidth]{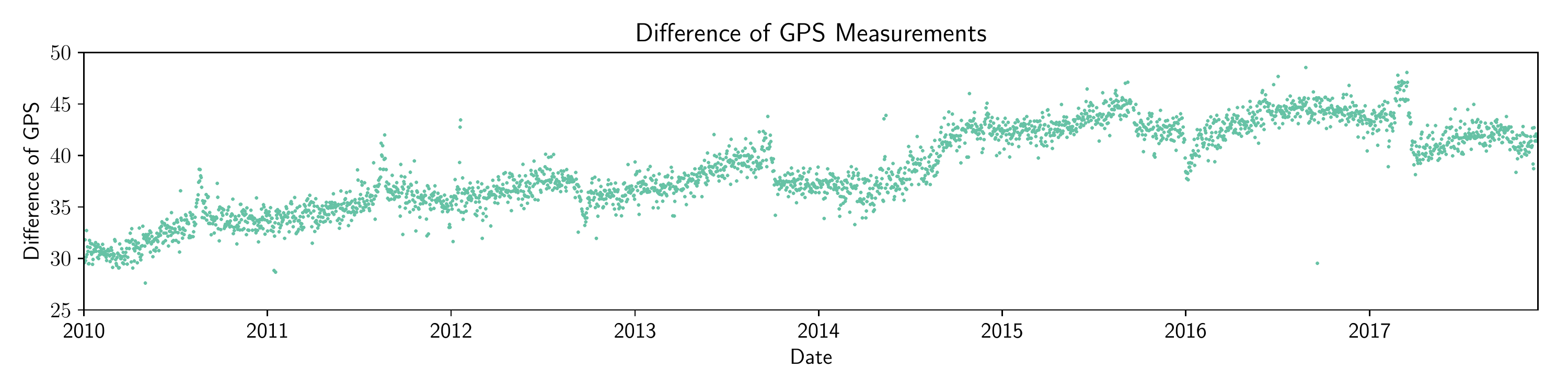}
\end{minipage}

\begin{minipage}[c]{\linewidth}
\centering
\includegraphics[width=0.95\linewidth]{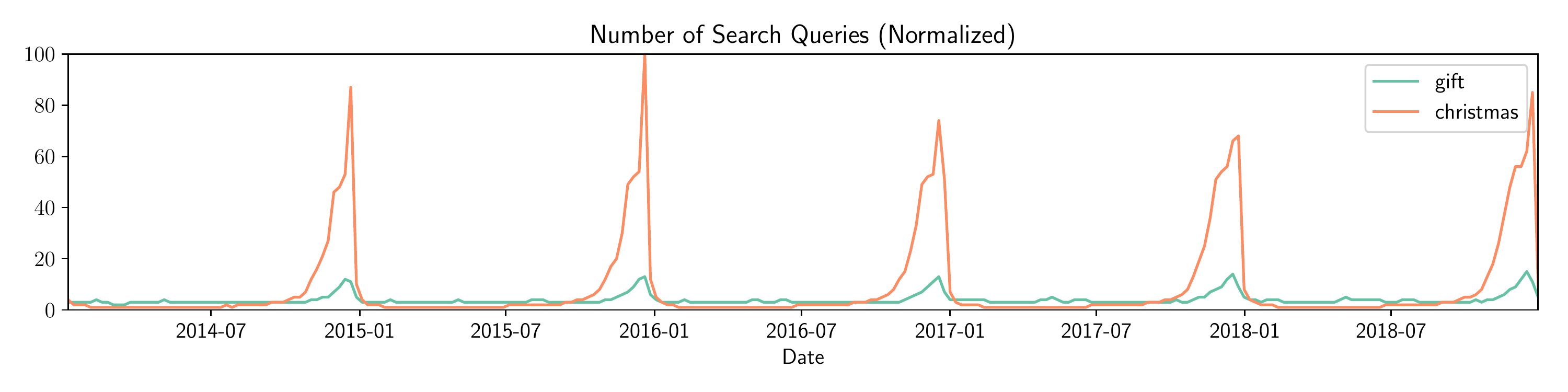}
\end{minipage}

\caption{\textbf{Examples of piecewise monotone signals in real-world data.}
\textbf{Top}: The difference of the east-west component of GPS measurements between Victoria (British Columbia, Canada) and Seattle (United States).
The trend factor seems to be approximated by a piecewise monotone signal.
A possible reason for this behavior is the seismological phenomenon reported in \citet{Rogers2003}.
See Section \ref{sec:geological} for a more detailed explanation of this data.
\textbf{Bottom}: The numbers of search queries for two words ``Christmas'' and ``gift'' in Google Trends (\url{https://www.google.com/trends}).}
\label{fig:real_data_example}
\end{figure}

In real-world applications, there are many signals that can be approximated by piecewise monotone vectors.
Here, we provide a few examples.
First, in seismology, geological observations such as tide gauge records \citep{Nagao2013} and GPS records \citep{Rogers2003} often consist of a long-term monotonic trend and discontinuous jumps caused by tectonic activities.
In particular, \citet{Rogers2003} reported that GPS measurements that are nearby a subduction zone in North America can be approximated by a sawtooth function.
The top panel of Figure \ref{fig:real_data_example} shows an example of GPS measurements.
Second, the numbers of search queries for some words related to seasons (e.g., ``Christmas'' and ``gift'') can be seen as periodic piecewise monotone signals (see the bottom panel of Figure \ref{fig:real_data_example} for examples).
Third, in the ranking systems in online shopping websites, sales ranks of rarely sold items behave like piecewise monotone signals because they suddenly rise every time the items are sold \citep{Hattori2010}.

In this paper, we focus on the performance of \textit{nearly-isotonic regression} proposed by \citet{Tibshirani2011}.
Given $y \in \RR^n$ and a tuning parameter $\lambda \geq 0$, the nearly-isotonic regression estimator $\hat{\theta}_\lambda$ is defined as
\begin{equation}\label{eq:neariso}
    \hat{\theta}_\lambda \in \argmin_{\theta \in \RR^n}
    \left\{
        \frac{1}{2}\norm{y - \theta}_2^2 +
        \lambda \sum_{i=1}^{n-1}(\theta_i - \theta_{i+1})_+
    \right\},
\end{equation}
where $(z)_+ := \max\{ z, 0 \}$.
Intuitively, the tuning parameter $\lambda$ controls the degree of monotonicity.
The term $(\theta_i - \theta_{i+1})_+$ poses a positive penalty if and only if the directed edge $(i, i+1)$ is \textit{order violating}, i.e., $\theta_i > \theta_{i+1}$.
Hence, a large value of $\lambda > 0$ makes the estimator $\hat{\theta}_\lambda$ close to a monotone vector.
In particular, there is a sufficiently large $\lambda$ such that the solution $\hat{\theta}_\lambda$ becomes exactly the same as the isotonic regression \eqref{eq:isotonic_regression_1}.

\begin{figure}[tb]
\centering
\includegraphics[width=0.95\linewidth]{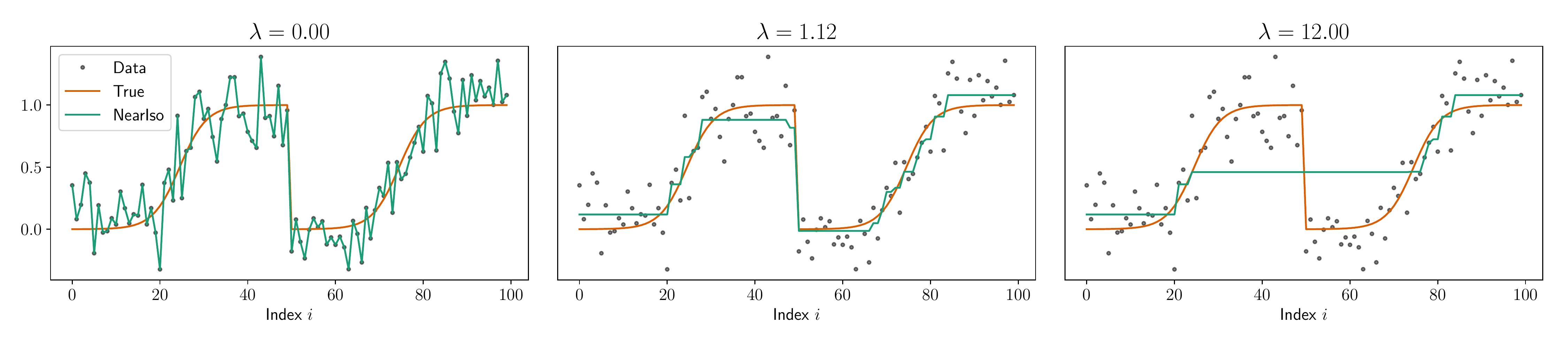}
\caption{\textbf{Examples of the nearly-isotonic regression estimators with different choices of tuning parameters.}
The nearly-isotonic regression interpolates between the identity estimator $\hat{\theta}_\mathrm{id} = y$ and the isotonic regression $\hat{\theta}_{K_n^\uparrow}$.
}
\label{fig:example_1dim}
\end{figure}

Our goal in this paper is to show that the nearly-isotonic regression can adapt to piecewise monotone vectors.
As suggested in \citet{Tibshirani2011}, the nearly-isotonic regression can fit to a ``nearly monotone'' vector that is close to $K^\uparrow_n$ in $\ell_2$-sense.
That is, the estimator performs well if $\theta^*$ has a small $\ell_2$-misspecification error $\dist(\theta^*, K_n^\uparrow)$ defined as
\[
    \dist(\theta^*, K_n^\uparrow) :=
    \inf_{\theta \in K_n^\uparrow} \norm{\theta^* - \theta}_2.
\]
Moreover, we can observe that the nearly-isotonic regression can fit to piecewise monotone vectors, even if $\theta^*$ is far from monotone in $\ell_2$-sense.
Figure \ref{fig:example_1dim} shows an example of the nearly-isotonic regression with $n = 100$.
The true parameter $\theta^*$ (orange line) is 2-piecewise monotone.
By varying the values of the tuning parameter $\lambda \geq 0$, the nearly-isotonic regression behaves as follows:
If $\lambda = 0$, the nearly-isotonic regression is just the identity estimator $\hat{\theta}_{\mathrm{id}} = y$, which clearly overfits to the noisy observation.
If $\lambda$ is set to a sufficiently large value, $\hat{\theta}_\lambda$ coincides with the isotonic regression.
In this example, however, the $\ell_2$-misspecification error $\mathrm{dist}^2(\theta^*, K_{n}^\uparrow)$ is large compared with the normalized noise variance $\sigma^2 / n$.
We can see that the mean squared error (MSE) $\frac{1}{n} \EE_{\theta^*} \norm{\hat{\theta} - \theta^*}_2^2$ of the isotonic regression can be much worse than that of the identity estimator, which coincides with $\sigma^2/n$ (see Section \ref{sec:lower_projection}).
Indeed, we can choose a 2-piecewise monotone vector $\theta^* \in K^\uparrow_{n/2} \times K^\uparrow_{n/2}$ with arbitrarily large $\ell_2$-misspecification error.
If we choose an intermediate value of $\lambda$, the nearly-isotonic regression seems to fit to the true parameter.
This suggests the adaptation property to piecewise monotone vectors.

\subsection{Summary of theoretical results}

In this paper, we investigate the adaptation property of the nearly-isotonic regression estimators defined in \eqref{eq:neariso}.

In the monotone regression setting (i.e., $m = 1$), it is known that the isotonic regression estimator $\hat{\theta}_{K_n^\uparrow}$ achieves the risk bound
\[
    \frac{1}{n} \EE_{\theta^*} \norm{\hat{\theta}_{K_n^\uparrow} - \theta^*}_2^2
    \leq C \left( \frac{\sigma^2 \mcV(\theta^*)}{n} \right)^{2/3}
    + \frac{C \sigma^2 \log \ee n}{n},
\]
where $\mcV(\theta) = \theta_n - \theta_1$ is the total variation of the monotone vector $\theta$.
It is also known that the rate $\bigO((\sigma^2 \mcV/n)^{2/3})$ is minimax optimal under the assumption that $\theta^*$ is monotone and $\mcV(\theta^*) \leq \mcV$ \citep{Zhang2002}.
Hence, a natural question is whether a similar rate can be achieved in piecewise monotone regression.

In Section \ref{sec:lower_general}, we provide the minimax lower bound over the class of piecewise monotone vectors.
Let $\Theta_n(m, \mcV)$ be the set of $m$-piecewise monotone vectors whose ``upper'' total variations are bounded by $\mcV$ (a precise definition is provided in Section \ref{sec:lower_general}).
Then, the minimax risk over $\Theta_n(m, \mcV)$ is bounded from below by a constant multiple of
\[
    \max \left\{
        \left( \frac{\sigma^2 \mcV}{n} \right)^{2/3}, \
        \frac{\sigma^2 m}{n} \log \frac{\ee n}{m}
    \right\}.
\]
In Section \ref{sec:model_selection}, we construct a concrete (but not computationally efficient) estimator that adaptively achieves this rate, and hence this lower bound is tight in the sense of the order in $n, m$, and $\mcV$.
Intuitively, this suggest that the cost of not knowing the true partition is of order $\bigO(\frac{\sigma^2 m}{n} \log \frac{\ee n}{m})$.

In Section \ref{sec:adaptation}, we provide the following risk bound for the nearly-isotonic regression estimator \eqref{eq:neariso}.
A precise statement is given in Corollary \ref{cor:moderate}.

\begin{claim}\label{claim:moderate}
Let $\theta^*$ be a piecewise monotone vector on a partition $\Pi = (A_1, A_2, \ldots, A_m)$.
Suppose that the following assumptions hold:
\begin{enumerate}[label=(\alph*)]
    \item The partition is equi-spaced: $|A_1| = |A_2| = \cdots = |A_m| \ (= \frac{n}{m})$.
    \item For each segment $A_j$, $\theta_{A_j}^*$ is monotone and the total variation is bounded as $\mcV(\theta_{A_j}^*) \leq \mcV/m$.
    \item $\theta^*_{A_j}$ satisfies an appropriate ``growth condition'' for each $j = 1, \ldots, m$.
\end{enumerate}
Then, the estimator \eqref{eq:neariso} with optimally tuned parameter $\lambda$ satisfies the following risk bound:
\begin{equation}\label{eq:moderate_bound}
    \frac{1}{n} \EE_{\theta^*} \norm{\hat{\theta}_\lambda - \theta^*}_2^2
    \leq C \left\{
        \left(
            \frac{\sigma^2 \mcV \log \ee n}{n}
        \right)^{2/3}
        + \frac{\sigma^2 m}{n} \log \frac{\ee n}{m}
    \right\}.
\end{equation}
\end{claim}

The above claim is obtained as a corollary of a more general risk bound in Section \ref{sec:adaptation}.
In the above statement, we make somewhat restrictive assumptions.
Here, (a) and (b) are introduced just for the sake of notation simplicity, whereas (c) is an essential assumption.
If we assume only (a) and (b), the rate that appeared in \eqref{eq:moderate_bound} is minimax optimal up to a logarithmic multiplication factor.
However, we require an extra growth condition (c), which seems to be unavoidable for the estimator \eqref{eq:neariso}.
We will provide a precise definition of the growth condition in Section \ref{sec:piecewise_application}.

\subsection{Organization}

The rest of this paper is organized as follows.
In Section \ref{sec:related}, we give a brief literature review on the shape restricted regression and regularization based estimators and relate our theoretical results to previous work.
We provide lower bounds on the risks in the piecewise monotone regression problem in Section \ref{sec:lower}.
In Section \ref{sec:adaptation}, we describe our main results on the risk upper bounds for the nearly-isotonic regression estimator and its constrained form variant.
In particular, a precise statement of Claim \ref{claim:moderate} in the above is provided in Section \ref{sec:piecewise_application}.
In Section \ref{sec:model_selection}, we discuss the attainability of the minimax lower bound;
herein, we provide a concrete example of a model selection-based estimator that achieves the optimal rate.
%In Section \ref{sec:algorithms}, we review the algorithms for the nearly-isotonic regression and related estimators and discuss their computational complexities.
Furthermore, we present some numerical examples in Section \ref{sec:simulations}.
Finally, we present our conclusion in Section \ref{sec:discussion}.
We have also included appendices which contain additional numerical examples on two-dimensional signals, explanations of algorithms, and all proofs of the theoretical results.

\subsection{Notation}

Throughout this paper, we assume that $y = \theta^* + \xi$ is distributed according to an isotropic normal distribution $N(\theta^*, \sigma^2 I_n)$, where $\theta^* \in \RR^n$ is the true mean parameter of interest and $\xi \sim N(0, \sigma^2 I_n)$ is the noise vector.
The symbol $\EE_{\theta^*}$ denotes the expectation with respect to $y$.

We sometimes denote by $C$ an absolute positive constant whose value may vary.

For any $\theta \in \RR^n$, we define the total variation $\mcV(\theta)$ and the \textit{lower total variation} $\mcV_-(\theta)$ by
\[
    \mcV(\theta) := \sum_{i=1}^{n-1} |\theta_i - \theta_{i+1}| \quad \text{and} \quad
    \mcV_-(\theta) := \sum_{i=1}^{n-1} (\theta_i - \theta_{i+1})_+,
\]
where $(z)_+ := \max\set{z, 0}$ for any $z \in \RR$.
For example, if $\theta$ is monotone nondecreasing, then $\mcV(\theta) = \theta_n - \theta_1$ and $\mcV_-(\theta) = 0$. 
In this paper, the meaning of subscripts of $\theta$ depends on the context (e.g., $\theta_i$, $\theta_A$, $\hat{\theta}_\lambda$, and $\hat{\theta}_{K_n^\uparrow}$).
If $A = \set{\tau, \tau + 1, \ldots, \tau + J - 1}$ is a connected subset of $[n]$, we denote by $\theta_{A}$ a sub-vector $(\theta_\tau, \theta_{\tau + 1}, \ldots, \theta_{\tau + J - 1})^\top \in \RR^J$.
We also denote by $\mcV^A(\theta_A)$ the total variation of $\theta_A$.

\section{Related work}\label{sec:related}

There are two classes of estimators that are closely related to the nearly-isotonic regression \eqref{eq:neariso}: the isotonic regression and the fused lasso.

As we mentioned above, the isotonic regression is an instance of shape restricted regression.
Many existing estimators in shape restricted regression can be formulated as least squares estimators (denoted by $\hat{\theta}_K$) onto closed convex sets (denoted by $K$).
Examples include, but not limited to, the isotonic regression, the isotonic regression in two-dimensional grid or more general partial orders (see e.g., \citet{Robertson1975} and \citet{Kyng2015}), and convex regression \citep{Hildreth1954}.

Recently, researchers have developed two important techniques for analyzing risk behaviors of least squares estimators.
First, \citet{Chatterjee2014} proved that the Euclidean norm $\norm{\hat{\theta}_K - \theta^*}_2$ is tightly concentrated around a certain quantity defined by the \textit{localized Gaussian width}.
As applications of Chatterjee's method, non-asymptotic upper bounds that have similar rates to the minimax risks have been proved for the isotonic regression \citep{Chatterjee2014, Bellec2015b}, the multi-isotonic regression on two or more high dimension \citep{Chatterjee2018, Han2017}, the multi-dimensional convex regression \citep{Han2016}, and the constrained form trend filtering estimator \citep{Guntuboyina2017a}.
See also Section 2.2 in \citet{Bellec2015b} for a related result.
Second, risk bounds based on the \textit{statistical dimension}  of the tangent cone of $K$ has been developed by \citet{Oymak2016} and \citet{Bellec2015b}.
This technique is useful because it takes into account the facial structure of $K$, which leads to risk bounds that are adaptive to low dimensional sub-structures.
It has been shown that some least squares estimators are adaptive to piecewise constant vectors: for example, the isotonic regression \citep{Bellec2015b} and the multi-isotonic regression \citep{Chatterjee2018, Han2017}.
In particular, for the one-dimensional isotonic regression, \citet{Chatterjee2015} and \citet{Bellec2015b} proved the following oracle inequality
\begin{equation}\label{eq:bellec_chatterjee_oracle_inequality}
    \frac{1}{n} \EE_{\theta^*}\norm{\hat{\theta}_{K^\uparrow_n} - \theta^*}_2^2
    \leq \inf_{\theta \in K_n^\uparrow} \left\{
        \frac{1}{n} \norm{\theta - \theta^*}_2^2
        + \frac{\sigma^2 k(\theta)}{n} \log \frac{\ee n}{k(\theta)}
    \right\},
\end{equation}
where $k(\theta)$ is the number of constant pieces of $\theta$.
If $\theta^*$ is monotone and $k(\theta^*)$ is small, the right-hand side can be much smaller than the worst-case rate of $\bigO((\sigma^2 \mcV/n)^{2/3})$.
However, the first term in the right-hand side can become arbitrarily large if $\theta^*$ is not included in $K_n^\uparrow$.

The fused lasso \citep{TSRZK05}, also known as the total variation regularization \citep{Rudin1992}, is a penalized estimator defined as
\begin{equation}\label{eq:total_variation}
    \hat{\theta}_{\mathrm{fused}, \lambda} = \argmin_{\theta \in \RR^n} \left\{
        \frac{1}{2} \norm{y - \theta}_2^2
        + \lambda \sum_{i=1}^{n-1} |\theta_i - \theta_{i+1}|
    \right\},
\end{equation}
where $\lambda \geq 0$ is the tuning parameter.
The fused lasso poses the penalty whenever $\theta_i \neq \theta_{i+1}$, whereas the penalty of the nearly-isotonic regression \eqref{eq:neariso} activates only if $\theta_i > \theta_{i+1}$.
Theoretical risk bounds for the fused lasso have been studied by \citet{Mammen1997}, \citet{Dalalyan17}, \citet{Lin2017}, and \citet{Guntuboyina2017a}.
In particular, \citet{Guntuboyina2017a} showed an oracle inequality of the following form:
\begin{equation}\label{eq:guntuboyina_oracle_inequality}
    \frac{1}{n} \EE_{\theta^*} \norm{\hat{\theta}_{\mathrm{fused}, \lambda^*} - \theta^*}_2^2
    \leq \inf_{\theta \in \RR^n} \left\{
        \frac{1}{n} \norm{\theta - \theta^*}_2^2
        + C \frac{\sigma^2 k(\theta)}{n} \log \frac{\ee n}{k(\theta)}
        + C \Delta_\mathrm{fused}(\theta)
    \right\},
\end{equation}
where $\lambda^*$ is an optimally tuned parameter.
One can control the quantity $\Delta_\mathrm{fused}(\theta)$ by assuming a mild regularity condition on $\theta^*$ so that the inequality \eqref{eq:guntuboyina_oracle_inequality} recovers the minimax rate for the piecewise constant vectors (see e.g., \citet{Gao2017}).
However, even if $\theta^*$ is a monotone vector, \eqref{eq:guntuboyina_oracle_inequality} does not recover the rate of the isotonic regression \eqref{eq:bellec_chatterjee_oracle_inequality} because $\Delta_\mathrm{fused}(\theta)$ becomes zero if and only if $\theta$ is just a constant vector.

Our risk bound for the nearly-isotonic regression in Section \ref{sec:penalized} fills the gap between the above risk bounds for the isotonic regression and the fused lasso.
We will show an oracle inequality of the following form:
\begin{equation*}
    \frac{1}{n} \EE_{\theta^*} \norm{\hat{\theta}_{\mathrm{neariso}, \lambda^*} - \theta^*}_2^2
    \leq \inf_{\theta \in \RR^n} \left\{
        \frac{1}{n} \norm{\theta - \theta^*}_2^2
        + C \frac{\sigma^2 k(\theta)}{n} \log \frac{\ee n}{k(\theta)}
        + C \Delta_\mathrm{neariso}(\theta)
    \right\}.
\end{equation*}
Like in the case of the fused lasso \eqref{eq:guntuboyina_oracle_inequality}, this inequality provides a meaningful risk bound even if we cannot approximate $\theta^*$ by a monotone vector.
Furthermore, $\Delta_\mathrm{neariso}(\theta)$ becomes zero for any monotone vector $\theta \in K_n^\uparrow$.
Hence, our result can exactly recover the rate achieved by the isotonic regression \eqref{eq:bellec_chatterjee_oracle_inequality}.

\section{Lower bounds}\label{sec:lower}

In this section, we provide lower bounds for the risk in one-dimensional piecewise monotone regression.

\subsection{Minimax lower bound}\label{sec:lower_general}

We are interested in the lower bound for the minimax risk defined as
\[
    \inf_{\hat{\theta}} \sup_{\theta^* \in \Theta}
    \frac{1}{n} \EE_{\theta^*} \norm{\hat{\theta} - \theta^*}_2^2,
\]
where $\Theta \subset \RR^n$ is a set of piecewise monotone vectors, and the infimum is taken over all (measurable) estimators $\hat{\theta}: \RR^n \to \RR^n$.
In particular, for $1 \leq m \leq n$, we consider the class of $m$-piecewise monotone vectors with a bounded total variation that is defined as follows.

\begin{defi}\label{defi:piecewise_class}
Let $n \geq 2$ and $1 \leq m \leq n$.
For any $\mcV > 0$, let $\tilde{\Theta}_n (m, \mcV)$ denote the set of (at most) $m$-piecewise monotone vectors such that the upper total variation is bounded by $\mcV$.
In other words, a vector $\theta \in \RR^n$ is an element of $\tilde{\Theta}_n(m, \mcV)$ if and only if the following conditions hold:
\begin{enumerate}[label=(\roman*)]
    \item $\theta$ is piecewise monotone on a connected partition $\Pi = \{ A_1, \ldots, A_{m^*} \}$ of $[n]$ whose cardinality $|\Pi| = m^*$ is not larger than $m$.
    \item There exist numbers $\mcV_1, \mcV_2, \ldots, \mcV_{m^*}$ such that $\sum_{i=1}^{m^*} \mcV_i = \mcV$, $\mcV_i \geq 0$, and $\mcV(\theta_{A_i}) \leq \mcV_i$ for all $i = 1, \ldots, m^*$.
\end{enumerate}
In addition, we also define $\Theta_n(m, \mcV)$ as the set of $m$-piecewise monotone vectors such that the total variations for all pieces are uniformly bounded by $\mcV / m$.
That is, $\Theta_n(m, \mcV)$ is obtained by replacing (ii) by the following condition:
\begin{enumerate}[label=(\roman*)', start=2]
    \item $\mcV(\theta_{A_i}) \leq \mcV / m$ for all $i = 1, \ldots, m^*$.
\end{enumerate}
\end{defi}

First, we consider $\theta^*$ is piecewise monotone on a \textit{known} partition $\Pi^* = \{ A_1, A_2, \ldots, A_{m^*} \}$ and that the total variation of the sub-vector $\theta^*_{A_i}$ is bounded as $\mcV(\theta_{i}^*) \leq \mcV_i$ for each $i = 1, 2, \ldots, m^*$.
Then, the problem is decomposed into $m^*$ independent subproblems of estimating monotone vectors $\theta_i^*$.
The minimax risk lower bound for monotone vectors has been proved by \citet{Zhang2002} and \citet{Chatterjee2015}.
For simplicity in the notation, we assume here that $n_i = |A_i| \geq 2$ for all $i = 1, 2, \ldots, m$.
The minimax risk can be written as
\begin{equation}\label{eq:monotone_lower}
    \inf_{\hat{\theta}_i}
    \sup_{\substack{\theta^*_{A_i} \in K_{A_i}^\uparrow:\\
    \mcV(\theta^*_i) \leq \mcV_i}}
    \frac{1}{n_i} \EE_{\theta^*_{A_i}} \norm{\hat{\theta}_i - \theta^*_i}_2^2
    \geq C_1 \left( \frac{\sigma^2 \mcV_i}{n_i} \right)^{2/3}
    \quad \text{for all $i = 1, \ldots, m$}.
\end{equation}
%Note that the upper bound can be achieved by the least square estimators.
Hence, the minimax risk over $\tilde{\Theta}_n (m, \mcV)$ is clearly bounded from below by
\begin{equation}\label{eq:monotone_lower_sum}
    C_1 \sum_{i=1}^{m^*} \frac{n_i}{n} \left( \frac{\sigma^2 \mcV_i}{n_i} \right)^{2/3}.
\end{equation}
If the partition $\Pi^*$ is known, then this convergence rate can be obtained by concatenating the least squares estimators on all pieces.
By Jensen's inequality, the quantity \eqref{eq:monotone_lower_sum} is not larger than $(\sigma^2 \sum_i \mcV_i / n)^{2/3}$.

In the general setting, we have to deal with \textit{unknown} partitions.
The following proposition gives the lower bound over the class of piecewise monotone vectors in Definition \ref{defi:piecewise_class}.

\begin{prop}\label{prop:general_lower_bound}
Let $n \geq 3$, $3 \leq m \leq n$, and $\mcV > 0$.
Suppose that $\Theta$ is either $\tilde{\Theta}_n(m, \mcV)$ or $\Theta_n(m, \mcV)$ in Definition \ref{defi:piecewise_class}.
Then, for any estimator $\hat{\theta}: \RR^n \to \RR^n$, we have the following lower bound:
\begin{equation}\label{eq:general_lower_bound}
    \sup_{\theta^* \in \Theta} \frac{1}{n}\EE_{\theta^*} \norm{\hat{\theta} - \theta^*}_2^2
    \geq C \max\left\{
    \left( \frac{\sigma^2 \mcV}{n} \right)^{2/3},
    \quad
    \frac{\sigma^2 m}{n} \log \frac{\ee n}{m}
    \right\},
\end{equation}
where $C > 0$ is a universal constant.
\end{prop}

It remains to verify that the lower bound \eqref{eq:general_lower_bound} is tight.
Thus, in Section \ref{sec:model_selection}, we will construct an estimator that adaptively achieves a similar rate.

\subsection{Lower bound of isotonic regression with misspecified partitions}\label{sec:lower_projection}

Suppose that $\theta^*$ is an $m$-piecewise monotone vector.
As we mentioned in the previous subsection, if we know the true partition on which $\theta^*$ is monotone, the least squares estimator can achieve the rate shown in \eqref{eq:monotone_lower_sum}.
Here, we consider what happens if we underestimate the true number of the pieces.

%For notation simplicity, we assume that $m = 2$ and $\theta^*$ is piecewise constant on a partition $\Pi^* = \{ A_1, A_2 \}$.
We consider the risk behavior of the isotonic regression $\hat{\theta}_{K_n^\uparrow}$, which corresponds to the least squares estimator for the underestimated number of pieces as $m = 1$.
If the true number of pieces is larger than or equal to two, $\theta^*$ may not be contained in $K_n^\uparrow$.
Recall that $\mathrm{dist}(\theta^*, K_n^\uparrow)$ is the $\ell_2$-misspecification error against the set of monotone vectors.
\citet{Bellec2015b} showed that the isotonic regression is robust against a small $\ell_2$-misspecification, that is, if $\mathrm{dist}(\theta^*, K_n^\uparrow) \leq \epsilon$, then
\[
    \frac{1}{n}\EE_{\theta^*} \norm{\hat{\theta}_{K_n^\uparrow} - \theta^*}_2^2
    \leq \epsilon^2 + \frac{\sigma^2 k(\bar{\theta})}{n} \log \frac{\ee n}{k(\bar{\theta})},
\]
where $k(\bar{\theta})$ is the orthogonal projection of $\theta^*$ onto $K_n^\uparrow$.
Conversely, if the $\ell_2$-misspecification error is large, we see that the isotonic regression can have an arbitrarily large risk.
%We observe the following sub-optimality result under the large $\ell_2$-misspecification setting.

\begin{prop}\label{prop:suboptimal}
There is a positive number $t = t_{n, \sigma^2}$ that depends on $n$ and $\sigma^2$ such that if the true parameter $\theta^*$ satisfies $\mathrm{dist}(\theta^*, K_n^\uparrow) > t$, then the MSE of the isotonic regression is bounded from below as
\[
    \frac{1}{n}\EE_{\theta^*} \norm{\hat{\theta}_{K_n^\uparrow} - \theta^*}_2^2
    > \sigma^2.
\]
In this case, the isotonic regression has a strictly larger MSE than that of the identity estimator $\hat{\theta}_{\mathrm{id}} = y$.
\end{prop}

We can easily check that there is a 2-piecewise monotone vector with an arbitrarily large $\ell_2$-misspecification error. To see this, let $\theta^* \in \RR^{2n}$ be a piecewise constant vector defined as $\theta^*_i = M > 0$ for $i = 1, \ldots, n$ and $\theta^*_i = 0$ for $i = n+1, \ldots, 2n$.
Then, it is easy to see that $\mathrm{dist}(\theta^*, K^{\uparrow}_{2n}) = \sqrt{nM^2/2}$ diverges as $M \to \infty$.
Figure \ref{fig:example_1dim} shows an example of a 2-piecewise monotone vector $\theta^*$ such that the isotonic regression has a larger squared loss value than the identity estimator.

\section{Risk bounds for nearly-isotonic regression}\label{sec:adaptation}

In this section, we develop the risk bound for the nearly-isotonic regression estimator \eqref{eq:neariso}.
Proofs of all the theorems and propositions in this section are presented in Appendix \ref{sec:appendix_proof_adaptation}.

\subsection{Risk bounds for constrained estimators}\label{sec:constrained}

Before considering the original version of the nearly-isotonic regression \eqref{eq:neariso}, we consider the performance of the \textit{constrained form nearly-isotonic regression} $\hat{\theta}_{\mcV}$ defined by the following constrained optimization problem:
\begin{align}\label{eq:neariso_constrained}
    \text{minimize} \; \lVert y - \theta \rVert_2^2 \quad
    \text{subject to} \; \sum_{i=1}^{n-1} (\theta_i - \theta_{i+1})_+ \leq \mcV,
\end{align}
where $\mcV \geq 0$ is the tuning parameter.
By the fundamental duality theorem in convex optimization, there exists a Lagrange multiplier $\lambda_{\mcV} \geq 0$ such that the regularization type formulation \eqref{eq:neariso} admits the same solution $\hat{\theta}_{\lambda_\mcV} = \hat{\theta}_{\mcV}$.
Hence, the solution path of penalized estimators $\set{\hat{\theta}_\lambda: \lambda \geq 0}$ and that of constrained estimators $\set{\hat{\theta}_\mcV: \mcV \geq 0}$ are equivalent.
However, the properties of estimators with fixed values of $\lambda \geq 0$ and $\mcV \geq 0$ can be different in the following sense:

\begin{itemize}
    \item From a computational perspective, calculating the constrained estimator \eqref{eq:neariso_constrained} for a given $\mcV \geq 0$ is more difficult than the regularization estimator \eqref{eq:neariso}.
    For the regularization estimator \eqref{eq:neariso}, we can use the Modified Pool Adjacent Violators Algorithm (Modified PAVA) proposed by \citet{Tibshirani2011}, which outputs the solution path for every $\lambda \geq 0$.
    In particular, given $\lambda \geq 0$, we can always obtain an \textit{exact} solution $\hat{\theta}_\lambda$.
    However, to the best of our knowledge, there are no practical algorithms that obtain an exact solution for the constrained problem \eqref{eq:neariso_constrained} that run as fast as the algorithms for the penalized problem \eqref{eq:neariso}.
    We present detailed explanations for the algorithms in Section \ref{sec:algorithms}.
    \item From a statistical perspective, the correspondence between tuning parameters $\lambda$ and $\mcV$ is not deterministic (i.e., it depends on the realization of the data $y$).
    For this reason, a risk bound that is obtained for one of \eqref{eq:neariso} or \eqref{eq:neariso_constrained} cannot be directly applied to the other.
\end{itemize}

We show the main results on the adaptation property to piecewise monotone vectors in terms of sharp oracle inequality.

Before proceeding, we introduce some notations.
Suppose that $\theta \in \RR^n$ is piecewise constant on a connected partition $\Pi_\mathrm{const} = \{ A_1, \ldots, A_k \}$ of $[n]$.
We denote by $k(\theta) := |\Pi_{\mathrm{const}}|$ the number of pieces in which $\theta$ becomes constant.
That is, there are integers $1 = \tau_1 < \cdots < \tau_{k + 1} = n + 1$ such that
(i) $A_i = \set{\tau_i, \tau_i + 1, \ldots, \tau_{i+1} - 1}$ for $i = 1, \ldots, k$
and (ii) for any $i \in [k]$, there exists $t_i \in \RR$ such that $\theta_j = t_i$ for all $j \in A_i$.
We define the sign $w_i \in \set{ 0, 1}$ associated with each knot $\tau_i$ ($i = 1, \ldots, k + 1$) as
\begin{align}\label{eq:sign_definition}
    w_1 & = w_{k+1} = 0 \quad \text{and} \nonumber \\
    w_i & = \left\{ \begin{aligned}
    1 & \quad (t_{i-1} > t_{i})\\
    0 & \quad (t_{i-1} < t_{i})
    \end{aligned} \right.
    \quad \text{for $i = 2, \ldots, k$}.
\end{align}
In other words, $w_i = 1$ if and only if the order violation $\theta_{j - 1} > \theta_{j}$ occurs at $j = \tau_i$.
See Figure \ref{fig:signs} for the graphical illustration.
Then, we define $M(\theta)$ as
\begin{equation}\label{eq:m_definition}
    M(\theta) := \sum_{j=2}^k \max \left\{ \frac{1}{|A_j|}, \frac{k}{n} \right\} 1_{\set{ w_{j-1} \neq w_j }}.
\end{equation}
$M(\theta)$ determines the non-monotonicity of a piecewise constant vector $\theta$.
If $\theta$ is $m$-piecewise monotone, then it is clear that $M(\theta) \leq 2(m - 1)$.
In particular, for any monotone vector $\theta$, we have $M(\theta) = 0$.
Based on these notations, we have the following sharp oracle inequality.

\begin{figure}[tb]
    \centering
    \begin{tikzpicture}
\draw [<->, thin] (0, 2.0) -- (0, -0.7) -- (8, -0.7);
\node [below] at (8, -0.7) {$i$};
\node [left] at (0, 1.5) {$\theta_i$};

\node [below] at (0.5, -0.7) {$A_1$};
\node [above, color=teal] at (0, 0) {$w_1 = 0$};
\fill (0, 0) circle (2pt);
\draw (0, 0) -- (1, 0);
\draw [dashed] (1, 0) -- (1, 0.5);
\draw [fill=white] (1, 0) circle (2pt);

\node [below] at (1.5, -0.7) {$A_2$};
\node [above, color=teal] at (1, 0.5) {$w_2 = 0$};
\fill (1, 0.5) circle (2pt);
\draw (1, 0.5) -- (2, 0.5);
\draw [dashed] (2, 0.5) -- (2, 1);
\draw [fill=white] (2, 0.5) circle (2pt);

\node [below] at (2.5, -0.7) {$A_3$};
\node [above, color=teal] at (2, 1) {$w_3 = 0$};
\fill (2, 1) circle (2pt);
\draw (2, 1) -- (3, 1);
\draw [dashed] (3, 1) -- (3, -0.5);
\draw [fill=white] (3, 1) circle (2pt);

\node [below] at (3.5, -0.7) {$A_4$};
\node [above, color=teal] at (3, -0.5) {$w_4 = 1$};
\fill (3, -0.5) circle (2pt);
\draw (3, -0.5) -- (4, -0.5);
\draw [dashed] (4, -0.5) -- (4, 0.5);
\draw [fill=white] (4, -0.5) circle (2pt);

\node [below] at (4.5, -0.7) {$A_5$};
\node [above, color=teal] at (4, 0.5) {$w_5 = 0$};
\fill (4, 0.5) circle (2pt);
\draw (4, 0.5) -- (5, 0.5);
\draw [dashed] (5, 0.5) -- (5, 0);
\draw [fill=white] (5, 0.5) circle (2pt);

\node [below] at (5.5, -0.7) {$A_6$};
\node [above, color=teal] at (5, 0) {$w_6 = 1$};
\fill (5, 0) circle (2pt);
\draw (5, 0) -- (6, 0);
\draw [dashed] (6, 0) -- (6, -0.5);
\draw [fill=white] (6, 0) circle (2pt);

\node [below] at (6.5, -0.7) {$A_7$};
\node [above, color=teal] at (6, -0.5) {$w_7 = 1$};
\fill (6, -0.5) circle (2pt);
\draw (6, -0.5) -- (7, -0.5);
\draw [dashed] (7, -0.5) -- (7, 1.5);
\draw [fill=white] (7, -0.5) circle (2pt);

\node [below] at (7.5, -0.7) {$A_8$};
\node [above, color=teal] at (7, 1.5) {$w_8 = 0$};
\fill (7, 1.5) circle (2pt);
\draw (7, 1.5) -- (8, 1.5);
\draw [fill=white] (8, 1.5) circle (2pt);

\node [above, color=teal] at (8.4, 1.5) {$w_9 = 0$};
\end{tikzpicture}
    \caption{\textbf{Illustration of the knot signs defined in \eqref{eq:sign_definition}.}
    In this example, $\theta$ is assumed to be $k$-piecewise constant with $k = 8$.
    The corresponding signs are given as $(w_1, w_2, \ldots, w_8, w_9) = (0, 0, 0, 1, 0, 1, 1, 0, 0)$.
    Moreover, if we assume $|A_1| = |A_2| = \cdots = |A_8|$, the quantity $M(\theta)$ defined in \eqref{eq:m_definition} is given as $M(\theta) = \frac{1}{|A_4|} + \frac{1}{|A_5|} + \frac{1}{|A_6|} + \frac{1}{|A_8|} = \frac{4k}{n}$.
    }
    \label{fig:signs}
\end{figure}
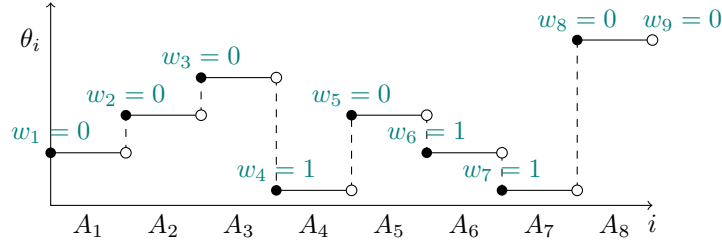

\begin{thm}\label{thm:constrained_tangent}
For any $\theta^* \in \RR^n$, the constrained nearly-isotonic regression \eqref{eq:neariso_constrained} satisfies the following oracle inequality:
\begin{align}\label{eq:thm_constrained_tangent_1}
    & \frac{1}{n} \EE_{\theta^*} \lVert \hat{\theta}_{\mathcal{V}} - \theta^* \rVert_2^2 \nonumber \\
    & \leq \inf_{\substack{\theta \in \RR^n:\\
    \mcV_-(\theta) = \mathcal{V}}}
    \bigg\{
        \frac{1}{n} \lVert \theta - \theta^* \rVert_2^2 
        + C\sigma^2 \frac{k(\theta)}{n} \log \frac{\ee n}{k(\theta)}
        + C\sigma^2 \frac{M(\theta)}{k(\theta)} \log \frac{\ee n}{k(\theta)}
    \bigg\}.
\end{align}
Moreover, for any $\eta \in (0, 1)$, we have
\begin{align}
    & \frac{1}{n} \lVert \hat{\theta}_{\mathcal{V}} - \theta^* \rVert_2^2 \nonumber \\
    & \leq \inf_{\substack{\theta \in \RR^n:\\ \mcV_-(\theta) = \mathcal{V}}} \left\{
    \frac{1}{n} \lVert \theta - \theta^* \rVert_2^2
        + C\sigma^2 \frac{k(\theta)}{n} \log \frac{\ee n}{k(\theta)}
        + C\sigma^2 \frac{M(\theta)}{k(\theta)} \log \frac{\ee n}{k(\theta)}
    \right\} \nonumber \\
    & + \frac{4 \sigma^2 \log\eta^{-1}}{n}
\end{align}
with probability at least $1 - \eta$.
\end{thm}

The following risk bound for the best choice of the tuning parameter $\mcV \geq 0$ is an immediate consequence of Theorem \ref{thm:constrained_tangent}.

\begin{cor}\label{cor:constrained_tangent}
Suppose $\theta^* \in \RR^n$.
Choose $\mcV^* \geq 0$ that minimizes the upper bound in \eqref{eq:thm_constrained_tangent_1} (thus, $\mcV^*$ depends on the true parameter $\theta^*$).
Then, we have
\begin{align}\label{eq:cor_ct_1}
    & \frac{1}{n} \EE_{\theta^*} \lVert \hat{\theta}_{\mcV^*} - \theta^* \rVert_2^2 \nonumber \\
    & \leq
    \inf_{\theta \in \RR^n} \left\{
        \frac{1}{n} \lVert \theta - \theta^* \rVert_2^2
        + C\sigma^2 \frac{k(\theta)}{n} \log \frac{\ee n}{k(\theta)}
        + C\sigma^2 \frac{M(\theta)}{k(\theta)} \log \frac{\ee n}{k(\theta)}
    \right\}.
\end{align}
Also, choosing $\mcV := \mcV^*$ or $\mcV := \mcV_-(\theta^*)$, we have
\begin{equation}\label{eq:cor_ct_2}
    \frac{1}{n} \EE_{\theta^*} \lVert \hat{\theta}_{\mcV} - \theta^* \rVert_2^2
    \leq
    C\sigma^2 \left\{
        \frac{k(\theta^*)}{n} \log \frac{\ee n}{k(\theta^*)}
        + \frac{M(\theta^*)}{k(\theta^*)} \log \frac{\ee n}{k(\theta^*)}
    \right\}.
\end{equation}
\end{cor}

\begin{rmk}[]
We briefly comment on the proof of Theorem \ref{thm:constrained_tangent} and Corollary \ref{cor:constrained_tangent}.
A key ingredient is to obtain a bound on the \textit{statistical dimension} \citep{Amelunxen2014} of the tangent cone of the constraint set $\set{ \theta \in \RR^n: \mcV_-(\theta) \leq \mcV }$.
This methodology was first developed for the isotonic regression and the convex regression by \citet{Bellec2015b}.
In particular, our approach is inspired by the analysis of the constrained trend filtering estimators by \citet{Guntuboyina2017a}.
See Appendix \ref{sec:appendix_proof_adaptation} for detailed proofs.
\end{rmk}

By restricting the region over which the infimum in \eqref{eq:cor_ct_1} is taken, we have the oracle inequality for monotone vectors
\[
    \frac{1}{n} \EE_{\theta^*} \lVert \hat{\theta}_{\mcV^*} - \theta^* \rVert_2^2
    \leq
    \inf_{\theta \in K^\uparrow_n} \left\{
        \frac{1}{n} \lVert \theta - \theta^* \rVert_2^2
        + C\sigma^2 \frac{k(\theta)}{n} \log \frac{en}{k(\theta)}
    \right\},
\]
which recovers the existing results on the isotonic regression \citep{Chatterjee2015,Bellec2015b} up to a constant multiplicative factor.

To understand the general upper bound in \eqref{eq:cor_ct_1}, we have to control the quantity $M(\theta)$ defined in \eqref{eq:m_definition}.
To this end, we consider the \textit{minimal length condition};
we say that $\theta \in \RR^n$ satisfies the minimal length condition for a constant $c > 0$ if it satisfies
\begin{equation}\label{eq:min_length}
    \min \set{ |A_i|: 1 \leq i \leq k, w_i \neq w_{i+1} } \geq \frac{cn}{k},
\end{equation}
where the partition $\Pi_{\mathrm{const}} = \set{A_1, A_2, \ldots, A_k}$ and the signs $w_i$ ($i = 1, \ldots, k+1$) are defined as in \eqref{eq:m_definition}. 
Intuitively, a signal $\theta \in \RR^n$ is well approximated by another signal that satisfies the minimal length condition if $\theta$ has ``moderate slopes'' around the order-violating jumps.
For further discussion on such growth conditions, see Section \ref{sec:piecewise_application}.

Based on the minimal length condition, we have the following result from Theorem \ref{thm:constrained_tangent} .

\begin{cor}\label{cor:min_length}
Suppose that $\theta^* \in \RR^n$ satisfies the minimal length condition \eqref{eq:min_length} for a constant $c > 0$. Assume that $\theta^*$ is $k(\theta^*)$-piecewise constant and $m(\theta^*)$-piecewise monotone. Then, the constrained nearly-isotonic regression \eqref{eq:neariso_constrained} satisfies
\begin{align}\label{eq:cor_min_length}
    & \frac{1}{n} \EE_{\theta^*} \lVert \hat{\theta}_{\mathcal{V}} - \theta^* \rVert_2^2 \nonumber \\
    & \leq (\mcV_-(\theta^*) - \mcV)^2
    + C \sigma^2 \left(
        \frac{k(\theta^*)}{n} + \frac{2c^{-1} (m(\theta^*) - 1)}{n}
    \right) \log \frac{\ee n}{k(\theta^*)}.
\end{align}
In particular, if the tuning parameter $\mcV$ is chosen so that
\[
    (\mcV_-(\theta^*) - \mcV)^2
    \leq C^\prime \frac{k(\theta^*)}{n} \log \frac{\ee n}{k(\theta^*)}
\]
for a positive constant $C^\prime$, we have
\[
    \frac{1}{n} \EE_{\theta^*} \lVert \hat{\theta}_{\mathcal{V}} - \theta^* \rVert_2^2
    \leq C'' \sigma^2 \left(
        \frac{k(\theta^*)}{n} + \frac{2c^{-1} (m(\theta^*) - 1)}{n}
    \right) \log \frac{\ee n}{k(\theta^*)},
\]
where $C''$ is a positive constant.
\end{cor}

\begin{rmk}
If $\theta$ is $k$-piecewise constant and $m$-piecewise monotone, it is always true that $k \geq 2(m - 1)$. Hence, the inequality \eqref{eq:cor_min_length} can be simplified as
\[
    \frac{1}{n} \EE_{\theta^*} \lVert \hat{\theta}_{\mathcal{V}} - \theta^* \rVert_2^2
    \leq (\mcV_-(\theta^*) - \mcV)^2 + C(c) \sigma^2
    \frac{k(\theta^*)}{n}\log \frac{\ee n}{k(\theta^*)},
\]
where $C(c) > 0$ is a constant that depends on $c$ alone.
\end{rmk}

\begin{rmk}
We comment on the minimal length condition and the relation to estimation of piecewise constant vectors. We conjecture that the minimum length condition \eqref{eq:min_length} is essentially unavoidable for the risk bound of the nearly-isotonic regression due to the following analogy to the fused lasso. The minimal length condition for the fused lasso is considered by \citet{Guntuboyina2017a}.
%In the problem of estimating $k$-piecewise constant vectors, it is shown that the minimax rate is $\frac{k}{n} \log \frac{\ee n}{k}$ \citep[see, e.g.,][]{Gao2017}.
For the fused lasso, \citet{Fan2017} showed that the minimum length condition cannot be removed in the sense that there is a lower bound depending on the minimum length $\Delta = \min_{i} |A_i|$ (see also the experimental result by \citet{Guntuboyina2017a}, Remark 2.5).
%On the other hand, it is proved that there are other classes of estimators that do not suffer from the minimal length condition \citep{Gao2017,Fan2017}.
\end{rmk}

\subsection{Risk bounds for penalized estimators}\label{sec:penalized}

In this section, we consider the risk bounds for the nearly-isotonic regression \eqref{eq:neariso} in the original penalized form by \citet{Tibshirani2011}. 

\begin{thm}\label{thm:penalized}
For any $\lambda \geq 0$, let $\hat{\theta}_\lambda$ denote the nearly-isotonic regression estimator defined in \eqref{eq:neariso}.
Let $\theta^*$ and $\theta$ be any vectors in $\RR^n$.
Then, there exists a tuning parameter $\lambda^* = \lambda^*(\theta) \geq 0$ that depends only on $\theta$ such that, for any $\lambda \geq \lambda^*$, we have the following risk bound:
\begin{align}\label{eq:thm_penalized_1}
    \frac{1}{n} \EE_{\theta^*} \norm{\hat{\theta}_{\lambda} - \theta^*}_2^2
    & \leq
    \frac{1}{n} \norm{\theta - \theta^*}_2^2
    + C\sigma^2 \frac{k(\theta)}{n} \log \frac{\ee n}{k(\theta)}
    + C\sigma^2 \frac{M(\theta)}{k(\theta)} \log \frac{\ee n}{k(\theta)} \nonumber \\
    & + 3 (\lambda - \lambda^*)^2 M(\theta),
\end{align}
where $M(\theta)$ and $k(\theta)$ are defined similarly as in Theorem \ref{thm:constrained_tangent}.
Furthermore, for any $\eta \in (0, 1)$, the inequality
\begin{align}\label{eq:thm_penalized_2}
    \frac{1}{n} \norm{\hat{\theta}_{\lambda} - \theta^*}_2^2
    & \leq
    \frac{1}{n} \norm{\theta - \theta^*}_2^2
    + 2C\sigma^2 \frac{k(\theta)}{n} \log \frac{\ee n}{k(\theta)}
    + 2C\sigma^2 \frac{M(\theta)}{k(\theta)} \log \frac{\ee n}{k(\theta)} \nonumber \\
    & + 6 (\lambda - \lambda^*)^2 M(\theta)
    + \frac{16 \sigma^2 \log\eta^{-1}}{n}
\end{align}
holds with probability $1 - \eta$.
\end{thm}

We comment on some direct consequences of Theorem \ref{thm:penalized}.
In this theorem, $\lambda^*(\theta)$ is defined as a function of $\theta$.
To understand the risk bound \eqref{eq:thm_penalized_1}, we consider the choice of the tuning parameter $\lambda \geq 0$ that depends on the true parameter $\theta^*$.
Let $\bar{\theta}$ be a vector that minimizes the quantity
\[
    \frac{1}{n} \norm{\theta - \theta^*}_2^2
    + C\sigma^2 \frac{k(\theta)}{n} \log \frac{\ee n}{k(\theta)}
    + C\sigma^2 \frac{M(\theta)}{k(\theta)} \log \frac{\ee n}{k(\theta)}
\]
among all $\theta \in \RR^n$.
Then, taking $\lambda^{**} := \lambda^*(\bar{\theta})$, we have the following oracle inequality which has the same form as \eqref{eq:cor_ct_1}:
\begin{align*}
    & \frac{1}{n} \EE_{\theta^*} \norm{\hat{\theta}_{\lambda^{**}} - \theta^*}_2^2 \\
    & \leq \inf_{\theta \in \RR^n} \left\{
        \frac{1}{n} \norm{\theta - \theta^*}_2^2
        + C\sigma^2 \frac{k(\theta)}{n} \log \frac{\ee n}{k(\theta)}
        + C\sigma^2 \frac{M(\theta)}{k(\theta)} \log \frac{\ee n}{k(\theta)}
    \right\}.
\end{align*}
Moreover, if $\lambda := \lambda^{**}$ or $\lambda := \lambda^*(\theta^*)$, we have
\[
    \frac{1}{n} \EE_{\theta^*} \norm{\hat{\theta}_{\lambda} - \theta^*}_2^2
    \leq C\sigma^2 \left\{
        \frac{k(\theta^*)}{n} \log \frac{\ee n}{k(\theta^*)}
        + \frac{M(\theta^*)}{k(\theta^*)} \log \frac{\ee n}{k(\theta^*)}
    \right\}.
\]
Again, if we assume the minimal length condition \eqref{eq:min_length} on $\theta^*$, we obtain a simplified bound of the form \eqref{eq:cor_ct_2}.

We move on to discuss a precise expression of $\lambda^*(\theta)$ in Theorem \ref{thm:penalized}.
The next proposition provides an upper bound for $\lambda^*(\theta)$.

\begin{prop}\label{prop:minimal_lambda}
Suppose $\theta \in \RR^n$.
Let $\Pi_\mathrm{const}(\theta) := \set{A_1, A_2, \ldots, A_k}$ be the constant partition of $\theta$, and $w_1, w_2, \ldots, w_{k+1}$ be the associated signs defined in \eqref{eq:sign_definition}.
Then, there is a universal constant $C > 0$ such that $\lambda^*(\theta)$ in Theorem \ref{thm:penalized} is bounded from above by
\[
    C \sigma \min \left\{
        \frac{\norm{\theta}_2}{\mcV_-(\theta)}, \
        \left(
            \sum_{i = 1}^k \frac{1_{\set{w_i \neq w_{i + 1}}}}{|A_i|}
        \right)^{-1/2}
    \right\}
    \sqrt{
        \left(
            k(\theta) + \frac{n M(\theta)}{k(\theta)}
        \right) \log \frac{\ee n}{k(\theta)}
    }.
\]
\end{prop}

% remark on the tuning parameters
% reference for the fused lasso
The purpose of the choice of $\lambda^*$ in Proposition \ref{prop:minimal_lambda} is to derive the theoretical convergence rate in terms of $k(\theta)$ and $M(\theta)$.
However, different choices are possible if we are interested in other theoretical aspects (e.g., estimation consistency for changepoints).
For the fused lasso estimator \eqref{eq:total_variation}, several authors have studied theoretical choices of tuning parameters that result in risk upper bounds \citep{Dalalyan17, Lin2017, Guntuboyina2017a}.

\begin{rmk}[Example of parameter choice]
Here, we provide an example choice of the tuning parameter $\lambda$ under a simple length condition. Let us assume that (i) $\theta^*$ is not globally monotone (i.e., $M(\theta^*) > 0)$) and (ii) $|A_i|$ is of order $n/k$, that is, 
\[
    c_1 \frac{n}{k} \leq |A_i| \leq c_2 \frac{n}{k},
    \quad i = 1, \ldots, k
\]
holds for some $0 < c_1 < c_2$. Then, we can see that $\lambda^*(\theta^*)$ is bounded from above by
\[
    C' \sigma \sqrt{n \log \ee n},
\]
where $C'$ is a constant that depends on $C, c_1, c_2$. For the fused lasso, the theoretical choice $\lambda = O(\sigma \sqrt{n \log \ee n})$ has been suggested by \citet{Dalalyan17} and \citet{Guntuboyina2017a}. For a detailed discussion, see Remark 2.7 by \citet{Guntuboyina2017a} and references therein.
\end{rmk}

\begin{rmk}
In general, the choice of the tuning parameter that minimizes the risk can be different from the theoretical suggestion.
More importantly, we cannot obtain the value of $\lambda$ suggested in Proposition \ref{prop:minimal_lambda} because it depends on the unknown true parameter $\theta^*$ and the noise standard deviation $\sigma$.
In practice, there are two typical data-dependent choices of $\lambda$:
\begin{itemize}
    \item \textbf{Stein's unbiased risk estimate:} If we know $\sigma$ or its estimate value $\hat{\sigma}$, we can reasonably choose a parameter $\lambda$ by minimizing Stein's unbiased risk estimate (SURE)
    \begin{equation}\label{eq:sure}
        \mathrm{SURE}(\lambda) = \frac{1}{n} \norm{ y - \hat{\theta}_\lambda}_2^2
        + \frac{2 \hat{\sigma}^2}{n} \hat{\mathrm{df}}(\hat{\theta}_\lambda) + (\mathrm{constant}).
    \end{equation}
    Here, $\hat{\mathrm{df}}(\hat{\theta}_\lambda) := k(\hat{\theta}_\lambda)$ is an unbiased estimate of the \textit{degrees of freedom}.
    See \citet{Tibshirani2011} for the derivation.
    \item \textbf{Cross-validation:}
    We can also apply the cross-validation when the model \eqref{eq:gaussian_sequence} is interpreted as a discrete observation of a continuous signal.
    Specifically, suppose that the data is generated according to the following nonparametric regression model:
    \begin{equation}\label{eq:regression_setting}
        y_i = f^*(x_i) + \xi_i, \quad i = 1, \ldots, n,
    \end{equation}
    where $x_1 < x_2 < \ldots < x_n$ are given design points in $[0, 1]$ and $f^*: [0, 1] \to \RR$ is an unknown piecewise monotone function.
    We define the nearly-isotonic regression estimator $\hat{f}_\lambda$ over the interval $[0, 1]$ as follows:
    First, we determine the values $\hat{\theta}_{\lambda, i}$ ($i = 1, 2, \ldots, n$) by solving
    \begin{equation}\label{eq:neariso_design_points}
        \hat{\theta}_\lambda \in \argmin_{\theta \in \RR^n} \left\{
            \frac{1}{2} \norm{y - \theta}_2^2
            + \lambda \sum_{i=1}^{n - 1} \frac{(\theta_i - \theta_{i+1})_+}{x_{i+1} - x_i}
        \right\}.
    \end{equation}
    Then, we define $\hat{f}_\lambda: [0, 1] \to \RR$ by interpolation.
    For instance, one can output a piecewise constant function so that $\hat{f}_\lambda(x_i) = \hat{\theta}_{\lambda, i}$.
    In this sense, given a new design point $x^{\mathrm{new}}$, we can predict the value of $f^*(x^\mathrm{new})$ by $\hat{f}_\lambda(x^{\mathrm{new}})$.
    Hence, we can naturally apply the cross-validation in this situation.
\end{itemize}
\end{rmk}

\subsection{Application to piecewise monotone vectors}\label{sec:piecewise_application}

To gain a deeper understanding of the adaptation property of the nearly-isotonic regression, we study the risk bound under a more specific assumption.
We define the following \textit{moderate growth condition} for piecewise monotone vectors.

\begin{defi}\label{defi:moderate_growth}
Let $n \geq 2$. We say that a monotone vector $\theta \in K^\uparrow_{n}$ satisfies the moderate growth condition if
\[
    \theta_i \leq \theta_1 + \frac{i - 1}{n - 1}\mcV(\theta)
    \quad \text{for $i = 1, 2, \ldots, \lceil n / 2 \rceil$}
\]
and
\[
    \theta_i \geq \theta_1 + \frac{i - 1}{n - 1}\mcV(\theta)
    \quad \text{for $i = \lceil n / 2 \rceil, \lceil n / 2 \rceil + 1, \ldots, n$}.
\]
\end{defi}

Figure \ref{fig:moderate_growth} gives an illustration of the moderate growth condition.
In words, the signal $\theta \in \RR^n$ satisfying the moderate growth condition is not larger than the linear signal in the left half of the domain, and not less than that in the right half of the domain.
Intuitively, the role of the moderate growth condition is to guarantee the minimal length condition \eqref{eq:min_length} for a piecewise constant approximation.

\begin{figure}[tb]
    \centering
    \includegraphics[width=\linewidth]{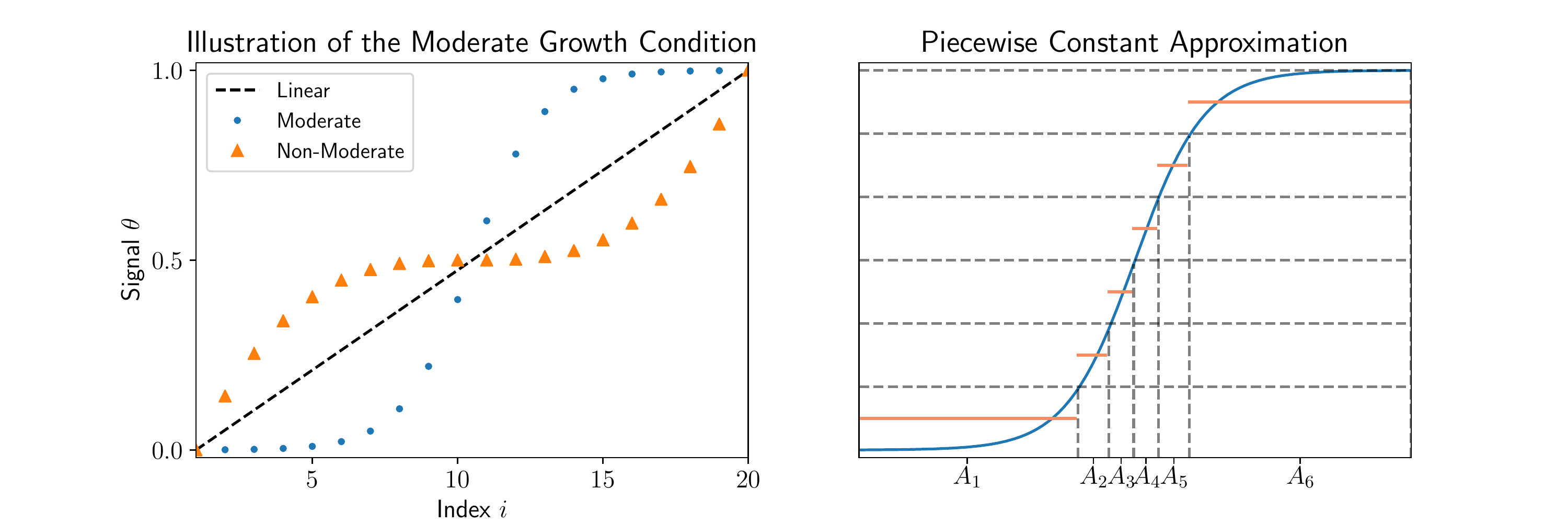}
    \caption{\textbf{Illustration of the moderate growth condition.}
    \textbf{Left:} The plotted three signals are monotone vectors in $K_n^\uparrow$ with $n = 20$ and $\mcV(\theta) = 1$.
    The dotted line represents the linear signal $\theta^{\mathrm{linear}}_i = i/n$ ($i = 1, 2, \ldots, n$).
    The blue circles depict an example of a signal that satisfies the moderate growth condition.
    That is, it is not larger than the linear signal $\theta^{\mathrm{linear}}_i$ for $1 \leq i \leq 10$, and not less than $\theta^{\mathrm{linear}}_i$ for $10 \leq i \leq 20$.
    On the other hand, the orange triangles depict a counterexample for this condition.
    \textbf{Right:} If $\theta$ satisfies the moderate growth condition, there is a $k$-piecewise monotone vector such that the lengths of segments at both ends are not less than $k / n$. See Appendix \ref{sec:cor_moderate_proof} for a detailed explanation.}
    \label{fig:moderate_growth}
\end{figure}

Suppose that the true signal $\theta^*$ is piecewise monotone and every segment satisfies the moderate growth condition.
Then, the nearly-isotonic regression achieves a nearly minimax convergence rate as follows.

\begin{cor}\label{cor:moderate}
Suppose that the following assumptions hold:
\begin{enumerate}[label=(\alph*)]
    \item The partition is equi-spaced: $|A_1| = |A_2| = \cdots = |A_m| \; (= \frac{n}{m})$.
    \item $\theta_{A_j}^*$ is monotone and $\mcV(\theta_{A_j}^*) \leq \mcV / m$ for each $j = 1, \ldots, m$.
    \item $\theta^*_{A_j}$ satisfies the moderate growth condition for each $j = 1, 2, \ldots, m$.
\end{enumerate}
Then, the estimator \eqref{eq:neariso} with optimally tuned parameter $\lambda$ satisfies the following risk bound:
\begin{equation}\label{eq:moderate_risk_bound}
    \frac{1}{n} \EE_{\theta^*} \norm{\hat{\theta}_\lambda - \theta^*}_2^2
    \leq C \max \left\{
        \left( \frac{\sigma^2 \mcV \log \frac{\ee n}{m}}{n} \right)^{2/3}, \
        \frac{\sigma^2 m}{n} \log \frac{\ee n}{m}
    \right\}.
\end{equation}
\end{cor}

% comparison to the minimax rate
%Recall the minimax rate over $\Theta_n(m, \mcV)$ in Proposition \ref{prop:general_lower_bound}.
The risk bound \eqref{eq:moderate_risk_bound} achieves the minimax rate over $\Theta_n(m, \mcV)$ in Proposition \ref{prop:general_lower_bound} up to a multiplicative factor of $\log^{2/3} \frac{\ee n}{m}$.
We should note that the restrictive assumption (a) in Corollary \ref{cor:moderate} is employed merely for the sake of simplicity of the proof.
We may relax this assumption as
\[
    \min_{1 \leq i \leq m} |A_i| \geq \frac{c' n}{m}
\]
for some $c' > 0$.

\section{Model selection based estimators}\label{sec:model_selection}

Here, we consider estimators obtained by model selection among all partitions $\Pi$.
The main purpose of this section is to discuss whether the minimax lower bound in Proposition \ref{prop:general_lower_bound} can be achieved without any additional assumption such as the moderate growth condition.

Given a connected partition $\Pi = (A_1, A_2, \ldots, A_m)$ of $[n]$, we write $K_\Pi^\uparrow$ for the set of piecewise monotone vectors on $\Pi$, i.e.,
\[
    K_\Pi^\uparrow := K_{|A_1|}^\uparrow \times K_{|A_2|}^\uparrow \times \cdots \times K_{|A_m|}^\uparrow.
\]
Let $\hat{\theta}_\Pi$ denote the projection estimator onto $K_\Pi^\uparrow$.
By definition, $\hat{\theta}_\Pi$ is obtained by concatenating isotonic regression estimators defined in every segment.

If we know the true partition $\Pi^*$ on which $\theta^*$ is piecewise monotone, then the risk of the projection estimator $\hat{\theta}_{\Pi^*}$ is bounded from above by
\[
    C \sum_{i=1}^m \frac{|A_i|}{n} \left(
        \frac{\sigma^2 \mcV^{A_i}(\theta^*_{A_i})}{|A_i|}
    \right)^{2/3}.
\]
If the true partition is unknown, a natural idea is to select a data-dependent partition $\hat{\Pi}$ by a penalized selection rule:
\begin{equation}\label{eq:model_selection_estimator}
    \hat{\Pi} \in \argmin_{\Pi} \left\{
        \norm{y - \hat{\theta}_\Pi}_2^2
        + \mathrm{pen}(\Pi)
    \right\}.
\end{equation}
Here, $\mathrm{pen}(\Pi)$ is a positive penalty for the partition $\Pi$.

The penalized selection rules have been well studied in statistics.
In particular, \citet{Birge2001} and \citet{Massart2007} developed non-asymptotic risk bounds for generic model selection settings in Gaussian sequence models.
Hereafter, we construct a penalized selection estimator in the spirit of Theorem 4.18 in \citet{Massart2007}.

Instead of selecting $\hat{\theta}_\Pi$ according to \eqref{eq:model_selection_estimator}, we introduce the \textit{total variation sieves}.
Namely, in addition to selecting partitions, we also select budgets of piecewise total variations as follows.
Let $\Pi = (A_1, A_2, \ldots, A_m)$ be a connected partition.
For any vector $\mbV = (\mcV_1, \mcV_2, \ldots, \mcV_m)$ with $\mcV_i \geq 0$ ($i = 1, 2, \ldots m$), we define the set of piecewise monotone vectors with bounded total variations as
\[
    K_\Pi^\uparrow(\mbV) = K_\Pi^\uparrow(\mcV_1, \mcV_2, \ldots, \mcV_m)
    := \set{\theta \in K_\Pi^\uparrow: \mcV^{A_i}(\theta_{A_i}) \leq \mcV_i
    \ \text{for $i = 1, 2, \ldots, m$}}.
\]
Then, we define $\hat{\theta}_{\Pi, \mbV}$ as the projection estimator onto $K_\Pi^\uparrow(\mbV)$.
Next, we define a countable set of vectors $\mbV$ as
\[
    \mathscr{V}(m) := \left\{
        (v(j_1), v(j_2), \ldots, v(j_m))
        : \ (j_1, j_2, \ldots, j_m) \in \NN^m
    \right\},
\]
where $v(j) := j^{3/2}$.
Finally, we select a pair $(\hat{\Pi}, \hat{\mbV})$ as the solution of the following minimization problem:
\begin{equation}\label{eq:model_selection_estimator_sieve}
    \min_{\Pi} \min_{\mbV \in \mathscr{V}(|\Pi|)} \left\{
        \norm{y - \hat{\theta}_{\Pi, \mbV}}_2^2
        + \mathrm{pen}(\Pi, \mbV)
    \right\}.
\end{equation}
%Here, the penalty term $\mathrm{pen}(\Pi, \mbV)$ will be chosen later.
With a careful choice of the penalty term $\mathrm{pen}(\Pi, \mbV)$, we have the following result:

\begin{thm}\label{thm:model_selection}
There exists an absolute constant $C_\mathrm{pen} > 0$ such that the following statement holds.
For any pair $(\Pi, \mbV)$, define the penalty $\mathrm{pen}(\Pi, \mbV)$ so that
\[
    \mathrm{pen}(\Pi, \mbV)
    = C_\mathrm{pen} \left(
        \sum_{i=1}^m \sigma^{4/3} |A_i|^{1/3} \mcV_i^{2/3}
        + \sigma^2 m \log \frac{\ee n}{m}
    \right).
\]
Let $(\hat{\Pi}, \hat{\mbV})$ be the minimizer in \eqref{eq:model_selection_estimator_sieve}.
\begin{align*}
    & \frac{1}{n} \EE_{\theta^*} \norm{\hat{\theta}_{\hat{\Pi}, \hat{\mbV}} - \theta^*}_2^2 \\
    & \leq \min_{\Pi} \min_{\mbV \in \mathscr{V}(|\Pi|)} \left\{
        \frac{3}{n} \dist^2(\theta^*, K_{\Pi}^\uparrow(\mbV))
        + \frac{2}{n} \mathrm{pen}(\Pi, \mbV)
    \right\} + \frac{256 \sigma^2}{n}.
\end{align*}
In particular, if $\theta^*$ is piecewise monotone on $\Pi = (A_1, A_2, \ldots, A_m)$, we have
\begin{align}\label{eq:thm_model_selection_1}
    & \frac{1}{n} \EE_{\theta^*} \norm{\hat{\theta}_{\hat{\Pi}, \hat{\mbV}} - \theta^*}_2^2 \nonumber \\
    & \leq 2 C_\mathrm{pen} \left\{
        \sum_{i=1}^m \frac{|A_i|}{n} \left(
        \frac{\sigma^2 (\mcV^{A_i}(\theta^*_{A_i}) + 1)}{|A_i|}
    \right)^{2/3}
    + \frac{\sigma^2 m}{n} \log \frac{\ee n}{m}
    \right\} + \frac{256 \sigma^2}{n}.
\end{align}
\end{thm}

We emphasize that Theorem \ref{thm:model_selection} does not require any additional assumptions on $\theta^*$, e.g., the minimum length condition or the moderate growth condition introduced in the previous section.
Therefore, it suggests the existence of a penalized model selection estimator that achieves the minimax rate in Proposition \ref{prop:general_lower_bound}.
However, the estimator \eqref{eq:model_selection_estimator_sieve} is not practical for a computational reason because it is obtained through the minimization over exponentially many possible partitions $\Pi$.
%One reason is that the constant $C_\mathrm{pen}$ in the definition of the penalty term is too large for a practical purpose.
%Another reason is the computational issue.

The dependence on the total variation of each segment in \eqref{eq:thm_model_selection_1} is $(\mcV^{A_i}(\theta^*_{A_i}) + 1)^{2/3}$ instead of $(\mcV^{A_i}(\theta^*_{A_i}))^{2/3}$.
The additional constant $1$ is due to the minimal resolution of the sieve.
To establish a non-asymptotic risk bound for the penalized model selection estimator without sieves (i.e., \eqref{eq:model_selection_estimator}) and remove the dependence on the sieve resolution remains an open problem.

\section{Simulations}\label{sec:simulations}

We provide some numerical examples for piecewise monotone regression problems.

\subsection{Dealing with inconsistency at boundaries}

Before presenting the simulation results, we here explain a well-known practical issue in the isotonic regression literature and a regularization method to cope with it.

In the study of statistical estimation under monotonicity constraints, it is known that the least squares estimator $\hat{\theta}_{K_n^\uparrow}$ is inconsistent at the boundary points (see e.g., \citet{Groeneboom} and \citet{Woodroofe1993}).
A similar issue arises for the nearly-isotonic regression estimators.
Since the penalty term in \eqref{eq:neariso} does not activate if the orders are not violated at the boundary points (i.e., $y_1 < y_2$ or $y_{n-1} < y_n$), the nearly-isotonic regression is not robust against a negative noise at the left boundary or a positive noise at the right boundary.
To overcome this issue, we consider the following boundary correction regularization for the nearly-isotonic regression:
\begin{equation}\label{eq:neariso_bc}
    \hat{\theta}_{\mathrm{boundary}, \lambda, \mu}
    = \argmin_{\theta \in \RR^n} \left\{
        \frac{1}{2} \norm{y - \theta}_2^2
        + \lambda \sum_{i=1}^n (\theta_i - \theta_{i+1})_+
        + \mu (\theta_n - \theta_1)
    \right\},
\end{equation}
where $\mu > 0$ is an additional tuning parameter.
It can easily be checked that the solution is equivalent to that of the ordinary nearly-isotonic regression \eqref{eq:neariso} applied to $\tilde{y} = (y_1 + \mu, y_2 \ldots, y_{n-1}, y_n - \mu)$.
Similar regularization methods for isotonic regression have been studied by \citet{Chen15, Wu2015} and \citet{Luss2017}.

\subsection{Simulation data}

Here, we evaluate the performance of the nearly-isotonic regression and related estimators on simulated data.
According to the one-dimensional regression model \eqref{eq:regression_setting}, we generated data with equi-spaced design points $x_i = (i - 1)/n$ ($i = 1, 2, \ldots, n$).
For the true function $f^*$, we consider $m$-piecewise monotone functions defined as
\[
    f^{(m)}(x) := \sum_{j = 1}^{m} f(m x - (j - 1)) 1_{I_j}(x)
\]
where $f: [0, 1) \to \RR$ is a given monotone function and $I_j := [(j-1)/m, j/m)$ for $j = 1, 2, \ldots, m$.
Following \citet{Meyer00}, we choose $f$ from the following two monotone functions:
\begin{align*}
    f_{\mathrm{sigmoid}}(x) & = \ee^{16 x - 8} / (1 + \ee^{16 x - 8}), \\
    f_{\mathrm{cubic}}(x) & = (2x - 1)^3 + 1.
\end{align*}
Figure \ref{fig:example_1dim} shows an example of $f = f_{\mathrm{sigmoid}}$ and $m = 2$.
It is worth noting that the former sigmoidal function $f_\mathrm{sigmoid}$ satisfies the moderate growth condition (see Definition \ref{defi:moderate_growth}), whereas the latter cubic function $f_\mathrm{cube}$ does not.
Hence, for the case of piecewise sigmoidal functions $f_\mathrm{sigmoid}^{(m)}$, the minimax rate of $\bigO(n^{-2/3})$ is achieved by both the nearly-isotonic regression and the fused lasso (see Corollary \ref{cor:moderate} above and Corollary 2.8 by \citet{Guntuboyina2017a}).

In our experiments, the size $n$ of the signal is chosen from $\set{2^6, 2^7, \ldots, 2^{10}}$.
The noise standard deviation $\sigma$ is assumed to be known and fixed to $0.25$.
We evaluated the MSE for the following four estimators:
\begin{itemize}
    \item \texttt{Neariso}: The nearly-isotonic regression \eqref{eq:neariso}.
    \item \texttt{NearisoBC}: The nearly-isotonic regression with boundary correction \eqref{eq:neariso_bc}
    \item \texttt{Fused}: The fused lasso \eqref{eq:total_variation}.
    \item \texttt{PO}: The projection estimator with the partition oracle, i.e., the projection estimator onto $K_\Pi^\uparrow$ provided with the true partition $\Pi$.
\end{itemize}
For \texttt{Neariso} and \texttt{Fused}, the tuning parameter $\lambda$ is selected by generalized $C_p$ criteria (i.e., minimizing SURE \eqref{eq:sure}).
For \texttt{NearisoBC}, the tuning parameters $(\lambda, \mu)$ are selected by a similar criterion.
To estimate the MSE, we generated 500 replications of the data and calculated the average value of the squared loss $\frac{1}{n} \norm{\hat{\theta} - \theta^*}_2^2$.

Figure \ref{fig:mse_1d} presents the results for $m = 2, 4$ and $f = f_\mathrm{sigmoid}, f_\mathrm{cubic}$.
The upper line shows log-log plots of the MSE versus $n$.
In each setting, the three regularization based estimators (i.e., \texttt{Neariso} \texttt{NearisoBC} and \texttt{Fused}) performed as well as the ideal estimator \texttt{PO}, whereas the former three estimators do not use the information about the true partition.
The risks of \texttt{PO} are well fitted by lines of slopes of $-2/3$, which means that the speed of the convergence is about the minimax optimal rate of $\bigO(n^{-2/3})$.

Next, we provide more detailed comparisons of regularization based estimators.
The lower line in Figure \ref{fig:mse_1d} shows the difference of MSEs from that of \texttt{PO}.
For piecewise sigmoidal functions, \texttt{NearisoBC} and \texttt{Fused} performed better than \texttt{Neariso}.
Notably, in the case of $m = 2$, the risks of \texttt{Fused} were even better than \texttt{PO} for large values of $n$.
A possible reason for the better performance of the fused lasso is that the sigmoidal function can be well approximated by a piecewise constant function near the boundaries.
On the other hand, for piecewise cubic functions, \texttt{Neariso} performed slightly better than the other two estimators for small values of $n$.

\begin{figure}[tb]

\begin{minipage}[c]{\linewidth}
\centering
\includegraphics[width=0.9\linewidth]{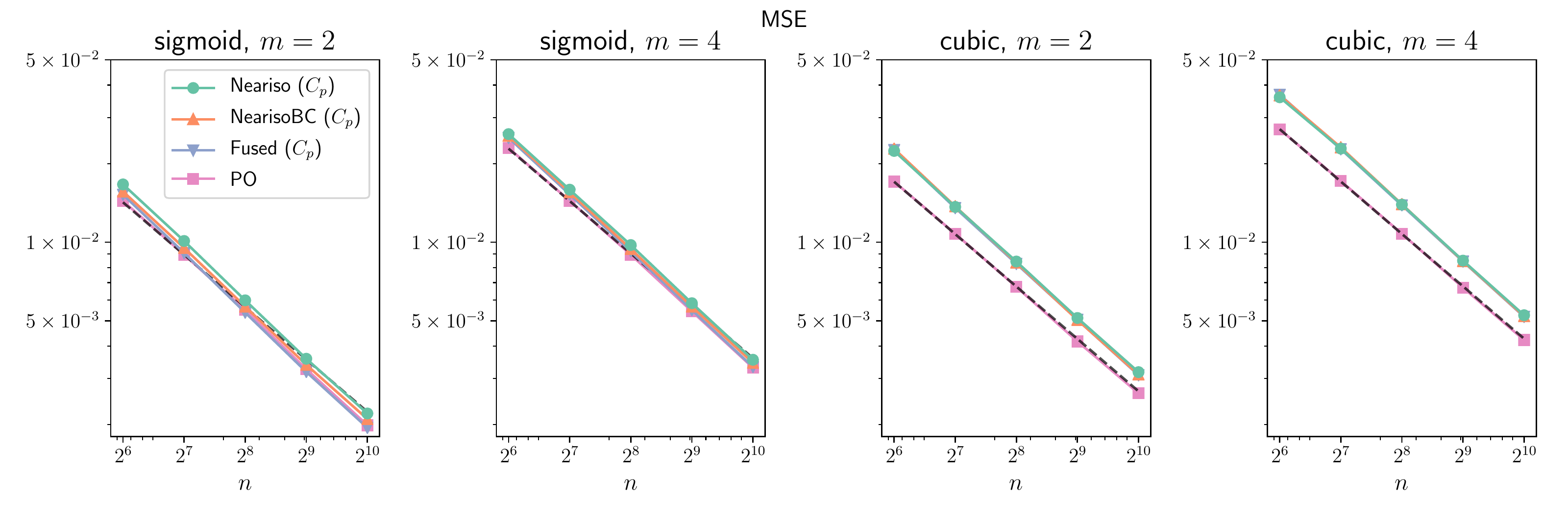}
\end{minipage}

\begin{minipage}[c]{\linewidth}
\centering
\includegraphics[width=0.9
\linewidth]{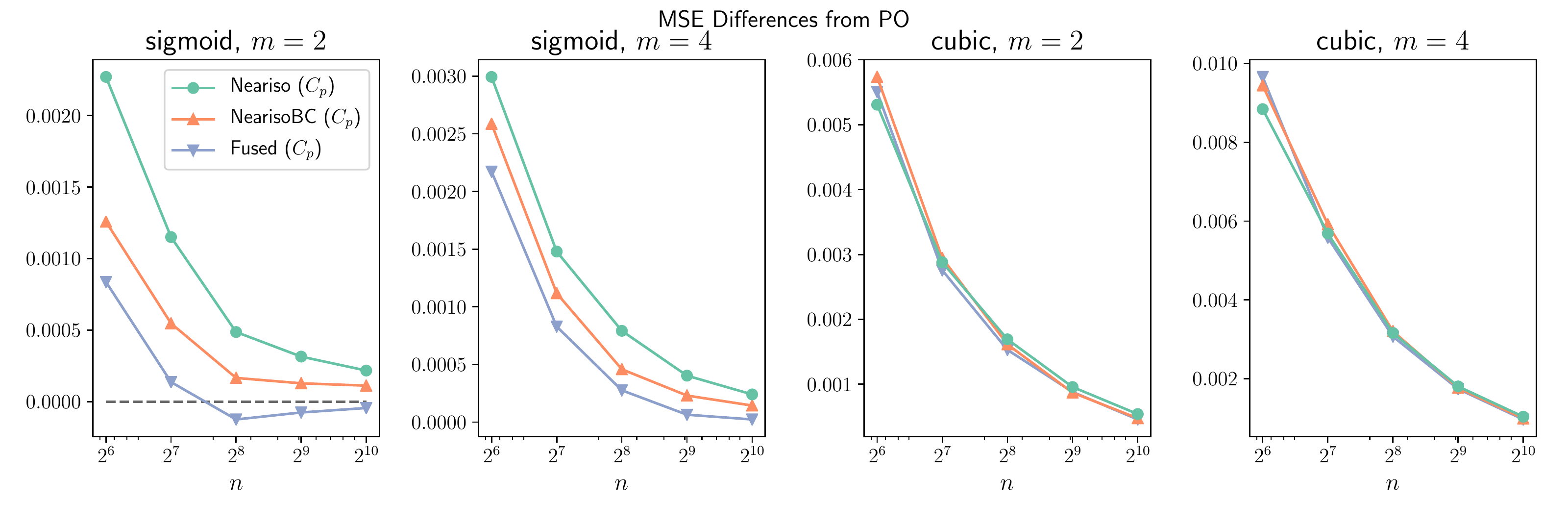}
\end{minipage}

\caption{\textbf{The risks of nearly-isotonic type estimators on simulated data}.
The upper line shows log-log plots of the MSEs versus $n$.
The lower line shows the difference of the MSEs between regularization type estimators (i.e., \texttt{Neariso} \texttt{NearisoBC} and \texttt{Fused}) and the projection estimator with the oracle partition choice (\texttt{PO}).}
\label{fig:mse_1d}
\end{figure}

\subsection{Geological data}\label{sec:geological}

% data background
We conducted experiments on GPS data related to a seismological phenomenon reported by \citet{Rogers2003}.
The aim here is to investigate the performance of the nearly-isotonic type estimators on real-world data in which piecewise monotone approximations have already been justified in the previous work.
For the signal $y$, we used the difference of the east-west components of GPS measurements between two observatories, which are located in Victoria (British Columbia, Canada) and Seattle (United States).
The GPS data is provided by \citet{Melbourne2018}.
The top panel in Figure \ref{fig:robust_neariso_on_gps} shows the plot.
The data period starts on January 1, 2010, and ends on December 2, 2017.
After removing missing records, the size of the signal is $n = 2885$.
The increasing trend of the signal is considered to be caused by the subduction process at the plate boundary.
We can also see periodic reversals in the signal, and the entire signal may be approximated by a piecewise monotone signal.
Such reversals may be related to the seismological phenomenon so-called the episodic tremor and slip.
According to \citet{Rogers2003}, such slip events were observed in every 13 to 16 months in their data taken from 1997 to 2003.

% on the outlier
GPS data contains several anomalous values.
For the signal $y$ considered above, most of the values $y_i$ are between 20 and 50, except for a single outlier $y_{2344} = 139.34$.
The behaviors of the estimators are extremely affected by the existence of such outliers.
In our situation, we can manually remove the anomalous value (denoted by $\tilde{y}$).
However, it is often difficult to distinguish outliers in practical situations.
From this perspective, we also considered the robust $M$-estimation version of the nearly-isotonic regression defined as \eqref{eq:general_loss_neariso} with $\mcL(\theta; y) = \sum_{i=1}^n \ell_\delta(\theta_i - y_i)$.
Here, $\ell_\delta$ is the Huber loss:
\[
    \ell_\delta(u) := \left\{
    \begin{aligned}
        & \frac{1}{2} u^2 & \quad (|u| \leq \delta) \\
        & \delta |u| - \frac{1}{2} \delta^2 & \quad (|u| > \delta)
    \end{aligned}
    \right.,
\]
which is commonly used in the robust regression literature.

% CV fit
We applied the nearly-isotonic regression \eqref{eq:neariso} and its robust variant to the signals $y$ and $\tilde{y}$ in the above.
The tuning parameters $\lambda$ were determined by the $5$-fold cross-validation, and $\delta$ in the Huber loss was fixed as $\delta = 0.01$.

First, we consider the case where the outlier is removed manually.
The second panel in Figure \ref{fig:robust_neariso_on_gps} shows the result for the cross-validated nearly-isotonic regression.
The vertical lines denote the locations of downward jumps in the estimators.
We can see that the period of jump clusters is about 12 to 14 months, which is close to that of the seismological slip events suggested by \citet{Rogers2003}.

Next, we consider the case where the signal contains an outlier.
In this case, the value of the squared loss largely depends on the error at the coordinate of the outlier.
Then, the cross-validation may choose a large tuning parameter, and the resulting estimator becomes close to a monotone signal.
The third panel in Figure \ref{fig:robust_neariso_on_gps} shows that the number of downward jumps is considerably less than the number that is expected from the known frequency of the slip events.
Conversely, the fourth panel in Figure \ref{fig:robust_neariso_on_gps} shows that the robust version of the nearly-isotonic regression outputs similar clusters of change points as in the second panel.

\begin{figure}[tb]
\includegraphics[width=0.95\linewidth]{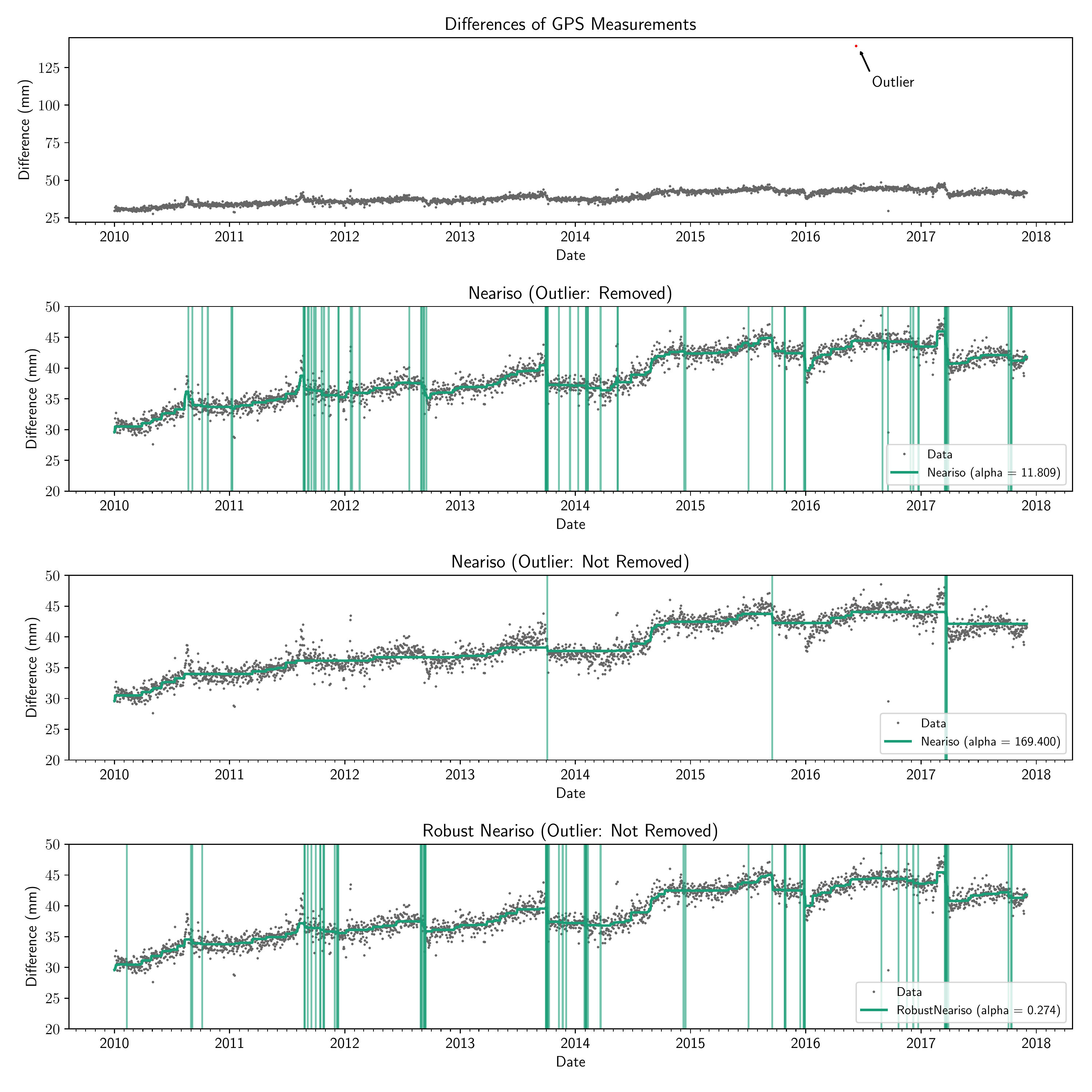}
\caption{\textbf{Nearly-isotonic type estimators applied to GPS data}.
See the text for details.
}
\label{fig:robust_neariso_on_gps}
\end{figure}

\section{Discussion}\label{sec:discussion}

% conclusion
In this paper, we studied the problem of estimating piecewise monotone signals.
The classical isotonic regression estimator cannot be applied in this setting because of the existence of arbitrarily large downward jumps.
We derived the minimax risk lower bound over piecewise monotone signals with bounded upper total variations.
The minimax rate is tight up to multiplicative constant because it can be achieved by a (computationally inefficient) model selection based estimator.
Our main results show that the nearly-isotonic regression estimator achieves this rate under an additional growth condition.
An advantage of the nearly-isotonic regression is that the estimator can be calculated efficiently on arbitrary directed graphs by parametric max-flow algorithms.
The simulation results demonstrate that the nearly-isotonic regression has an almost similar convergence rate as the ideal estimator that knows the true partition.

\subsection{Non-Gaussian noises}

In this paper, we provided risk bound for the nearly-isotonic regression under the assumption that the noise distribution is Gaussian. However, in practice, this assumption is too restrictive. We here briefly discuss the risk bound with non-Gaussian error distributions.

Suppose that $\xi_1, \ldots, \xi_n$ are i.i.d.~random variables with $\EE[\xi_1] = 0$ and $\mathrm{Var}(\xi_1) = \sigma^2$. Then, we can see that the ``expectation bound'' \eqref{eq:thm_penalized_1} holds with a different constant $C' > 0$. See Remark \ref{rmk:non_gauss} in the appendix for the key ingredients for the derivation. On the other hand, the ``high-probability bound'' \eqref{eq:thm_penalized_2} does not hold in general since it requires a more strong concentration property (i.e., the Gaussian concentration).
%It might be possible to generalize the such high-probability bound to more weaker concentration assumptions, but the presentation can become highly complex.

% discussion
\subsection{Future directions}
An interesting direction for future work is to investigate the optimal rate of piecewise monotone regression on higher dimensional grids or general graphs.
Recently, several researchers have analyzed the risk bounds for the isotonic regression estimators on two or more higher dimensional grid graphs \citep{Chatterjee2018, Han2017}.
It is natural to ask whether one can construct a computationally efficient estimator that is adaptive to piecewise monotone vectors on a given graph.
We believe that the nearly-isotonic type estimator \eqref{eq:neariso_general_graph} is a candidate.
A major difficulty is to determine an appropriate graph topology.
Given a partial order $\preceq$ on a set $V = [n]$, the corresponding isotonic regression estimator is uniquely determined.
However, there are many directed acyclic graphs that correspond to partial order $\preceq$.
Hence, the graph topology for the nearly-isotonic type estimators is not unique.
To control the connectivity, it may be useful to introduce edge weightings proposed by \citet{Fan2017}.

Another direction is to develop a model selection method for least squares estimators over unbounded cones.
We introduced sieves on the total variation in Section \ref{sec:model_selection} to construct an estimator that is adaptive to piecewise monotone vectors.
In practice, sieve-based methods can be computationally inefficient.
Conversely, if the true vector $\theta^*$ is monotone, the isotonic regression automatically achieves the minimax rate with respect to the total variation.
We conjecture that it is also possible to select the least squares estimator $\hat{\theta}_{\Pi}$ without using sieves.
In particular, we leave it as an open question whether the adaptive risk bound is achieved by the penalized selection rule of the form \eqref{eq:model_selection_estimator}.

\appendix

\section{Algorithms for nearly-isotonic estimators}\label{sec:algorithms}

In this section, we present algorithms for the nearly-isotonic regression and related estimators and discuss their computational complexities.
%Throughout, we assume $\bigO(1)$-time executions of any arithmetic operations and a random access to any element in a given vector.
Note that the main purpose of this section is to give a review of existing algorithms, and hence most results presented in this section are not new (except for Proposition \ref{prop:mod_pava_validity}).

\subsection{Penalized estimators}

Here, we introduce two algorithms to solve the penalized form nearly-isotonic regression \eqref{eq:neariso}.
In Section \ref{sec:alg_mod_pava}, we introduce the solution path algorithm developed by \citet{Tibshirani2011}.
The advantage of the solution path algorithm is that it outputs the solutions $\hat{\theta}_\lambda$ for every $\lambda \geq 0$ simultaneously.
However, the solution path algorithm cannot be applied to the estimators with general weights and graphs.
In Section \ref{sec:alg_divide_conquer}, we provide another algorithm that outputs the exact solution for a single $\lambda$.
The latter algorithm can be applied to the nearly-isotonic type estimators defined on any weighted directed graphs.

\subsubsection{One-dimensional problem}\label{sec:alg_mod_pava}

The modified pool adjacent violators algorithm (modified PAVA, \citet{Tibshirani2011}) is the algorithm used to calculate the solution path for the problem \eqref{eq:neariso}.
Here, we present a variant of the modified PAVA for the following weighted version of the estimator:
\begin{equation}
    \hat{\theta}_\lambda
    = \argmin_{\theta \in \RR^n} \left\{
        \frac{1}{2} \norm{y - \theta}_2^2
        + \lambda \sum_{i-1}^n c_i (\theta_i - \theta_{i+1})_+
    \right\},
    \label{eq:weighted_neariso}
\end{equation}
where $c_i > 0$ ($i = 1, 2, \ldots, n - 1$) are positive weight parameters.
Letting $c_i = (x_{i + 1} - x_i)^{-1}$, this formulation covers the nearly-isotonic regression for general increasing design points \eqref{eq:neariso_design_points}.

% modified pava (Tibshirani, et al. 2011)
\begin{algorithm}[ht]
\caption{Modified Pool Adjacent Violators Algorithm \citep{Tibshirani2011}}
\label{alg:modified_pava}
\DontPrintSemicolon
\KwIn{$y \in \RR^n$, $c_1, \ldots, c_{n - 1} > 0$}
\KwOut{Set of finitely many breakpoints $\Lambda = \set{\lambda_0, \lambda_1, \ldots, \lambda_N}$, solution path $\set{\hat{\theta}_\lambda}_{\lambda \in \Lambda}$}
\nl $\lambda_0 \leftarrow 0$, $\hat{\theta}_{\lambda_0} \leftarrow y$\\
\nl Let $\Pi_0$ be the constant partition of $\hat{\theta}_{\lambda_0}$.
Below, the solution $\hat{\theta}_{\lambda_i}$ is kept to be constant on $\Pi_i$.\\
\For{$i = 1, 2, \ldots$}{
\nl Let $k = |\Pi_{i - 1}|$.
Let $A_j = \set{\tau_j, \tau_j + 1, \ldots, \tau_{j+1} - 1}$ be the $j$-th element in the partition $\Pi_{i - 1}$, and $t_j$ be the value of $\hat{\theta}_{\lambda_{i-1}}$ on $A_j$ ($j = 1, 2, \ldots, k$).\\
\nl Set $s_0 = s_{k} = 0$ and $c_0 = 0$.
Compute $s_j = 1_{\set{t_j > t_{j+1}}}$ for $j = 1, 2, \ldots, k - 1$.\\
\nl Compute the slopes $m_j$ ($j = 1, 2, \ldots, k$) by
\[
    m_j = \frac{c_{\tau_j - 1} s_{j-1} - c_{\tau_{j + 1} - 1} s_j}{|A_j|}.
\]\\
\nl Compute $\delta$ by
\[
    \delta = \min_{1 \leq j \leq k - 1} \frac{t_{j+1} - t_j}{m_j - m_{j+1}}.
\]\\
\nl If $\delta \leq 0$, then terminate.\\
\nl $\lambda_{i} \leftarrow \lambda_{i-1} + \delta$.\\
\nl Set $\hat{\theta}_{\lambda_{i}}$ to be the piecewise constant vector whose values on $A_j$ are $t_j + m_j \delta$ ($j = 1, 2, \ldots, k$).\\
\nl Set $\Pi_i$ to be the constant partition of $\hat{\theta}_{\lambda_i}$.\\
}
\end{algorithm}

The derivation of Algorithm \ref{alg:modified_pava} is straightforward from the original paper of \citet{Tibshirani2011}.
We should note that the validity of this algorithm crucially depends on the property that the solution path is piecewise linear and ``agglomerative''.
It is well known that the piecewise linearity of the solution path holds for many classes of regularization estimators \citep{Rosset2007}.
We say that the solution path $\set{\hat{\theta}_\lambda}_{\lambda \geq 0}$ is \textit{agglomerative} if it satisfies the following condition:
if $\hat{\theta}_{\lambda, i} = \hat{\theta}_{\lambda, j}$ holds for some $\lambda = \lambda_0$, then the same equality holds for any $\lambda \geq \lambda_0$.
For the constant weights ($c_i \equiv 1$), such agglomerative property was proved by \citet{Tibshirani2011}.
However, for general non-unitary edge weights ($c_i \neq 1$), this need not be true. Here, we provide the following proposition to ensure the agglomerative property for non-unitary edge weights.
%Instead, we have a sufficient condition for the validity of Algorithm \ref{alg:modified_pava} as follows:

\begin{prop}\label{prop:mod_pava_validity}
The solution path of weighted nearly-isotonic regression \eqref{eq:weighted_neariso} is piecewise linear and agglomerative if the edge weights satisfy the following concavity condition.
\begin{equation}
    %\frac{c_{j+1}}{c_j} \leq \frac{j + 1}{j} \quad \text{for all $j = 1, 2, \ldots, n - 2$}.
    c_{j -1} + c_{j + 1} \leq 2 c_j
    \quad \text{for all $j = 0, 1, \ldots, n - 2$},
    \label{eq:mod_pava_validity_condition}
\end{equation}
where we defined $c_0 := 0$.
In particular, this condition implies that Algorithm \ref{alg:modified_pava} outputs the exact solution path.
\end{prop}

The condition \eqref{eq:mod_pava_validity_condition} demands that $c_j$ can be written as $c_j = f(j)$ for some concave function $f: \RR_{\geq 0} \to \RR_{\geq 0}$ with $f(0) = 0$ and $f(x) > 0$ for all $x > 0$. In particular, for any $i \leq j \leq k$, we have
\[
    c_j \geq \frac{(k - j) c_i + (j - i) c_k}{k - i}
\]
and
\[
    c_j \geq \frac{j}{k} c_k.
\]
\if0
For instance, we can apply Algorithm \ref{alg:modified_pava} to calculate the solution path of \eqref{eq:neariso_design_points} if the design points $x_1 < x_2 < \ldots < x_n$ that satisfies $x_2 - x_1 \leq 2 (x_3 - x_2)$ and
\[
    \frac{1}{x_{j + 1} - x_{j}} + \frac{1}{x_{j + 3} - x_{j + 2}}
    \leq \frac{2}{x_{j + 2} - x_{j + 1}}
\]
for all $j = 1, 2, \ldots, n - 3$.
\fi

\begin{proof}[Proof sketch of Proposition \ref{prop:mod_pava_validity}]
We can prove the validity of Algorithm \ref{alg:modified_pava} by a similar argument as \citet{Tibshirani2011} if we assume the piecewise linearity and the agglomerative property. The piecewise linearity is already shown in \citet{Rosset2007}. Hence, it remains to prove the agglomerative property under the condition \eqref{eq:mod_pava_validity_condition}. To this end, we leverage the ``agglomerative clustering condition'' defined in Appendix \ref{sec:weak_decomp}. In particular, we defer the details to Remark \ref{rmk:on_agglomerative} as well as Remark \ref{rmk:concavity_usage}.
%For a detailed discussion for this condition, see Remark \ref{rmk:on_agglomerative} in Appendix \ref{sec:weak_decomp}.
\end{proof}

\subsubsection{General graphs}\label{sec:alg_divide_conquer}
% divide-conquer

Let $G = (V, E)$ be a directed graph with $V := [n]$.
Suppose that each edge $(i, j) \in E$ is equipped with a positive weight $c_{(i, j)} > 0$.
We define the \textit{generalized nearly-isotonic regression} as
\begin{equation}\label{eq:neariso_general_graph}
    \hat{\theta}_{G, \lambda}
    = \argmin_{\theta \in \RR^n} \left \{
        \frac{1}{2} \norm{y - \theta}_2^2
        + \lambda \mcV_G(\theta)
    \right\}
\end{equation}
where $\mcV_G$ is a nearly-isotonic type penalty defined as
\begin{equation}\label{eq:neariso_penalty_general}
    \mcV_G (\theta) := \sum_{(i, j) \in E} c_{(i, j)} (\theta_i - \theta_j)_+.
\end{equation}
For any choices of $G$ and $c$, $\mcV_G$ becomes a convex function.
Clearly, the lower total variation $\mcV_-$ is a special case where $E = \set{(i, i + 1) : i = 1, 2, \ldots, n - 1}$ and $c_{(i, i + 1)} \equiv 1$.
Thus, \eqref{eq:neariso_general_graph} can be regarded as a generalization of the nearly-isotonic regression to general directed graphs.

The problem of the form \eqref{eq:neariso_general_graph} has been well studied in the optimization literature.
In particular, we can see that solving \eqref{eq:neariso_general_graph} is equivalent to solving a certain parametrized family of minimum-cut problems.
For detailed explanations of such an equivalence, see \citet{Obozinski16} and Chapter 8 in \citet{Bach13}.
Hence, \eqref{eq:neariso_general_graph} can be solved by the parametric max-flow algorithm \citep{Gallo1989} that runs in $\bigO(n |E| \log \frac{n^2}{|E|})$.
Conversely, it has been pointed out by \citet{Mairal2011} that, for many practical instances, some simplified variants of the parametric max-flow algorithm output the solution faster than the original algorithm by \citet{Gallo1989}.
We remark that \citet{Hochbaum2003} also developed the relationship between the isotonic regression and the parametric max-flow algorithm.

Algorithm \ref{alg:divide_conquer} shows the Divide-and-Conquer algorithm (Chapter 9 of \citet{Bach13}) that solves \eqref{eq:neariso_general_graph}.
In the inner loop, the algorithm recursively solves max-flow problems by defining smaller networks (Algorithm \ref{alg:divide_conquer_inner}).
See Figure \ref{fig:flow_networks} for examples of networks used in the first two recursions in the algorithm.

\begin{algorithm}
\caption{Divide-and-Conquer algorithm for the generalized nearly-isotonic regression \ref{eq:neariso_general_graph}}
\label{alg:divide_conquer}
\DontPrintSemicolon
\KwIn{$y \in \RR^V$, a directed graph $G = (V, E)$ with positive edge weights $\set{c_{(i, j)}}$, a tuning parameter $\lambda \geq 0$.}
\KwOut{The solution $\hat{\theta}_\lambda$ of \eqref{eq:neariso_general_graph}}
\nl Construct a flow network $\mcN$ by adding a source node $s$ and a sink node $t$ to the graph $G$.\\
\nl Compute $\hat{\theta}_\lambda = \mathrm{Prox}_{\lambda F_\mcN}(y)$ according to Algorithm \ref{alg:divide_conquer_inner}.
\end{algorithm}

\begin{algorithm}[ht]
\caption{$\mathrm{Prox}_{\lambda F_\mcN}(y)$}
\label{alg:divide_conquer_inner}
\DontPrintSemicolon
\KwIn{A flow network $\mcN = (V \cup \set{s} \cup \set{t}, E, c)$, $y\in \RR^V$ and $\lambda > 0$.}
\KwOut{Proximal operator $\mathrm{Prox}_{\lambda F_\mcN}(y)$.}
\nl Let $\alpha \leftarrow \frac{1}{|V|} (\sum_{i \in V} y_i - \lambda F_\mcN(V))$, where $F_\mcN(V)$ is the capacity of the edge $(s, t)$.\\
\nl \If{$|V| = 1$}{
\KwReturn{$\hat{\theta} = \alpha$}\\
}
\nl Find a subset $A \subseteq V$ minimizing the function
$A \mapsto \lambda F_\mcN(A) - \sum_{i \in A} y_i + \alpha |A|$.
Herein, $F_\mcN$ is the $s$-$t$ cut function of the network $\mcN$. 
This step is equivalent to solving the max-flow problem defined by the flow network in Figure \ref{fig:flow_networks}-(a).\\
\nl \If{$\lambda F_\mcN(A) - \sum_{i \in A} y_i + \alpha |A| = 0$}{
\KwReturn{$\hat{\theta} = \alpha 1_V$}.\\
}
\nl Let $\hat{\theta}_A \leftarrow \mathrm{Prox}_{\lambda F_{\mcN|A}}(y_A)$, where $\mcN|A$ is the reduction of $\mcN$ on $A$. The corresponding network is obtained by shrinking nodes $V \setminus A$ into the sink node $t$ (Figure \ref{fig:flow_networks}-(b)).\\
\nl Let $\hat{\theta}_{V \setminus A} \leftarrow \mathrm{Prox}_{\lambda F_{\mcN^A}}(y_{V \setminus A})$, where $\mcN^A$ is the contraction of $\mcN$ by $A$. The corresponding network is obtained by shrinking nodes $A$ into the source node $s$ and adding $- F_\mcN(A)$ to the capacity of $(s, t)$ (Figure \ref{fig:flow_networks}-(c)).\\
\end{algorithm}

\begin{figure}[tb]
\centering
\includegraphics[width=0.8\linewidth]{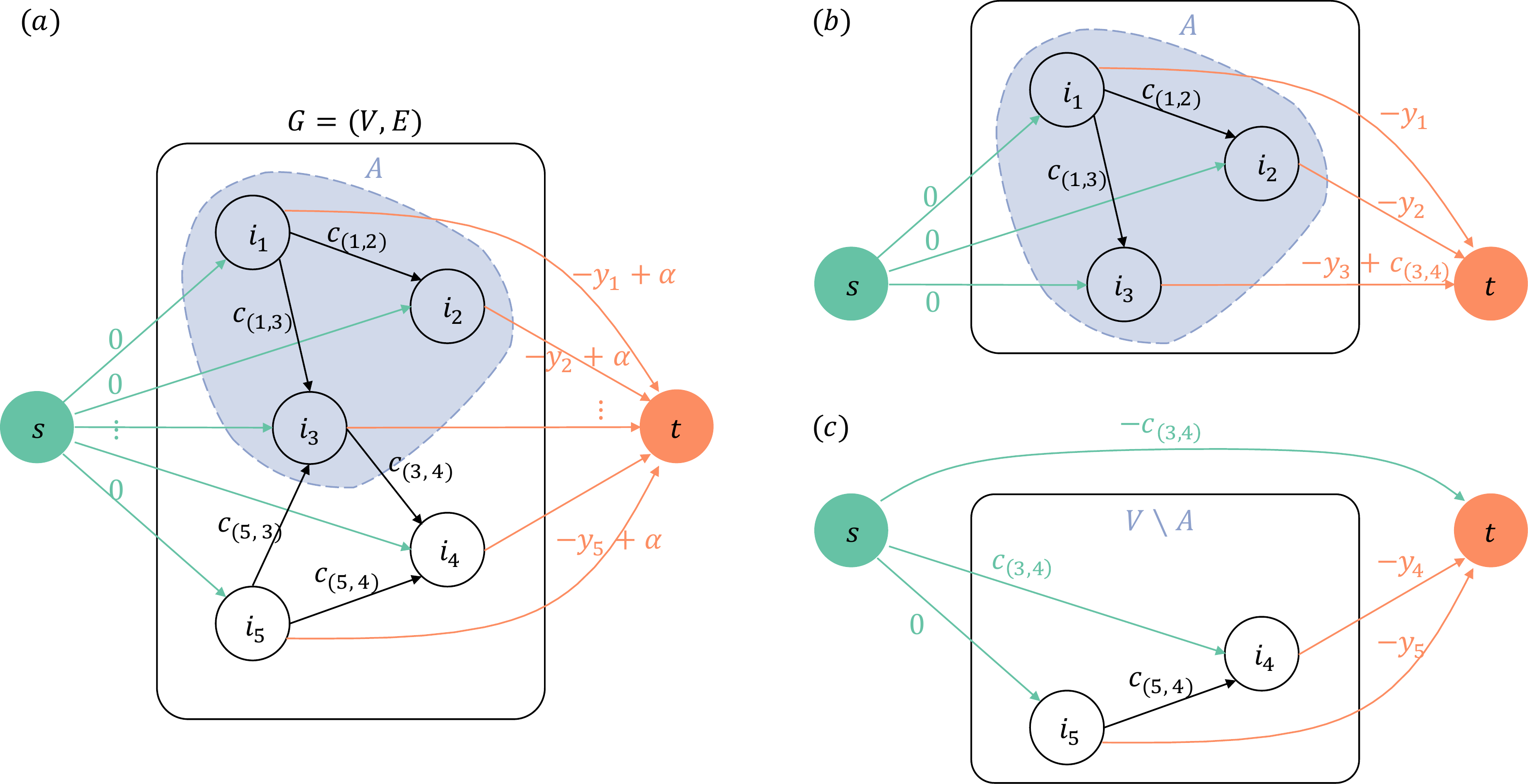}
\caption{\textbf{Flow networks in Algorithm \ref{alg:divide_conquer_inner}.}
Algorithm \ref{alg:divide_conquer_inner} requires to solve minimum $s$-$t$ cut problems (or equivalently maximum flow problems) defined on certain flow networks.
(a) A network that corresponds to the minimization problem in line 3.
(b) A network that corresponds to the function $B \mapsto \lambda F_{\mcN|A}(B) - y(B)$ in line 5.
(c) A network that corresponds to the function $B \mapsto \lambda F_{\mcN^A}(B) - y(B)$ in line 6.
Note that we assumed $\lambda = 1$ in this example.
}
\label{fig:flow_networks}
\end{figure}

\subsubsection{General convex loss functions}

In practice, we are often interested in general convex loss functions other than the squared loss.
Here, we consider a generalized problem of the following form:
\begin{equation}\label{eq:general_loss_neariso}
    \hat{\theta} \in \argmin_{\theta \in \RR^p} \left\{
        \mcL(\theta; y) + \lambda \mcV_G(\theta)
    \right\},
\end{equation}
where $\theta \mapsto \mcL(\theta; y)$ is a convex loss function for any $y \in \RR^n$.
As an example, this formulation contains the $M$-estimator in the regression setting $\mcL(\theta; y) = \frac{1}{2} \ell(y_i - \langle x_i, \theta \rangle)$, where $(y_i, x_i) \in \RR \times \RR^p$ ($i = 1, 2, \ldots, n$) are the observed data and $\ell: \RR \to \RR$ is a convex function.

We can also obtain algorithms that output approximate minimizers of \eqref{eq:general_loss_neariso} as follows. First of all, note that Algorithm \ref{alg:divide_conquer} outputs the \textit{proximal operator} of the regularization term $\mcV_G(\theta)$. Once we have an oracle for the proximal operator, we can apply \textit{proximal gradient methods} to solve \eqref{eq:general_loss_neariso}. In particular, if $\mcL(\theta; y)$ is convex and smooth, the Fast Iterative Shrinkage Thresholding Algorithm (FISTA, \citet{Beck2009}) outputs an $\bigO(\epsilon)$-optimal solution after $\bigO(\epsilon^{-2})$ evaluations of the proximal operator.
%Here, we already have the proximal operator \eqref{eq:neariso_general_graph} of the penalty term $\lambda \mcV_G$,

\subsection{Constrained estimators}

Consider the following generalized version of the constrained form of nearly-isotonic regression \eqref{eq:neariso_constrained}:
\begin{equation}\label{eq:neariso_constrained_general}
    \text{minimize} \ \norm{y - \theta}_2^2 \quad
    \text{subject to} \ \sum_{(i, j) \in E} c_{(i, j)} (\theta_i - \theta_j)_+ \leq \mcV,
\end{equation}
Unlike the penalized estimators, it is difficult to find an exact solution of \eqref{eq:neariso_constrained_general}.
However, since problem \eqref{eq:neariso_constrained_general} is an instance of a quadratic programming problem, there are polynomial time algorithms to obtain approximate solutions.
Here, we explain the \textit{existence} of such algorithms.
The following result is a direct application of Theorem 1 by \citet{Lee2018}, which provides a convergence guarantee of a variant of cutting plane methods.

\begin{prop}\label{prop:constraint_algorithm_theory}
Suppose that $G = ([n], E)$ is a directed graph equipped with positive weights $c_{(i, j)}$ for every $(i, j) \in E$.
Let $y \in \RR^n$ be any vector and $\mcV > 0$.
Then, for any $\epsilon > 0$, there exists a randomized algorithm that outputs $\tilde{\theta}$ satisfying
\[
    \mcV_G(\tilde{\theta}) := \sum_{(i, j) \in E} c_{(i, j)} (\tilde{\theta}_i - \tilde{\theta}_j)_+
    \leq \mcV + 2 \epsilon \sum_{(i, j) \in E} c_{(i, j)}
\]
and
\[
    \norm{y - \tilde{\theta}}_2
    \leq \min_{\theta \in \RR^n: \ \mcV_G(\theta) \leq \mcV} \norm{y - \theta}_2
    + 2 \epsilon \norm{y}_2
\]
with a probability of $0.99$.
The overall complexity of the algorithm is $\bigO((n + |E|)n^2 \log^{\bigO(1)} \frac{n}{\epsilon |E|})$.
\end{prop}

\begin{rmk}
In practice, due to computational considerations, we recommend to use the penalized estimator \eqref{eq:neariso_penalty_general} instead of the constrained estimator \eqref{eq:neariso_constrained_general}. For the penalized estimator, we empirically observed that Algorithm \ref{alg:divide_conquer} runs sufficiently fast graphs with several hundreds of nodes. For the constrained estimator, Proposition \ref{prop:constraint_algorithm_theory} theoretically guarantees polynomial time solvability of the constrained problem \eqref{eq:neariso_constrained_general}, whereas it does not provide a practical algorithm.
\end{rmk}

\section{Supplemental experiments}\label{sec:additional_experiments}

To understand the behavior of the nearly-isotonic regression in more generic settings, we present additional simulation results for the nearly-isotonic regression on general graphs \eqref{eq:neariso_general_graph}.
Here, we consider the problem of estimating piecewise monotone signals on two-dimensional grids.

We say that an $n_1 \times n_2$ matrix $\theta$ is monotone if $\theta_{ij} \leq \theta_{kl}$ whenever $i \leq k$ and $j \leq l$.
In other words, $\theta$ is monotone if it has no order-violating edges in the two-dimensional grid graph $G_2 = (V_2, E_2)$, where $V_2 = [n_1] \times [n_2]$ is the set of all subscripts $(i, j)$ and
\begin{align*}
    E_2 := & \set{((i, j), (i, j+1)) \ : \ 1 \leq i \leq n_1, 1 \leq j \leq n_2 - 1}\\
    & \cup \set{((i, j), (i+1, j)) \ : \ 1 \leq i \leq n_1 - 1, 1 \leq j \leq n_2}.
\end{align*}
We say that $\theta$ is piecewise monotone if there is a partition $\Pi$ of $V$ such that, for each $A \in \Pi$, $A$ is a weakly connected component of $G_2$ and $\theta_A$ has no order-violating edges in the induced subgraph.
For simplicity of experimental settings, we here only consider ``block'' type partitions, i.e., we say that $\Pi$ is of block type if it can be represented as a product of two partitions of the two coordinates.
The left panel in Figure \ref{fig:neariso_grid_2d_example} is an example of two-dimensional piecewise monotone signals on a block type partition.

\begin{figure}[tb]
    \centering
    \includegraphics[width=0.95\linewidth]{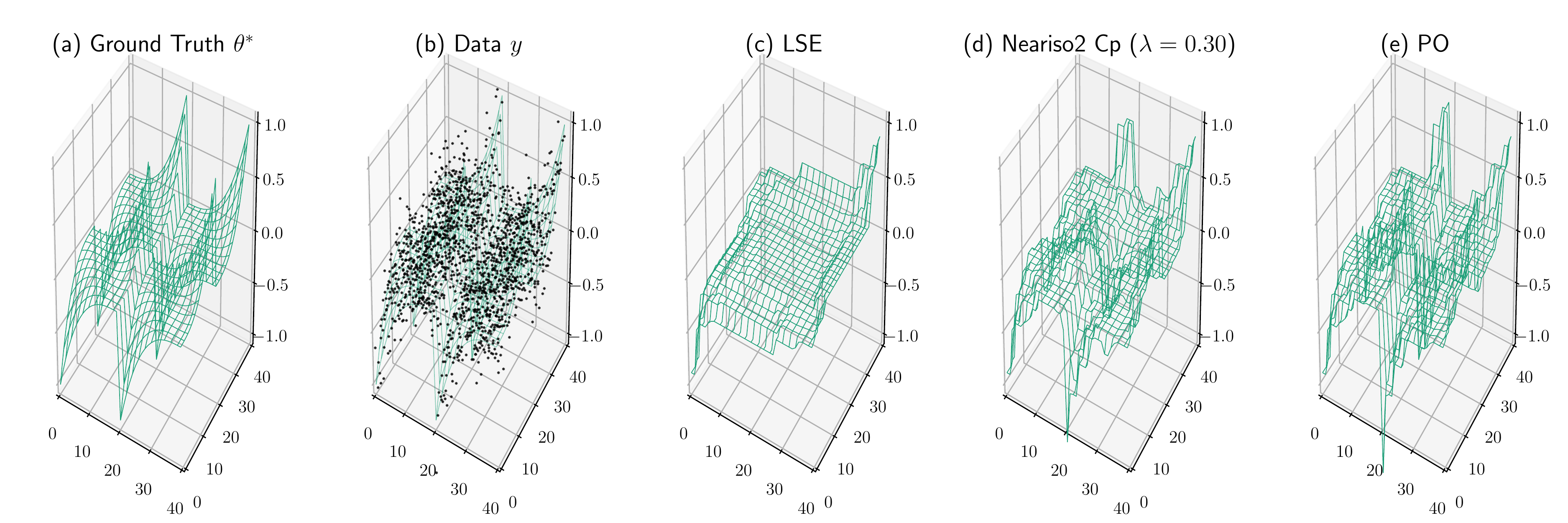}
    \caption{\textbf{Examples of estimators for piecewise monotone matrices.}
    The true parameter $\theta^*$ is a $32 \times 32$ matrix that is monotone on each $16 \times 16$ segment.
    The bivariate isotonic regression (\texttt{LSE}) does not capture the piecewise monotone structure.
    The solution of the nearly-isotonic regression (\texttt{Neariso2}) seems to be close to the partition oracle (\texttt{PO}).
    }
    \label{fig:neariso_grid_2d_example}
\end{figure}

We compare the following three estimators:
\begin{itemize}
    \item \texttt{LSE}: The bivariate isotonic regression (see e.g., \citet{Robertson88}).
    \item \texttt{Neariso2}: The two-dimensional nearly-isotonic regression with $C_p$-tuned parameter.
    \item \texttt{PO}: The bivariate isotonic regression applied to the true partition.
\end{itemize}
For monotone matrices, \citet{Chatterjee2018} proved that \texttt{LSE} is minimax rate optimal with respect to $n = n_1 n_2$.
Hence, the partition oracle estimator \texttt{PO} can be regarded as an ideal benchmark that is minimax optimal over piecewise monotone matrices.
On the other hand, if the true matrix $\theta^*$ is piecewise monotone, the risk of \texttt{LSE} can be arbitrarily large for the same reason as Proposition \ref{prop:suboptimal}.
\texttt{Neariso2} is the special case of the generalized nearly-isotonic regression \eqref{eq:neariso_general_graph} applied to the graph $G_2$ defined above.
\texttt{Neariso2} was originally discussed in \citet{Tibshirani2011}, but no experimental results have been presented.
Figure \ref{fig:neariso_grid_2d_example} shows examples of the solutions of the three estimators.

We construct an $n \times n$ matrix $\theta^*$ as follows: 
We define a $k \times k$ small monotone matrix $U$, and then we define $\theta^*$ as an $mk \times mk$ block matrix by repeating $U$ for $m$ times both in rows and columns (thus $n = mk$).
We choose the small matrix $U = (U_{ij})$ from
\[
    U^{\mathrm{cubic2d}}_{ij} = (x_i + x_j - 1)^3
\]
or
\[
    U^{\mathrm{cubic1d}}_{ij} = (2 x_i - 1)^3,
\]
where we write $x_i = \frac{i-1}{k-1}$ for $i = 1, 2, \ldots, k$.
With the former choice, $\theta^*$ becomes an $m^2$-piecewise monotone matrix.
With the latter choice, $\theta^*$ becomes an $m$-piecewise monotone matrix such that $\theta^*_{ij}$ does not depend on $j$.

We generated noisy observations $y$ by adding independent Gaussian noises $\xi_{ij} \sim N(0, (0.25)^2)$ to every entries of $\theta^*$.
To estimate the MSE, we used 500 replications of the data.
Figure \ref{fig:mse_2d_cubic} shows the results.
Clearly, the risks of \texttt{LSE} (blue triangles) are much larger than those of the other two estimators.
\texttt{Neariso2} (green circles) has slightly larger risks compared to \texttt{PO} (magenta squares), while their slopes seem to be close.

To visualize convergence rates, we fit the risks of \texttt{PO} by monomials $\propto n^{-a}$ ($a > 0$), and plotted as dashed lines in Figure \ref{fig:mse_2d_cubic}.
The values of the exponent $a$ are respectively as follows: $0.58$ (\texttt{cubic2d}, $m = 2$); $0.56$ (\texttt{cubic2d}, $m = 4$); $0.50$ (\texttt{cubic1d}, $m = 2$); $0.45$ (\texttt{cubic2d}, $m = 4$).
We should note that, in monotone matrix estimation, the theoretical convergence rate of \texttt{LSE} is known to be $\tilde{\bigO}(n^{-1/2})$ \citep{Chatterjee2018}.
% On the other hand, the rates for \texttt{cubic1d} are slower than variable adaptation rates $\tilde{\bigO}(n^{-2/3})$ suggested in \citep{Chatterjee2018}, Theorem 2.4.

\begin{figure}[tb]
    \centering
    \includegraphics[width=0.95\linewidth]{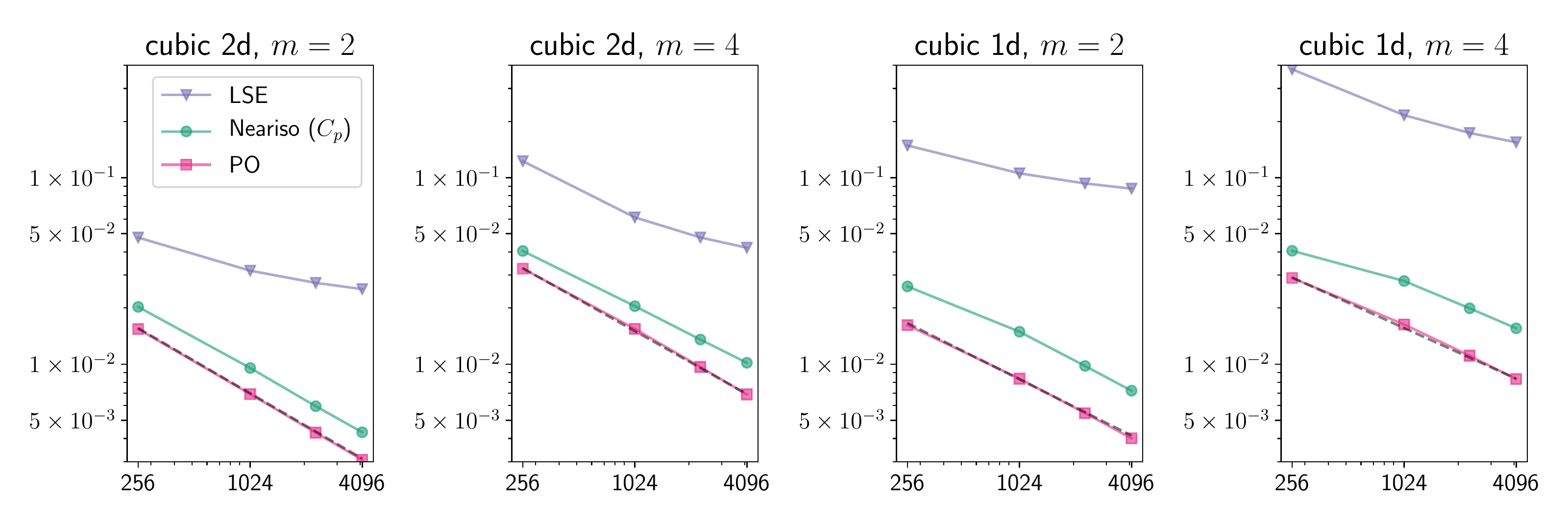}
    \caption{\textbf{The risks in piecewise monotone matrix estimation}.
    See the text for details.}
    \label{fig:mse_2d_cubic}
\end{figure}

\section{Proofs in Section \ref{sec:lower}}\label{sec:proof_1}

\subsection{Proof of Proposition \ref{prop:general_lower_bound}}

Let $\Theta$ be either $\tilde{\Theta}_n(m, \mcV)$ or $\Theta_n(m, \mcV)$, which are defined in Definition \ref{defi:piecewise_class}.
The minimax lower bound \eqref{eq:general_lower_bound} is proved by combining the following two lower bounds:

\begin{enumerate}[label=(\roman*)]
    \item \textbf{(Lower bound for monotone vectors \citep{Zhang2002, Chatterjee2015})}
    Let $\mcK(\mcV) = \set{ \theta \in K_n^\uparrow: \mcV(\theta) \leq \mcV }$ be the set of monotone vectors with bounded total variations.
    There is a universal constant $C_1 > 0$ such that for any estimator $\hat{\theta}$,
    \[
        \sup_{\theta^* \in \mcK(\mcV)} \frac{1}{n} \EE_{\theta^*} \norm{\hat{\theta} - \theta^*}_2^2 \geq C_1 \left( \frac{\sigma^2 \mcV}{n} \right)^{2/3}.
    \]
    \item \textbf{(Lower bound for piecewise constant vectors)}
    Let $\mcC(m)$ be the set of $m$-piecewise constant vectors in $\RR^n$, i.e.,
    $\theta \in \mcC(m)$ if $|\set{i : \theta_i \neq \theta_{i+1}}| \leq m - 1$.
    The minimax lower bound over $\mcC(m)$ can be related to sparse estimation as follows.
    Let $X$ be an $n \times n$ matrix whose $(i, j)$ entries are given as $1_{\set{i \geq j}}$.
    Then, $\mcC(m)$ contains the set $\set{ \theta = X\beta: \norm{\beta}_0 \leq m}$, and the lower bound for the minimax risk over $\mcC(m)$ follows from the well-known results for $\ell_0$ balls (e.g., \citet{Raskutti2011}, Theorem 3-(b)).
    In particular, for any $m \geq 3$, the following lower bound is presented in \citet{Gao2017}:
    \[
        \sup_{\theta^* \in \mcC(m)} \frac{1}{n} \EE_{\theta^*} \norm{\hat{\theta} - \theta^*}_2^2
        \geq C_2 \frac{\sigma^2 m}{n} \log \frac{en}{m},
    \]
    where $C_2 > 0$ is a universal constant.
\end{enumerate}

It remains to show that $\Theta$ contains $\mcK(\mcV)$ and $\mcC(m)$.
$\mcC(m) \subseteq \Theta$ is obvious because an $m$-piecewise constant vector is also an $m$-piecewise monotone vector such that the piecewise total variations are zero.
From the definition, it is also clear that $\mcK(\mcV) \subseteq \tilde{\Theta}_n(m, \mcV)$.
If $\theta \in \mcK(\mcV)$, the jumps $\theta_{i + 1} - \theta_i$ that strictly exceeds $\mcV / m$ cannot occur more than $m - 1$ times.
Hence, we can choose a partition $\Pi$ with $|\Pi| \leq m$ so that each $A \in \Pi$ does not contain such large jumps, which implies that $\theta \in \Theta_n(m, \mcV)$.

\subsection{Proof of Proposition \ref{prop:suboptimal}}

The following theorem in the seminal paper of \citet{Chatterjee2014} provides useful upper and lower bounds for the risk of the least square estimator over any closed convex set $K$.

\begin{thm}[\citet{Chatterjee2014}, Corollary 1.2]\label{thm:chatterjee_method}
Let $K \subseteq \RR^n$ be any closed convex set, and let $\hat{\theta}_K$ denote the least squares estimator over $K$.
For any $\theta^* \in \RR^n$, define the function $g_{\theta^*}: \RR_{+} \to \RR \cup \set{- \infty}$ as
\[
    g_{\theta^*}(t) := \EE_{Z \sim N(0, \sigma^2 I_n)} \left[
        \sup_{\theta \in K: \norm{\theta - \theta^*}_2 \leq t} \langle Z, \theta - \theta^* \rangle
    \right] - \frac{t^2}{2}.
\]
Here, if the set $\set{\theta \in K: \norm{\theta - \theta^*}_2 \leq t}$ is empty, we define $g_{\theta^*}(t) = - \infty$.
Then, $g_{\theta^*}$ is strictly concave for $t \geq \mathrm{dist}(\theta^*, K)$ and has a unique maximizer $t_{\theta^*}$.
Moreover, there are universal constants $C_1, C_2 > 0$ such that
\begin{equation}\label{eq:chatterjee_method}
    \frac{1}{n} \max \left\{ t_{\theta^*}^2 - C_1 t_{\theta^*}^{3/2}, 0\right\}
    \leq \frac{1}{n} \EE_{\theta^*} \norm{\hat{\theta}_K - \theta^*}_2^2
    \leq \frac{C_2}{n} \max \left\{ t_{\theta^*}^2, \sigma^2 \right\}.
\end{equation}
\end{thm}

To prove Proposition \ref{prop:suboptimal}, we use the lower bound in \eqref{eq:chatterjee_method}.
Note that for a sufficiently large $t_0 > 0$, $t \mapsto t^2 - Ct^{3/2}$ is a strictly increasing in $t \in [t_0, \infty)$.
For any $n$ and $\sigma^2$, choose $t \geq t_0$ so that $t^2 - C t^{3/2} \geq n \sigma^2$.
Then, for any $\theta^*$ such that $\mathrm{dist}(\theta^*, K) \geq t$, we have
\[
    \frac{1}{n} \EE_{\theta^*} \norm{\hat{\theta}_K - \theta^*}_2^2
    \geq \frac{1}{n} (t_{\theta^*}^2 - C_1 t_{\theta^*}^{3/2})
    \geq \frac{1}{n} (t^2 - C_1 t^{3/2}) \geq \sigma^2.
\]

\begin{rmk}
We should note that the above proof is valid for \textit{any} closed convex set $K$.
For the specific choice of $K = K_n^\uparrow$, the lower bound of $t_{n, \sigma^2}$ used in the proof can be quite conservative.
In practice, the risk of the isotonic regression estimator can be larger than $\sigma^2$ under a smaller value of $\ell_2$-misspecification error.
\end{rmk}

\section{Proofs in Section \ref{sec:adaptation}}\label{sec:appendix_proof_adaptation}

\subsection{Preliminaries}

%\com{tangent cone, statistical dimension, gaussian width, etc.}
To state the results for risk upper bounds, we first introduce some quantities related to Gaussian processes.

\begin{defi}\label{defi:gaussian_width}
Let $C$ be a closed convex set in $\RR^n$.
Let $\EE$ denote the expectation with respect to an isotropic Gaussian random variable $Z \sim N(0, I_n)$.
\begin{enumerate}[label=(\roman*)]
    \item The \textit{Gaussian width} of $C$ is defined as
    \[
        w(C) := \EE \left[
            \sup_{\theta \in C} \langle Z, \theta \rangle
        \right].
    \]
    \item The \textit{Gaussian mean squared distance} is defined as
    \[
        \mbD(C) := \EE [\dist^2(Z, C)],
    \]
    where $\dist(z, C) := \inf_{x \in C} \norm{x - z}_2$.
    \item Suppose that $C$ is a convex cone. The \textit{statistical dimension} of $C$ is defined as
    \[
        \delta(C) := \EE \left[ \left(
            \sup_{\theta \in C: \norm{\theta}_2 \leq 1} \langle Z, \theta \rangle
        \right)^2 \right].
    \]
\end{enumerate}
\end{defi}

We present some historical remarks on these definitions.
The three quantities in Definition \ref{defi:gaussian_width} can be interpreted as complexity measures for the subset $C$ in the Euclidean space.
The Gaussian width has been well studied in convex geometry, signal processing, high-dimensional statistics, and empirical process theory; See e.g., Section 7.8 in \citet{Vershynin} for a literature review.
The definition of the Gaussian mean squared distance is due to \citet{Oymak2016}. As we will see in Lemma \ref{lem:proximal_general} below, the Gaussian mean squared distance is useful to provide the risk bounds for proximal denoising estimators.
The statistical dimension was defined in \citet{Amelunxen2014}. Recently, \citet{Bellec2015b} pointed out that the statistical dimension characterizes the adaptive risk bounds for some shape restricted estimators including the isotonic regression and the convex regression.

As suggested by the definitions, these three quantities are closely related to each other.
In particular, if $C$ is a convex cone, these are comparable as follows.

\begin{prop}\label{prop:statistical_dim_compare}
Let $C$ be a closed convex cone.
\begin{enumerate}[label=(\roman*)]
    \item (\citet{Amelunxen2014}, Proposition 10.2) Let $S_{n-1} = \set{x \in \RR^n: \norm{x}_2 = 1}$ be the unit sphere in $\RR^n$. Then, we have $w^2(C \cap S_{n-1}) \leq \delta(C) \leq w^2(C \cap S_{n-1}) + 1$.
    \item (\citet{Amelunxen2014}, Proposition 3.1) Let $C^\circ$ be the polar cone of $C$ defined as
    \[
         C^\circ := \set{x \in \RR^n: \langle x, z \rangle \leq 0 \; \text{for all $z \in C$}}.
    \]
    Then, we have $\mbD(C) = \delta(C^\circ)$.
\end{enumerate}
\end{prop}

Now, we introduce two general results for risk bounds for general projection estimators and proximal denoising estimators.

Let $K$ be a closed convex set in $\RR^n$, and define the projection estimator onto $K$ as $\hat{\theta}_K = \argmin_{\theta \in K} \norm{y - \theta}_2$.
\citet{Bellec2015b} proved the following oracle inequality that relates the risk of the projection estimator to the statistical dimension of the tangent cone of $K$.
Here, the tangent cone $T_K(\theta)$ of $K$ at $\theta \in K$ is defined as
\[
    T_K(\theta) := \mathord{\mathrm{closure}} (\set{ t(z - \theta): t \geq 0, z \in K}).
\]

\begin{lem}[\citet{Bellec2015b}, Corollary 2.2]\label{lem:projection_general}
Let $\theta^* \in \RR^n$ be any vector, and suppose that the observation $y$ is drawn according to $N(\theta^*, \sigma^2 I_n)$.
Then, we have the following risk bound:
\[
    \frac{1}{n} \EE_{\theta^*} \norm{\hat{\theta}_K - \theta^*}_2^2
    \leq \inf_{\theta \in K} \left\{
        \frac{1}{n} \norm{\theta - \theta^*}_2^2
        + \frac{\sigma^2}{n} \delta(T_K(\theta))
    \right\}.
\]
Moreover, for any $\eta \in (0, 1)$, the inequality
\[
    \frac{1}{n} \norm{\hat{\theta}_K - \theta^*}_2^2
    \leq \inf_{\theta \in K} \left\{
        \frac{1}{n} \norm{\theta - \theta^*}_2^2
        + \frac{2 \sigma^2}{n} \delta(T_K(\theta))
    \right\}
    + \frac{4 \sigma^2 \log(\eta^{-1})}{n}
\]
holds with probability at least $1 - \eta$.
\end{lem}

Next, we provide a general result for proximal denoising estimators.
Let $f: \RR^n \to \RR$ be a convex function, and $\lambda \geq 0$.
We define the proximal denoising estimator $\hat{\theta}_\lambda$ as
\begin{equation}\label{eq:proximal_denoising}
    \hat{\theta}_\lambda := \argmin_{\theta \in \RR^n} \left\{
        \frac{1}{2} \norm{ y - \theta }_2^2 + \sigma \lambda f(\theta)
    \right\}.
\end{equation}The class of proximal denoising estimators contains the soft-thresholding estimator \citep{Donoho1992}, the total variation regularization \citep{Rudin1992}, the trend filtering \citep{Kim2009} and the nearly-isotonic regression \citep{Tibshirani2011}.
\citet{Oymak2016} pointed out that the risk bound of proximal denoising estimators can be characterized by the Gaussian mean squared distance of the set $\lambda \partial f(\theta^*)$.
Remarkably, based on this technique, \citet{Guntuboyina2017a} proved sharp adaptation results for the trend filtering estimators.
The following oracle inequality can be regarded as a generalization of Theorem 2.2 in \citet{Oymak2016}. For the sake of completeness, we also provide its proof below.

\begin{lem}\label{lem:proximal_general}
Let $\theta^* \in \RR^n$ be any vector, and suppose that the observation $y$ is drawn according to $N(\theta^*, \sigma^2 I_n)$.
Let $f: \RR^n \to \RR$ be a convex function, and let $\hat{\theta}_\lambda$ denote the proximal denoising estimator defined as \eqref{eq:proximal_denoising}.
Then, we have
\begin{equation}\label{eq:proximal_oracle_inequality}
    \frac{1}{n} \EE_{\theta^*} \norm{\hat{\theta}_\lambda - \theta^*}_2^2
    \leq \inf_{\theta \in \RR^n} \left\{
        \frac{1}{n} \norm{\theta - \theta^*}_2^2
        + \frac{\sigma^2}{n} \mbD(\lambda \partial f(\theta))
    \right\}.
\end{equation}
Moreover, for any $\eta \in (0, 1)$, the inequality
\begin{equation}\label{eq:proximal_oracle_inequality_hp}
    \frac{1}{n} \norm{\hat{\theta}_\lambda - \theta^*}_2^2
    \leq \inf_{\theta \in \RR^n} \left\{
        \frac{1}{n} \norm{\theta - \theta^*}_2^2
        + \frac{2 \sigma^2}{n} \mbD(\lambda \partial f(\theta^*))
    \right\}
    + \frac{16 \sigma^2 \log(\eta^{-1})}{n}
\end{equation}
holds with probability at least $1 - \eta$.
\end{lem}
\begin{proof}
Below, we write $\hat{\theta} := \hat{\theta}_\lambda$.
To prove \eqref{eq:proximal_oracle_inequality}, it suffices to show that we have almost surely
\[
    \norm{\hat{\theta} - \theta^*}_2^2 - \norm{\theta - \theta^*}_2^2
    \leq \sigma^2 \mbD(\lambda \partial f(\theta))
\]
for any fixed vector $\theta \in \RR^n$.
We will assume $\theta \neq \hat{\theta}$ because otherwise the inequality is trivial.

From the first order optimality condition of the convex minimization problem \eqref{eq:proximal_denoising}, we have
\[
    \langle \theta - \hat{\theta}, y - \hat{\theta} \rangle \leq \sigma \lambda (f(\theta) - f(\hat{\theta}))
    \quad \text{for any $\theta \in \RR^n$}.
\]
See Lemma 6.1 in \citet{vandeGeer2015} for a formal proof.
Using the elementary fact that $2 \langle u, v \rangle = \norm{u}_2^2 + \norm{v}_2^2 - \norm{u - v}_2^2$ and substituting $y = \theta^* + \sigma z$, we have
\begin{equation}\label{eq:proximal_general_prf1}
    \norm{\hat{\theta} - \theta^*}_2^2 - \norm{\theta - \theta^*}_2^2
    \leq 2 \sigma \lambda  (f(\theta) - f(\hat{\theta})) - 2 \sigma \langle z, \theta - \hat{\theta} \rangle
    - \norm{\theta - \hat{\theta}}_2^2.
\end{equation}

Now, take $v \in \partial f(\theta)$ arbitrarily.
From the definition of the subgradient, we have
\[
    f(\theta) - f(\hat{\theta}) \leq \langle v, \theta - \hat{\theta} \rangle.
\]
Hence, the right-hand side of \eqref{eq:proximal_general_prf1} is bounded from above by
\begin{align*}
    & 2 \sigma \langle \lambda v - z, \theta - \hat{\theta} \rangle
    - \norm{\theta - \hat{\theta}}_2^2 \\
    & = 2 \sigma \left \langle \lambda v - z, \frac{\theta - \hat{\theta}}{\norm{\theta - \hat{\theta}}_2} \right \rangle \norm{\theta - \hat{\theta}}_2
    - \norm{\theta - \hat{\theta}}_2^2\\
    & \leq \sigma^2 \left \langle \lambda v - z, \frac{\theta - \hat{\theta}}{\norm{\theta - \hat{\theta}}_2} \right \rangle^2
    \quad (\because 2ab - b^2 \leq a^2)\\
    & \leq \sigma^2 \norm{\lambda v - z}_2^2 \quad \text{( $\because$ The Cauchy--Schwarz inequality)}.
\end{align*}
Since the choice of $v \in \partial f(\theta)$ is arbitrary, we have
\begin{equation}\label{eq:proximal_general_prf2}
    \norm{\hat{\theta} - \theta^*}_2^2 - \norm{\theta - \theta^*}_2^2
    \leq \sigma^2 \inf_{v \in \partial f(\theta)} \norm{\lambda v - z}_2^2
    = \sigma^2 \mathrm{dist}^2(z, \lambda \partial f(\theta)).
\end{equation}
By taking the expectation of both sides, \eqref{eq:proximal_oracle_inequality} is proved.

To prove the high-probability bound \eqref{eq:proximal_oracle_inequality_hp}, we use the well-known Gaussian concentration inequality (see e.g., Theorem 5.6 in \citet{Boucheron2013}); for any $L$-Lipschitz function $h: \RR^n \to \RR$ and $\eta \in (0, 1)$, we have
\[
    \mathrm{Pr}_{Z \sim N(0, I_n)}\left\{
        h(Z) - \EE[h] \geq \sqrt{2L^2 \log \eta^{-1}}
    \right\} \leq \eta.
\]
In fact, the map $z \mapsto \dist(z, \lambda \partial f(\theta))$ is a $2$-Lipschitz function because, for any $z_1, z_2 \in \RR^n$, we have
\[
    |\dist(z_1, \lambda \partial f(\theta)) - \dist(z_2, \lambda \partial f(\theta))|
    \leq \Norm{(z_1 - P(z_1)) - (z_2 - P(z_2))}_2 \leq 2 \norm{z_1 - z_2}_2,
\]
where $P$ is the orthogonal projection map onto the set $\lambda \partial f(\theta)$.
Now, we take $\bar{\theta}$ as
\[
    \bar{\theta} \in \argmin_{\theta \in \RR^n} \left\{
        \norm{\theta - \theta^*}_2^2
        + \sigma^2 \left(\sqrt{\mbD(\lambda \partial f(\theta))} + \sqrt{8 \log \eta^{-1}} \right)^2
    \right\}.
\]
Combining \eqref{eq:proximal_general_prf2} and the Gaussian concentration applied for $\theta = \bar{\theta}$, we have the desired result.
\if0
Hence, by the Gaussian concentration and the inequality \eqref{eq:proximal_general_prf2} with $\theta = \theta^*$, we have
\[
    \mathrm{Pr}_{y \sim N(\theta^*, \sigma^2 I)} \left\{
        \norm{\hat{\theta} - \theta^*}_2^2 \leq
        \sigma^2 \left(\mbD(\lambda \partial f(\theta^*)) + \sqrt{8 \log \eta^{-1}}\right)^2
    \right\} \geq 1 - \eta,
\]
which implies \eqref{eq:proximal_oracle_inequality_hp}.
\fi
\end{proof}

\subsection{Risk bounds for constrained estimators (Proof of Theorem \ref{thm:constrained_tangent})}

In this subsection, we provide the proof of Theorem \ref{thm:constrained_tangent} as an application of Lemma \ref{lem:projection_general}.
To this end, we have to evaluate the statistical dimension of the tangent cone of a convex set
\begin{equation}
    K_-(\mcV) := \set{\theta \in \RR^n: \mcV_-(\theta) \leq \mcV}
    = \left\{
        \theta \in \RR^n: \sum_{i=1}^{n-1}(\theta_i - \theta_{i+1})_+ \leq \mcV
    \right\}.
\end{equation}
It is not surprising that the analysis of the tangent cone of $K_-(\mcV)$ goes very similar to that of the set with bounded total variation $K(\mcV) = \set{\theta \in \RR^n: \mcV(\theta) \leq \mcV}$ in \citet{Guntuboyina2017a}.
Our goal is to show the following upper bound for the statistical dimension:
\begin{prop}\label{prop:statistical_dimension_bound}
Suppose that $\theta$ is a vector with $\mcV_-(\theta) = \mcV$.
Then, there exists a universal constant $C > 0$ such that
\[
    \delta(T_{K_-(\mcV)}(\theta))
    \leq C n \left\{
        \frac{k(\theta)}{n} \log \frac{\ee n}{k(\theta)}
        + \frac{M(\theta)}{k(\theta)} \log \frac{\ee n}{k(\theta)}
    \right\},
\]
where $M(\theta)$ is defined in \eqref{eq:m_definition}.
\end{prop}
We briefly outline the proof for this result.
We divide the proof into four steps:
First, we provide some useful characterizations of the tangent cone.
Second, we decompose the tangent cone into finitely many pieces so that the Gaussian widths become easy to evaluate.
Third, we provide the concrete upper bounds the Gaussian widths of these pieces.
Lastly, we combine the upper bounds and apply Lemma \ref{lem:projection_general} to complete the proof.

\noindent {\bfseries\upshape Step 1: Characterizing the tangent cone \ }
If $\mcV_-(\theta) < \mcV$, $\theta$ is contained in the interior of $K_-(\mcV)$, and the tangent cone becomes the entire Euclidean space $\RR^n$.
Hereafter, we assume that $\theta$ lies on the boundary of $K_-(\mcV)$, that is, $\mcV_-(\theta) = \mcV$.
Let us recall the definition of the sign of jumps $w_i$ in \eqref{eq:sign_definition}.
Roughly speaking, the tangent cone of $K_-(\mcV)$ is characterized by the sign of jumps.

\begin{lem}\label{lem:tangent_characterization}
Let $\theta$ be a vector in $\RR^n$ such that $\mcV_-(\theta) = \mcV$.
Let $\Pi = \{ B_1, B_2, \ldots, B_{k'} \}$ be any connected refinement
\footnote{
    Here, we say that $\Pi$ is a connected refinement of another connected partition $\Pi'$ if, for any $B \in \Pi$, there exists a unique element $A \in \Pi'$ such that $B \subseteq A$.
}
of the constant partition $\Pi_\mathrm{const}(\theta)$ of $\theta$.
Let $1 = \tau_1 < \tau_2 < \cdots < \tau_{k'} < \tau_{k' + 1} = n + 1$ be a sequence such that $B_i = \set{\tau_i, \tau_i + 1, \ldots, \tau_{i+1} - 1}$ for any $i \in \set{1, 2, \ldots, k'}$.
We define the signs $w_2, w_3, \ldots, w_{k'} \in \{ 0, 1 \}$ as
\[
    w_i = \left\{
        \begin{aligned}
            & 1 & \quad \text{if $\theta_{\tau_i - 1} > \theta_{\tau_i}$} \\
            & 0 & \quad \text{if $\theta_{\tau_i - 1} < \theta_{\tau_i}$} \\
            & \text{arbitrary value in $\{ 0, 1\}$} & \quad \text{if $\theta_{\tau_i - 1} = \theta_{\tau_i}$}
        \end{aligned}
    \right. .
\]
For any $\Pi$ and $w_2, w_3, \ldots, w_{k'}$ taken as above, we define a convex cone $T(\Pi, w)$ as
\begin{equation}\label{eq:lem_tangent_characterization}
    T(\Pi, w) = \left\{
        v \in \RR^n:
        \sum_{i=1}^{k'} \mcV_-^{B_i}(v_{B_i})
        \leq \sum_{i=2}^{k'} w_i (v_{\tau_i} - v_{\tau_i - 1})
    \right\},
\end{equation}
where $\mcV_-^{B_i}(v_{B_i})$ is the lower total variation for the restricted vector $v_{B_i}$.
Then, for the tangent cone $T_{K_-(\mcV)}(\theta)$, we have the followings:
\begin{enumerate}[label=(\roman*)]
    \item If $\Pi = \Pi_\mathrm{const}(\theta)$, then $T_{K_-(\mcV)}(\theta) = T(\Pi, w)$.
    \item If $\Pi$ is a connected refinement of $\Pi_\mathrm{const}(\theta)$ and $w$ is taken arbitrarily as above, then $T_{K_-(\mcV)}(\theta) \subseteq T(\Pi, w)$.
\end{enumerate}
\end{lem}

\begin{proof}
First, we show that $T_{K_-(\mcV)}(\theta) \subseteq T(\Pi, w)$.
By the definition of the tangent cone, it suffices to show that $v := z - \theta \in T(\Pi, w)$ holds for any $z \in K_-(\mcV)$.
Note that $\theta$ is constant on every $B_i \in \Pi$ since $\Pi$ is finer than the constant partition of $\theta$.
Since the lower total variation is not changed by adding any constant value to each coordinates, we have $\mcV_-^{B_i}(z_{B_i} - \theta_{B_i}) = \mcV_-^{B_i}(z_{B_i})$.
Then, we have
\begin{align*}
    & \sum_{i=1}^{k'} \mcV_-^{B_i}(v_{B_i}) - \sum_{i=2}^{k'} w_i (v_{\tau_i} - v_{\tau_i - 1}) \\
    & = \sum_{i=1}^{k'} \mcV_-^{B_i}(z_{B_i})
    + \sum_{i=2}^{k'} w_i (z_{\tau_i - 1} - z_{\tau_i})
    - \sum_{i=2}^{k'} w_i (\theta_{\tau_i - 1} - \theta_{\tau_i}) \\
    & \underbrace{
        \leq \sum_{i=1}^{k'} \mcV_-^{B_i}(z_{B_i})
        + \sum_{i=2}^{k'} (z_{\tau_i - 1} - z_{\tau_i})_+
    }_{= \mcV_-(z) \leq \mcV}
    - \underbrace{
        \sum_{i=2}^{k'} w_i (\theta_{\tau_i - 1} - \theta_{\tau_i})
    }_{= \mcV_-(\theta) = \mcV} \\
    & \leq 0,
\end{align*}
which proves $v \in T(\Pi, w)$ and hence (ii).

Next, we prove that $T(\Pi, w) \subseteq T_{K_-(\mcV)}(\theta)$ under the assumption $\Pi = \Pi_{\mathrm{const}}(\theta) = \set{B_1, B_2, \ldots, B_k}$.
In this case, the definition of $w_2, \ldots, w_{k}$ coincides that in \eqref{eq:sign_definition}.
Fix any $v \in T(\Pi, w)$.
We want to show that $z$ is obtained as $v = t(z - \theta)$ for some $t > 0$ and $z \in K_-(\mcV)$.
To this end, we check that there exists a (sufficiently small) $t^{-1} > 0$ such that $\theta + t^{-1} v \in K_-(\mcV)$.
Here, we have
\begin{align*}
    \mcV_-(\theta + t^{-1}v)
    & = \sum_{i = 1}^{k} \mcV_-^{B_i}(\theta_{B_i} + t^{-1}v_{B_i})
    + \sum_{i = 2}^{k} ((\theta_{\tau_{i} - 1} + t^{-1} v_{\tau_i - 1}) - (\theta_{\tau_{i}} + t^{-1} v_{\tau_i}))_+ \\
    & = t^{-1} \sum_{i = 1}^{k} \mcV_-^{B_i}(v_{B_i})
    + \sum_{i = 2}^{k} ((\theta_{\tau_{i} - 1} + t^{-1} v_{\tau_i - 1}) - (\theta_{\tau_{i}} + t^{-1} v_{\tau_i}))_+.
\end{align*}
Recall that $w_2, \ldots, w_k$ are chosen so that $(\theta_{\tau_i - 1} - \theta_{\tau_i})_+ = w_i (\theta_{\tau_i - 1} - \theta_{\tau_i})$.
We can choose sufficiently small $t^{-1} > 0$ so that
\[
    ((\theta_{\tau_{i} - 1} + t^{-1}v_{\tau_i - 1}) - (\theta_{\tau_{i}} + t^{-1}v_{\tau_i}))_+
    = w_i ((\theta_{\tau_{i} - 1} + t^{-1}v_{\tau_i - 1}) - (\theta_{\tau_{i}} + t^{-1}v_{\tau_i}))
\]
for every $i = 2, 3, \ldots, k$.
Indeed, if we choose $t^{-1} > 0$ so that
\[
    t^{-1} |v_{\tau_i - 1} - v_{\tau_i}| < \theta_{\tau_{i} - 1} - \theta_{\tau_i}
    \quad \text{for every $i = 2, 3, \ldots, k$},
\]
the signs of $\theta$ do not change by adding $t^{-1} v$.
Consequently, we have
\begin{align*}
    \mcV_-(\theta + t^{-1}v)
    & = t^{-1} \sum_{i = 1}^{k} \mcV_-^{B_i}(z_{B_i})
    + \sum_{i = 2}^{k} w_i ((\theta_{\tau_{i} - 1} + t^{-1}v_{\tau_i - 1}) - (\theta_{\tau_{i}} + t^{-1}v_{\tau_i})) \\
    & = \mcV_-(\theta) + t^{-1} \left\{
        \sum_{i = 1}^{k} \mcV_-^{B_i}(v_{B_i})
        + \sum_{i = 2}^{k} w_i (v_{\tau_i - 1} - v_{\tau_i})
    \right\} \\
    & \leq \mcV_-(\theta) = \mcV.
\end{align*}
This proves that $T(\Pi, w) \subseteq T_{K_-(\mcV)}(\theta)$ and hence (i).
\end{proof}

From Proposition \ref{prop:statistical_dim_compare}-(i), we can bound the statistical dimension by the Gaussian width as follows:
\[
    \delta(T_{K_-(\mcV)}(\theta)) \leq w^2(T_{K_-(\mcV)}(\theta) \cap S_{n-1}) + 1
    \leq w^2(T_{K_-(\mcV)}(\theta) \cap B_n) + 1.
\]
Here, $B_n := \{ v \in \RR^n: \norm{v}_2 \leq 1 \}$ is the unit ball in $\RR^n$.
Hence, it suffices to consider the set $T_{K_-(\mcV)}(\theta) \cap B_n$.
In analogy to Lemma B.2 in \citet{Guntuboyina2017a}, we obtain the following characterization of this set.

\begin{lem}\label{lem:gamma_control}
Let $\theta$ be a vector in $\RR^n$ such that $\mcV_-(\theta) = \mcV$.
Let $\Pi = \{ B_1, B_2, \ldots, B_{k'} \}$ be any connected refinement of $\Pi_\mathrm{const}(\theta)$.
Define the signs $w_2, w_3, \ldots, w_{k'}$ as in Lemma \ref{lem:tangent_characterization}, and let $w_1 = w_{k' + 1} = 0$.
Then, for every $v \in T_{K_-(\mcV)}(\theta)$ with $\norm{v}_2 \leq 1$, there exists indices $\ell_1 \in B_1, \ell_2 \in B_2, \ldots, \ell_{k'} \in B_{k'}$ such that
\begin{equation}\label{eq:lem_gamma_control_1}
    \sum_{i = 1}^{k'} \Gamma_i(v, \ell_i)
    \leq \left(
        \sum_{i=1}^{k'} \frac{1}{|B_i|} 1_{\set{w_i \neq w_{i+1}}}
    \right)^{\frac{1}{2}},
\end{equation}
where we define $\Gamma_i(v, \ell_i)$ as
\begin{equation}\label{eq:lem_gamma_control_2}
    \Gamma_i(v, \ell_i)
    := \mcV_-^{B_i}(v_{B_i}) - w_i(v_{\tau_i} - v_{\ell_i}) - w_{i+1}(v_{\ell_i} - v_{\tau_{i+1} - 1})
    \quad \text{for $i = 1, 2, \ldots, k'$}.
\end{equation}
\end{lem}

\begin{proof}
Fix $v \in T_{K_-(\mcV)}(\theta) \cap B_n$.
By Lemma \ref{lem:tangent_characterization}, we have
\begin{equation}\label{eq:lem_gamma_control_prf_1}
    \sum_{i=1}^{k'} \mcV_-^{B_i}(v_{B_i})
    \leq \sum_{i=2}^{k'} w_i (v_{\tau_i} - v_{\tau_i - 1})
    = \sum_{i=1}^{k' + 1} w_i (v_{\tau_i} - v_{\tau_i - 1}).
\end{equation}
Let $\ell_1 \in B_1, \ell_2 \in B_2, \ldots, \ell_{k'} \in B_{k'}$ be indices which will be specified later.
Defining $\Gamma_i(v, \ell_i)$ as in \eqref{eq:lem_gamma_control_2}, we can rewrite \eqref{eq:lem_gamma_control_prf_1} as
\begin{align}\label{eq:lem_gamma_control_prf_2}
    \sum_{i=1}^{k'} \Gamma_i(v, \ell_i)
    & \leq \sum_{i=1}^{k'} w_i(v_{\ell_i} - v_{\tau_i})
    + \sum_{i=1}^{k'} w_{i+1}(v_{\tau_{i+1} - 1} - v_{\ell_i})
    + \sum_{i=1}^{k' + 1} w_i (v_{\tau_i} - v_{\tau_i - 1}) \nonumber \\
    & = \sum_{i=1}^{k'} (w_i - w_{i+1})v_{\ell_i} \nonumber \\
    & \leq \sum_{i=1}^{k'} 1_{\set{w_i \neq w_{i+1}}} |v_{\ell_i}|
\end{align}

Now, let $t_i$ denote the $\ell_2$ norm of $v_{B_i}$ for $i = 1, 2, \ldots, k'$.
By the assumption, $\sum_{i=1}^{k'} t_i^2 = \norm{v}_2^2 \leq 1$.
Then, for any $i \in \set{1, 2, \ldots, k'}$, there exists $\ell_i \in B_i$ such that $|v_{\ell_i}| \leq t_i / \sqrt{|B_i|}$.
For these choices of $\ell_i$, the right-hand side of \eqref{eq:lem_gamma_control_prf_2} is bounded from above by
\begin{align*}
    \sum_{i=1}^{k'} \frac{t_i}{\sqrt{|B_i|}} 1_{\set{w_i \neq w_{i+1}}}
    & \leq \left(
        \sum_{i=1}^{k'} \frac{1}{|B_i|} 1_{\set{w_i \neq w_{i+1}}}
    \right)^{1/2} \left(
        \sum_{i=1}^{k'} t_i^2
    \right)^{1/2} \\
    & \leq \left(
        \sum_{i=1}^{k'} \frac{1}{|B_i|} 1_{\set{w_i \neq w_{i+1}}}
    \right)^{1/2},
\end{align*}
which proves the desired result.
\end{proof}

\begin{rmk}\label{rmk:gamma_control}
%The meaning of $\Gamma_i$ seems to be less clear than the counterpart in the original argument in \citet{Guntuboyina2017a}, albeit the proof works.
Note that $\Gamma_i(v, \ell_i)$ is always non-negative.
This is checked as follows:
First, the lower total variation is always larger than the difference of boundary points, that is, for every $v \in \RR^m$, we have
\[
    \sum_{j=1}^{m-1} (v_j - v_{j+1})_+ \geq (v_1 - v_m)_+ \geq w(v_1 - v_m),
\]
where $w$ is taken arbitrarily from $\{ 0, 1 \}$.
The equality holds if and only if $v$ is monotone non-increasing.
Then, for any $\ell \in [m]$ and $w_1, w_2 \in \{ 0, 1 \}$, we have
\[
    \mcV_-(v)
    \geq \sum_{j = 1}^{\ell - 1} (v_{j} - v_{j + 1})_+
    + \sum_{j = \ell}^{m - 1} (v_{j} - v_{j + 1})_+
    \geq w_1(v_{1} - v_{\ell}) + w_{2}(v_{\ell} - v_{m}).
\]
In particular, we obtain $\Gamma_i(v, \ell_i) \geq 0$.
If $\theta$ is monotone non-decreasing (i.e., $w_0 = w_1 = \cdots = w_{k+1} = 0$), then the right-hand side of \eqref{eq:lem_gamma_control_1} equals to $0$, and so $\Gamma_i(v, \ell_i) = 0$.
\end{rmk}

\noindent {\bfseries\upshape Step 2: Quantizing the tangent cone \ }
Now, let $\Pi = \{ B_1, B_2, \ldots, B_{k'} \}$ be a connected refinement of $\Pi_\mathrm{const}(\theta)$.
Lemma \ref{lem:gamma_control} implies that $T_{K_-(\mcV)}(\theta) \cap B_n$ is contained in the set such that $\sum_{i=1}^{k'} \lVert v_{B_i} \rVert_2^2 \leq 1$ and $\sum_{i=1}^{k'} \Gamma_i(v, \ell_i) \leq \gamma$ for some $\ell_i \in B_i$ and $\gamma > 0$.
From this perspective, we consider finitely many allocation patterns of the budgets for $\norm{v_{B_i}}_2^2$ and $\Gamma_i(v, \ell_i)$.
To be more precise, we construct a cover of the tangent cone in the following way.
Consider a triple $(\mathbf{t}, \mathbf{q}, \mathbf{l})$ such that:
\begin{enumerate}[label=(\alph*)]
    \item $\mathbf{t} = (t_1, t_2, \ldots, t_{k'})$ and $\mathbf{q} = (q_1, q_2, \ldots, q_{k'})$ are vectors consisting of non-negative numbers, and
    \item $\mathbf{l} = (\ell_1, \ell_2, \ldots, \ell_{k'})$ is a set of indices such that $\ell_i \in B_i$ for $i = 1, 2, \ldots, k'$.
\end{enumerate}
For such triple, we define a set
\begin{equation}\label{eq:quantized_cone}
    T(\mathbf{t}, \mathbf{q}, \mathbf{l}) = \left\{
        v \in \RR^n: \;
        \norm{v_{B_i}}_2^2 \leq t_i \quad \text{and} \quad
        \Gamma_i(v, \ell_i) \leq q_i \gamma
        \quad \text{for $i = 1, 2, \ldots, k'$}
    \right\},
\end{equation}
where $\gamma$ is taken as the right-hand side of \eqref{eq:lem_gamma_control_1}:
\begin{equation}\label{eq:quantize_delta}
    \gamma := \gamma(\theta, \Pi) = \left(
        \sum_{i=1}^{k'} \frac{1}{|B_i|} 1_{\set{w_i \neq w_{i+1}}}
    \right)^{\frac{1}{2}}.
\end{equation}
Then, quantizing the allocation vectors $\mathbf{t}$ and $\mathbf{q}$, we can cover the set $T_{K_-(\mcV)}(\theta) \cap B_n$ with finitely many $T(\mathbf{t}, \mathbf{q}, \mathbf{l})$s as the following lemma.

\begin{lem}\label{lem:quantization}
Suppose that $\Pi = (B_1, B_2, \ldots, B_{k'})$ is a connected refinement of $\Pi_\mathrm{const}(\theta)$.
Define the signs $w_1, w_2, \ldots, w_{k'}$ as in Lemma \ref{lem:gamma_control}.
Let $\mcQ$ be a set of allocation vectors satisfying the following condition;
there exists an integer vector $\mathbf{m} = (m_1, m_2, \ldots, m_{k'}) \in \mathbb{N}^{k'}$ such that $1 \leq m_i \leq k'$ ($i = 1, 2, \ldots, k'$) and $\sum_{i=1}^{k'} m_i \leq 2k'$, and the allocation vector $q = (q_1, q_2, \ldots, q_{k'}) \in \mathcal{Q}$ can be written as
\[
    q_i = \frac{m_i}{k'} \quad \text{for all $i = 1, 2, \ldots, k'$}.
\]
Let $\mcL$ be a set of indices $\mathbf{l} = (\ell_1, \ell_2, \ldots, \ell_{k'})$ such that $\ell_i \in B_i$ for all $i = 1, 2, \ldots, k'$.
Given $\mathbf{t}, \mathbf{q} \in \mcQ$ and $\mathbf{l} \in \mcL$, we define a set $T(\mathbf{t}, \mathbf{q}, \mathbf{l})$ as \eqref{eq:quantized_cone}.
Then, we have
\begin{equation}\label{eq:lem_quantization}
    T_{K_-(\mcV)}(\theta) \cap B_n
    \subseteq \bigcup_{\substack{\mathbf{t}, \mathbf{q} \in \mcQ,\\ \mathbf{l} \in \mcL}} T(\mathbf{t}, \mathbf{q}, \mathbf{l}).
\end{equation}
\end{lem}

\begin{proof}
Fix any vector $v$ in $T(\Pi, w) \cap B_n$.
Since $\norm{v_{B_i}}_2^2 \leq \norm{v}_2^2 \leq 1$, there exists an integer $1 \leq m_i \leq k'$ such that
\[
    \frac{m_i - 1}{k'} \leq \norm{v_{B_i}}_2^2 \leq \frac{m_i}{k'}.
\]
Summing over $i = 1, 2, \ldots, k'$, we have
\[
    \sum_{i = 1}^{k'} m_i \leq k' \sum_{i=1}^{k'} \norm{v_{B_i}}_2^2 + k' \leq 2k',
\]
which implies $\mathbf{t} = (m_1 / k', \ldots, m_{k'}/k') \in \mcQ$.

Next, by Lemma \ref{lem:gamma_control}, there exist $\mathbf{l} = (\ell_1, \ldots, \ell_{k'}) \in \mcL$ such that
$
    \sum_{i=1}^{k'} \Gamma_i(v, \ell_i) \leq \gamma.
$
Hence, for any $i$, there exists an integer $1 \leq l_i \leq k'$ such that
\[
    \frac{(l_i - 1)\gamma}{k'} \leq \Gamma_i(v, \ell_i) \leq \frac{l_i \gamma}{k'}
\]
Suppose $\gamma > 0$.
Summing over $i = 1, 2, \ldots, k'$, we have $\sum_{i=1}^{k'} l_i \leq 2k'$ and thus $\mathbf{q} = (l_1 / k', \ldots, l_{k'} / k') \in \mcQ$.
For the case of $\gamma = 0$, it is clear that $\mathbf{q} = (1/k', 1/k', \ldots, 1/k') \in \mcQ$.
\end{proof}

We should note that the cardinalities of $\mcQ$ and $\mcL$ are respectively bounded as follows:

\begin{prop}\label{prop:cardinality}
Let $\mcQ$ and $\mcL$ are the sets defined in Lemma \ref{lem:quantization}.
Then, we have:
\begin{enumerate}[label=(\roman*)]
    \item $\log |\mcQ| \leq 2k' \log 2\ee$, and
    \item $\log |\mcL| \leq k' \log \frac{n}{k'}$.
\end{enumerate}
\end{prop}

\begin{proof}
For the first part, we observe that $|\mcQ|$ is not larger than the cardinality of 
\[
    \bigcup_{M = k'}^{2k'}
    \left\{
        \mathbf{m} = (m_1, \ldots, m_{k'}) \in \mathbb{N}^{k'}: 1 \leq m_i \leq k', \sum_{i} m_i = M
    \right\}.
\]
Then, we have
\begin{align*}
    |\mcQ| & 
    \leq \sum_{j = 0}^{k'} \binom{k' + j - 1}{k' - 1}
    = \sum_{j = 0}^{k'} \binom{k' + j - 1}{j}
    \leq \sum_{j = 0}^{k'} \binom{2k' - 1}{j}\\
    & \underset{\text{(a)}}{\leq} \left( \frac{(2k' - 1)\ee}{k'} \right)^{k'} \leq (2\ee)^{k'}. 
\end{align*}
The proof of the inequality (a) in the above can be found in Proposition 4.3 of \citet{Dudley}.

The second part is obtained by Jensen's inequality as 
\[
    \log |\mcL| = \sum_{i=1}^{k'} \log |B_i| \leq k' \log \left( \sum_{i=1}^{k'} \frac{|B_i|}{k'}\right) = k' \log \frac{n}{k'}.
\]
\end{proof}

\noindent {\bfseries\upshape Step 3: Controlling Gaussian widths \ }
As mentioned before, our goal is to obtain an upper bound of the Gaussian width
\begin{equation}\label{eq:def_tilde_i}
    \tilde{W}(\theta) := w(T_{K_-(\mcV)}(\theta) \cap B_n)
    = \EE \left[
        \sup_{v \in T_{K_-(\mcV)}(\theta) \cap B_n} \langle v, Z \rangle
    \right],
\end{equation}
where we convene that $\EE = \EE_{Z \sim N(0, I_n)}$.
Let $(\Pi, w)$ is a pair of a partition and a sign vector of knots defined as in Lemma \ref{lem:gamma_control}.
Using the decomposition in Lemma \ref{lem:quantization}, we have
\[
    \tilde{W}(\theta) \leq \EE\left[
        \max_{\mathbf{t}, \mathbf{q} \in \mcQ, \ \mathbf{l} \in \mcL}
        \sup_{v \in T(\mathbf{t}, \mathbf{q}, \mathbf{l})} \langle v, Z \rangle
    \right].
\]
Besides, leveraging a general result for Gaussian suprema (see Lemma \ref{lem:guntuboyina_d1} below), we have
\begin{equation}\label{eq:tilde_w_decomposition}
    \tilde{W}(\theta) \leq
    \max_{\mathbf{t}, \mathbf{q} \in \mcQ, \ \mathbf{l} \in \mcL} 
    \EE\left[
        \sup_{v \in T(\mathbf{t}, \mathbf{q}, \mathbf{l})} \langle v, Z \rangle
    \right]
    + 3\sqrt{k' \log \frac{\ee n}{k'} }+ \sqrt{\frac{\pi}{2}}.
\end{equation}
Here, we used Proposition \ref{prop:cardinality} to bound the cardinality of the set $\mcQ^2 \times \mcL$.
More precisely, we used the following evaluation:
\[
    2 \log |\mcQ^2 \times \mcL|
    \leq 4k' \log 2\ee + 2k' \log \frac{\ee n}{k'}
    \leq (4 \log 2 \ee + 2) k' \log \frac{\ee n}{k'}
    < 8.8 k' \log \frac{\ee n}{k'}.
\]
%For more detailed derivation, see equation (46) in \citet{Guntuboyina2017}.

Given $\mathbf{t}, \mathbf{q} \in \mcQ$ and $\mathbf{l} \in \mcL$, we define
\[
    \tilde{W}(\mathbf{t}, \mathbf{q}, \mathbf{l}) = \EE \left[
        \sup_{v \in T(\mathbf{t}, \mathbf{q}, \mathbf{l})} \langle v, Z \rangle
    \right].
\]
Dividing the supremum into $k'$ pieces $v_{B_1}, v_{B_2}, \ldots, v_{B_{k'}}$, this quantity is bounded from above as
$\tilde{W}(\mathbf{t}, \mathbf{q}, \mathbf{l})  \leq \sum_{i=1}^{k'} \tilde{W}_i(t_i, q_i, \ell_i)$,
where
\begin{equation}\label{eq:tilde_w_i}
    \tilde{W}_i(t_i, q_i, \ell_i)
    := \EE_{Z_i \sim N(0, I_{|B_i|})} \left[
        \sup_{v_{B_i} \in T_i(t_i, q_i, \ell_i)}
        \langle v_{B_i}, Z_i \rangle
\right].
\end{equation}
Here, we write $T_i(t_i, q_i, \ell_i) := \set{v_{B_i} \in \RR^{B_i}: \ \norm{v_{B_i}}_2^2 \leq t_i, \ \Gamma_i(v, \ell_i) \leq q_i \gamma}$.
%In the next subsection, we will provide upper bounds for $\tilde{W}_i(t_i, q_i, \ell_i)$.
%\subsubsection{Bounding $\tilde{W}_i(t, q, \ell)$}

We now consider the quantity \eqref{eq:tilde_w_i}.
In the set $T_i(t_i, q_i, \ell_i)$ over which the supremum taken, the lower total variation of $v_{B_i}$ is bounded from above as
\begin{equation}\label{eq:sec_tildewi_1}
    \mcV_-^{B_i}(v_{B_i}) \leq w_i(v_{\tau_i} - v_\ell) + w_{i+1}(v_{\ell_i} - v_{\tau_{i+1}-1}) + q_i \gamma.
\end{equation}
As mentioned in Remark \ref{rmk:gamma_control}, the reverse inequality
\[
    \mcV_-^{B_i}(v_{B_i}) \geq w_i(v_{\tau_i} - v_\ell) + w_{i+1}(v_{\ell_i} - v_{\tau_{i+1}-1})
\]
is always true, and the equality can hold only if two sub-vectors $(v_{\tau_i}, v_{\tau_i} + 1, \ldots, \ell_i)$ and $(\ell_i, \ell_i + 1, \ldots, v_{\tau_{i + 1}} - 1)$ are either monotone increasing or non-increasing.
From this point of view, we may consider that the meaning of the condition \eqref{eq:sec_tildewi_1} is that $v_{B_i}$ is approximated by two nearly monotone pieces.
This suggests that the complexity of $T_i(t_i, q_i, \ell_i)$ can be evaluated by that of the class of monotone functions.

Below, we provide the upper bound of the Gaussian width of the form \eqref{eq:tilde_w_i}.
First, the following lemma treats a special case where $\ell_i$ is taken as the rightmost point in $B_i$.
\begin{lem}\label{lem:suprema_1}
For every $n \geq 1$, $t > 0$, $w \in \{ 0, 1 \}$ and $\gamma \geq 0$, we have
\begin{align}\label{eq:lem_suprema_1}
    \EE \bigg[
        \sup \bigg\{ \langle v, Z \rangle \ : & \ v \in \RR^n,
        \norm{v}_2 \leq t, \ \text{and} \nonumber \\
        & \sum_{i=1}^{n-1} (v_i - v_{i+1})_+ \leq w(v_1 - v_n) + \gamma
        \bigg \}
    \bigg ] \leq (t + 2\gamma \sqrt{n-1}) \sqrt{\log(\ee n)}.
\end{align}
\end{lem}

\begin{proof}

The proof is divided into two cases where $w = 1$ and $w = 0$.

\textbf{Case 1 ($w = 1$):}
By scaling properly, we need only consider the case where $t = 1$.
For a vector $v \in \RR^n$, we define a monotone vector $v^+$ as
\[
    v^+_1 = 0 \quad \text{and} \quad
    v^+_i = \sum_{j=2}^i (v_j - v_{j-1})_+ \quad \text{for $i = 2, \ldots, n$}.
\]
We also define another monotone vector $v^-$ as
\[
    v^-_1 = -v_1 \quad \text{and} \quad
    v^-_i = v^-_1 + \sum_{j=2}^i (v_{j-1} - v_j)_+ \quad \text{for $i = 2, \ldots, n$}.
\]
It is easy to check that $v = v^+ - v^-$.
Using these notations, we have
\[
    \mcV_-(v) = \sum_{i=1}^{n-1} (v_i - v_{i+1})_+ = v^-_n - v^-_1.
\]
Hence, the condition $\mcV_-(v) \leq v_1 - v_n + \gamma$ is equivalent to $v^+_n \leq \gamma$, which leads to
\[
    \norm{v^+}_2^2 \leq (n - 1) (v^+_n)^2 \leq (n-1) \gamma^2
\]
and
\[
    \norm{v_-}_2 \leq \lVert v \rVert_2 + \lVert v^+ \rVert_2 \leq 1 + \gamma \sqrt{n-1}.
\]

Denote by $\tilde{W}$ the left-hand side in \eqref{eq:lem_suprema_1} with $t = 1$.
The argument in the previous paragraph implies that
\begin{align}
    \tilde{W} & \leq
    \EE \left[
        \sup_{v^+ \in K_n^\uparrow: \ \lVert v^+ \rVert_2 \leq \gamma\sqrt{n-1}}
        \langle v^+, Z \rangle
    \right]
    + \EE \left[
        \sup_{v^- \in K_n^\uparrow: \ \lVert v^- \rVert_2 \leq 1 + \gamma\sqrt{n-1}}
        \langle v^-, Z \rangle
    \right] \nonumber \\
    & \leq (1 + 2 \gamma \sqrt{n-1}) \cdot \EE \left[
        \sup_{v \in K_n^\uparrow: \ \lVert v \rVert_2 \leq 1}
        \langle v, Z \rangle
    \right].
\label{eq:eq:lem_suprema_1_pf}
\end{align}
The expectation in the last line is bounded as
\[
    \left(
        \EE \left[
            \sup_{v \in K_n^\uparrow: \ \lVert v \rVert_2 \leq 1}
            \langle v, Z \rangle
        \right]
    \right)^2
    \leq
    \EE \left[ \left(
        \sup_{v \in K_n^\uparrow: \ \lVert v \rVert_2 \leq 1}
        \langle v, Z \rangle
    \right)^2 \right]
    \leq \log (\ee n).
\]
Here, the first inequality is the Jensen's inequality, and the second inequality is a consequence of equation (D.12) in \citet{Amelunxen2014}.
Combining with \eqref{eq:eq:lem_suprema_1_pf}, we have the desired result.

\textbf{Case 2 ($w = 0$):}
We can assume w.l.o.g. $t = 1$.
As in Case 1, and we write a vector as a difference of monotone vectors.
For $v \in \RR^n$, we define $v^+$ and $v^-$ as
\[
    v^+_1 = v_1 \quad \text{and} \quad
    v^+_i = \sum_{j=2}^i (v_j - v_{j-1})_+ \quad \text{for $i = 2, \ldots, n$}.
\]
and
\[
    v^-_1 = 0 \quad \text{and} \quad
    v^-_i = v^-_1 + \sum_{j=2}^i (v_{j-1} - v_j)_+ \quad \text{for $i = 2, \ldots, n$},
\]
respectively.
Under this notation, the condition $\mcV_-(v) \leq \gamma$ is equivalent to $v^-_n \leq \gamma$, and therefore we have
\[
    \norm{v^+}_2 \leq 1 + \gamma \sqrt{n-1} \quad \text{and} \quad
    \norm{v^-}_2 \leq \gamma \sqrt{n-1}.
\]
Then, a similar argument as Case 1 yields the result.
\end{proof}

Next, the following lemma provides an upper bound of $\tilde{W}_i$ for general choices of $\ell_i \in B_i$.

\begin{lem}\label{lem:suprema_2}
Fix $n \geq1$, $1 \leq \ell \leq n$, $t > 0$ and $\gamma \geq 0$.
For every $w_1, w_2 \in \{ 0, 1\}$, the quantity
\begin{align*}
    \tilde{W} := \EE 
    \bigg[
        \sup \bigg\{ \langle v, Z \rangle \ : & \ v \in \RR^n,
            \norm{v}_2 \leq t, \ \text{and} \nonumber \\
            & \mcV_-(v) \leq w_1(v_1 - v_\ell) + w_2(v_\ell - v_n) + \gamma
        \bigg \}
    \bigg ]
\end{align*}
is bounded from above as
\begin{equation}\label{eq:lem_suprema_2a}
    \tilde{W} \leq
    \left\{
        \begin{aligned}
            & (t + 2\gamma\sqrt{\ell - 1}) \sqrt{\log (\ee \ell)} +
            (t + 2\gamma \sqrt{n - \ell}) \sqrt{\log (\ee (n - \ell + 1))} \quad & \text{if $1 < \ell < n$} \\
            & (t + 2\gamma \sqrt{n - 1})\sqrt{\log (\ee n)} \quad & \text{if $\ell = 1$ or $n$}.
        \end{aligned}
    \right.
\end{equation}
In particular, we deduce a simpler bound
\begin{equation}\label{eq:lem_suprema_2b}
    \tilde{W} \leq 2 (t + 2 \gamma \sqrt{n - 1})\sqrt{\log (\ee n)}.
\end{equation}
\end{lem}

\begin{proof}
Let $(A_1, A_2)$ be a pair of sub-vectors of $[n]$ defined as $A_1 = \{ 1, 2, \ldots, \ell \}$ and $A_2 = \{ \ell, \ell + 1, \ldots, n \}$.
If either $\ell = 1$ or $\ell = n$ (i.e., one of $A_1$ and $A_2$ becomes a singleton), the result is a direct consequence of Lemma \ref{lem:suprema_1}.

Henceforth, we assume that $1 < \ell < n$.
Suppose that $v \in \RR^n$ satisfies the assumption $\mcV_-(v) \leq w_1(v_1 - v_\ell) + w_2(v_\ell - v_n) + \gamma$.
Since $\mcV_-(v) \geq \mcV_-^{A_1}(v_{A_1}) + w_2 (v_\ell - v_n)$, we have
\[
    \mcV_-^{A_1}(v_{A_1}) \leq w_1(v_1 - v_\ell) +\gamma.
\]
Similarly, we have
\[
    \mcV_-^{A_2}(v_{A_2}) \leq
    \mcV_-(v) - w_1(v_1 - v_\ell)
    \leq w_2 (v_{\ell} - v_n) + \gamma.
\]
Based on these observations, we reduce to
\[
    \tilde{W} \leq
    \EE \left[
        \sup_{\substack{v_{A_1} \in \RR^\ell:
        \norm{v_{A_1}}_2 \leq t,\\
        \mcV_-^{A_1}(v_{A_1}) \leq w_1 (v_1 - v_\ell) + \gamma}} \langle v_{A_1}, Z_{A_1} \rangle
    \right]
    + \EE \left[
        \sup_{\substack{v_{A_2} \in \RR^{n - \ell + 1}:
        \norm{v_{A_2}}_2 \leq t,\\
        \mcV_-^{A_2}(v_{A_2}) \leq w_2 (v_{\ell} - v_n) + \gamma}} \langle v_{A_2}, Z_{A_2} \rangle
    \right],
\]
in which both terms in the right-hand side can be bounded using Lemma \ref{lem:suprema_1}.
\end{proof}

Before going to the next step, we summarize the results in Step 3 as follows.

\begin{prop}\label{prop:suprema_summary}
Fix $\theta \in \RR^n$.
Let $\Pi = (B_1, B_2, \ldots, B_{k'})$ be any connected refinement of $\Pi_\mathrm{const}(\theta)$, and $w_1, w_2, \ldots, w_{k'}$ be the signs associated with $\Pi$ as in Lemma \ref{lem:gamma_control}.
Define $\gamma \geq 0$ as \eqref{eq:quantize_delta}.
Then, the quantity $\tilde{W}(\theta)$ defined in \eqref{eq:tilde_w_i} is bounded from above by

\begin{equation}\label{eq:prop_suprema_summary}
    \tilde{W}(\theta) \leq
    \max_{\mathbf{t}, \mathbf{q} \in \mcQ} \left\{
        \sum_{i=1}^{k'}
        2(\sqrt{t_i} + 2 q_i \gamma \sqrt{|B_i| - 1}) \sqrt{\log(\ee |B_i|)}
        + 3\sqrt{k' \log \frac{\ee n}{k'} }
        + \sqrt{\frac{\pi}{2}}
    \right\}.
\end{equation}
\end{prop}
\begin{proof}
This is a direct consequence of \eqref{eq:tilde_w_decomposition} and \eqref{eq:lem_suprema_2b}.
\end{proof}

\noindent {\bfseries\upshape Step 4: Applying Lemma \ref{lem:projection_general} \ }
We now are ready to complete the proof of Theorem \ref{thm:constrained_tangent}.

Recall that our goal is to obtain an upper bound for $\tilde{W}(\theta)$ which is defined in \eqref{eq:tilde_w_i}.
To this end, we will construct a suitable refinement of $\Pi_\mathrm{const}(\theta)$ with moderate piece lengths so that we can control the first term in \eqref{eq:prop_suprema_summary}.
In fact, from an argument parallel to that in \citet{Guntuboyina2017a}, there exists a refinement $\Pi = (B_1, B_2, \ldots, B_{k'})$ such that
\[
    |B_i|
    \leq  \frac{4n}{k'} \quad \text{for $i = 1, 2, \ldots, k'$}
\] 
and $k(\theta) \leq k' \leq 2k(\theta)$.
We also define the signs $w_1, w_2, \ldots, w_{k'}$ in a similar way as Lemma \ref{lem:tangent_characterization}, but if the knot $\tau_i$ is not contained in the original partition $\Pi_\mathrm{const}(\theta)$, the corresponding sign $w_i$ will be specified later.
%Hereafter, we fix such a partition $\Pi$.

We can bound the first term in \eqref{eq:prop_suprema_summary} as the following two steps.
First, from the Cauchy--Schwarz inequality and the fact that $\mathbf{t} \in \mcQ$, we have
\begin{align*}
    \sum_{i=1}^{k'} \sqrt{t_i} \sqrt{\log(\ee |B_i|)}
    & \leq \left(
        \sum_{i=1}^{k'} t_i
    \right)^{1/2} \left(
        \sum_{i=1}^{k'} \log(\ee |B_i|)
    \right)^{1/2} \\
    & \leq \sqrt{2} \sqrt{k' \log \frac{\ee n}{k'}}
    \leq 2 \sqrt{k(\theta) \log \frac{\ee n}{k(\theta)}}.
\end{align*}
Second, by the above construction of $\Pi$, we have
\begin{align*}
    \sum_{i=1}^{k'} q_i \gamma \sqrt{|B_i| - 1} \sqrt{\log (\ee |B_i|)}
    & \leq \max_{1 \leq i \leq k'} \left[
        \sqrt{|B_i| \log (\ee |B_i|)}
    \right] \sum_{i=1}^{k'} q_i \gamma\\
    & \leq 2 \gamma \cdot 2(1 + \log 4)\sqrt{\frac{n}{k'}\log\frac{\ee n}{k'}}\\
    & \leq 10 \gamma \sqrt{\frac{n}{k(\theta)}\log\frac{\ee n}{k(\theta)}}.
\end{align*}
Therefore, the right-hand side in \eqref{eq:prop_suprema_summary} can be bounded from above by
\begin{equation}\label{eq:proof_main_1}
    10 \sqrt{k(\theta) \log \frac{\ee n}{k(\theta)}}
    + 20 \gamma \sqrt{\frac{n}{k(\theta)}\log\frac{\ee n}{k(\theta)}}.
\end{equation}
Here, to hide the constant term $\sqrt{\pi / 2}$, we have also used the fact that $\sqrt{m \log (\ee n/m)} \geq 1$ for every integer $1 \leq m \leq n$.

Let $w^0_1, w^0_2, \ldots, w^0_{k(\theta) + 1}$ be the signs associated with the constant partition $\Pi_\mathrm{const}(\theta) = (A_1, A_2, \ldots, A_{k(\theta)})$ (recall the definition \eqref{eq:sign_definition}).
Then, we can choose the values of $w_i$ so that the following inequality holds:
\begin{align}\label{eq:proof_main_2}
    \gamma^2
    & = \sum_{i=1}^{k'} |B_i|^{-1} 1_{\set{w_i \neq w_{i+1}}}
    \leq \sum_{j=1}^{k(\theta)} \left[
        \min \left\{ |A_j|,\ \left \lfloor \frac{2n}{k(\theta)} \right \rfloor \right\}
    \right]^{-1} 1_{\set{ w_j^0 \neq w_{j+1}^0 }} \nonumber \\
    & \leq \sum_{i=1}^{k(\theta)}
    \left [
        \min \left\{ |A_i|, \ \frac{n}{k(\theta)} \right\}
    \right]^{-1} 1_{\set{ w_i^0 \neq w_{i+1}^0}} \nonumber \\
    & = M(\theta).
\end{align}
In fact, this is possible if we choose $w_i$ as the sign $w_j^0$ for the nearest knot that is to the right of $\tau_i$.
Combining \eqref{eq:proof_main_2}, \eqref{eq:proof_main_1} and Proposition \ref{prop:statistical_dim_compare}, the statistical dimension of $T_{K_-(\mcV)}(\theta)$ is bounded from above as
\[
    \delta(T_{K_-(\mcV)}(\theta))
    \leq \tilde{W}^2(\theta) + 1
    \leq 800 n \left[
        \frac{k(\theta)}{n} \log \frac{\ee n}{k(\theta)}
        + \frac{M(\theta)}{k(\theta)} \log \frac{\ee n}{k(\theta)}
    \right] + 1,
\]
where we also used the elementary fact that $(a + b)^2 \leq 2 (a^2 + b^2)$.
Consequently, applying Lemma \ref{lem:projection_general}, we have desired result.

\begin{rmk}[Non-Gaussian noises]\label{rmk:non_gauss}
For non-Gaussian noise setting, we could prove an analogous result to Proposition \ref{prop:statistical_dimension_bound}. We comment on a sketch of the proof for such a generalization.

The proof of Proposition \ref{prop:statistical_dimension_bound} consists of (i) a decomposition argument for the tangent cone and (ii) bounds for some probabilistic quantities (i.e., the statistical dimension and the Gaussian width). The former argument is completely deterministic and independent from the distributional assumption on the noise variables. Regarding the probabilistic bounds, we used the following bound for (Gaussian) statistical dimension of $K_n^\uparrow$:
\[
    \delta(K_n^\uparrow) \leq \log (\ee n).
\]
Hence, if we can obtain a similar bound for non-Gaussian random variables, we can prove a analogous result to Proposition \ref{prop:statistical_dimension_bound}.

Let $\xi_1, \ldots, x_n$ be i.i.d.~random variables with $\EE[\xi_1] = 0$ and $\mathrm{Var}(\xi_1) = \sigma^2$. For a convex cone $C$, we define the statistical dimension as
\[
    \bar{\delta}(C) =
    \frac{1}{\sigma^2} \EE \left[
        \left(
            \sup_{\theta \in C: \norm{\theta}_2 \leq 1}
            \langle \xi, \theta \rangle
        \right)^2
    \right]
    = \frac{1}{\sigma^2} \EE \norm{\mathrm{Proj}_C(\xi)}^2.
\]
Here, we write $\mathrm{Proj}_C(x) = \argmin_{z \in C} \norm{z - x}_2$, and the last equality holds from a deterministic relation
\[
    \left(
        \sup_{\theta \in C: \norm{\theta}_2 \leq 1}
        \langle \xi, \theta \rangle
    \right)^2
    = \norm{\mathrm{Proj}_C(\xi)}^2.
\]
(See \citet{Amelunxen2014} for details). Then, from Theorem 3.1 in \citet{Chatterjee2015}, we can check that
\[
    \bar{\delta}(K_n^\uparrow) \leq 16 \log (\ee n).
\]
Therefore, by following a similar argument as the proof of Proposition \ref{prop:statistical_dimension_bound}, we conclude that
\[
    \bar{\delta}(T_{K_-(\mcV)})
    \leq C' n \left\{
        \frac{k(\theta)}{n} \log \frac{\ee n}{k(\theta)}
        + \frac{M(\theta)}{k(\theta)} \log \frac{\ee n}{k(\theta)}
    \right\}
\]
for some universal constant $C' > 0$. As a consequence, we can prove the expected risk bound similar to \eqref{eq:thm_penalized_1} for non-Gaussian noise variables.
\end{rmk}

\subsection{Proof of Corollary \ref{cor:min_length}}

Let $\alpha > 0$ be a number to be specified later.
Define a vector $\theta' \in \RR^n$ as $\theta'_1 = \theta^*_1$ and
\begin{align*}
    \theta'_i & =
    \theta^*_1
    + \sum_{j = 1}^{i - 1} (\theta^*_{j + 1} - \theta^*_{j})_+
    - \alpha \sum_{j = 1}^{i - 1} (\theta^*_j - \theta^*_{j + 1})_+    
    \quad \text{for $i = 2, 3, \ldots, n$}.
\end{align*}
Then, we have $\mcV_-(\theta') = \alpha \mcV_-(\theta^*)$.
Moreover, the constant partition and the sign of $\theta'$ (defined in \eqref{eq:sign_definition}) are the same as those of $\theta^*$, and therefore $k(\theta') = k(\theta^*)$ and $M(\theta') = M(\theta^*)$.

Now, we set $\alpha = \mcV / \mcV_-(\theta^*)$ so that $\mcV_-(\theta') = \mcV$.
Applying the upper bound \eqref{eq:thm_constrained_tangent_1}, we have
\[
    \frac{1}{n} \EE_{\theta^*} \norm{\hat{\theta}_\mcV - \theta^*}_2^2
    \leq \frac{1}{n} \norm{\theta' - \theta^*}_2^2
    + C \sigma^2 \frac{k(\theta^*)}{n} \log \frac{\ee n}{k(\theta^*)}
    + C \sigma^2 \frac{M(\theta^*)}{k(\theta^*)} \log \frac{\ee n}{k(\theta^*)}.
\]
The first term in the right-hand side is bounded from above as
\begin{align*}
    \frac{1}{n} \norm{\theta' - \theta^*}_2^2
    = \frac{(1 - \alpha)^2}{n} \sum_{i=2}^{n} \left(
        \sum_{j = 1}^{i - 1} (\theta^*_j - \theta^*_{j + 1})_+
    \right)^2
    \leq (1 - \alpha)^2 (\mcV_-(\theta^*))^2
    = (\mcV - \mcV_-(\theta^*))^2.
\end{align*}
From the minimal length condition \eqref{eq:min_length} and the definition of $M(\theta)$, we also have
\[
    \frac{M(\theta^*)}{k(\theta^*)} \log \frac{\ee n}{k(\theta^*)}
    \leq \frac{2 c^{-1} (m(\theta^*) - 1)}{n} \log \frac{\ee n}{k(\theta^*)}. 
\]
Combining the above inequalities, we have the desired result.

\subsection{Risk bounds for penalized estimators (Proof of Theorem \ref{thm:penalized})}\label{sec:penalized_risk_proof}

We prove Theorem \ref{thm:penalized} as an application of Lemma \ref{lem:proximal_general}.
Let $\partial \mcV_-(\theta)$ denote the set of subgradients (i.e., subdifferential) of the convex function $\mcV_-(\cdot)$ at $\theta \in \RR^n$.
The task is to provide a suitable upper bound for the Gaussian mean squared distance of the set $\lambda \partial \mcV_-(\theta)$.
To do this, we use the technique developed in \citet{Guntuboyina2017a}.
The idea is stated roughly as follows:
Recall that the Gaussian mean squared distance of a convex cone can be written as the statistical dimension of the polar cone (Proposition \ref{prop:statistical_dim_compare}-(ii)).
This motivates us to relate the Gaussian mean squared distance $\mbD(\lambda \partial \mcV_-(\theta))$ to that of an associated cone.
In particular, we consider the conic hull of the subdifferential:
\[
    \cone(\partial \mcV_-(\theta))
    := \bigcup_{\lambda \geq 0} \lambda \partial \mcV_-(\theta).
\]
As we explain later, $\mbD(\cone(\partial \mcV_-(\theta)))$ can be evaluated by the results in the previous subsection.
Then, we can complete the proof if we have an upper bound of the following form:
\begin{equation}\label{eq:conic_hull_goal}
    \mbD(\lambda \partial \mcV_-(\theta))
    \leq \mbD(\cone (\partial \mcV_-(\theta))) + \Delta(\theta, \lambda),
\end{equation}
where $\Delta(\theta, \lambda)$ is a residual term that depends on $\theta$ and $\lambda$.

First, we show that $\mbD(\cone (\partial \mcV_-(\theta)))$ has exactly the same value as the statistical dimension of the tangent cone of $T_{K_-(\mcV_-(\theta))}(\theta)$, which we have already provided a bound in the previous part in this paper.

\begin{prop}\label{prop:conic_polar}
For any $\theta \in \RR^n$, the following equality holds:
\[
    \mbD(\cone(\partial \mcV_-(\theta))) = \delta(T_{K_-(\mcV(\theta))}(\theta)).
\]
In particular, we have the following upper bound:
\[
    \mbD(\cone(\partial \mcV_-(\theta)))
    \leq C n \left\{
        \frac{k(\theta)}{n} \log \frac{\ee n}{k(\theta)}
        + \frac{M(\theta)}{k(\theta)} \log \frac{\ee n}{k(\theta)}
    \right\},
\]
where $C$ is the same universal constant as in Proposition \ref{prop:statistical_dimension_bound}.
\end{prop}

\begin{proof}
Let us write $T := T_{K_-(\mcV(\theta))}(\theta)$.
In the light of Proposition \ref{prop:statistical_dim_compare}-(ii), it suffices to show that $T$ is the polar cone of $\cone(\partial \mcV_-(\theta))$.
However, from fundamental results in convex geometry, we always have
\[
    \cone(\partial f(\theta)) = \left(
        T_{K(\theta)}(\theta)
    \right)^\circ
    \quad \text{with} \quad
    K(\theta) := \set{z \in \RR^n: f(z) \leq f(\theta)}
\]
for any convex function $f: \RR^n \to \RR$ (see Lemma A.5 and Lemma A.5 in \citet{Guntuboyina2017a}).
For the case where $f = \mcV_-$, the set $K(\theta)$ above is
\[
    K_-(\mcV_-(\theta)) = \set{z \in \RR^n: \mcV_-(z) \leq \mcV_-(\theta)},
\]
which implies the desired result.
\end{proof}

Next, we provide an inequality of the form \eqref{eq:conic_hull_goal}.
Since $\cone(\partial \mcV_-(\theta)) \supseteq \lambda \partial \mcV_-(\theta)$ holds for every $\lambda \geq 0$, the definition of the Gaussian mean squared distance (Definition \ref{defi:gaussian_width}-(ii)) suggests that $\mbD(\cone(\partial \mcV_-(\theta))) \leq \mbD(\lambda \partial \mcV_-(\theta))$.
However, we need a reverse inequality \eqref{eq:conic_hull_goal}.
To this end, we use the following result proved by \citet{Guntuboyina2017a}.

\begin{lem}[\citet{Guntuboyina2017a}, Proposition B.5]\label{lem:guntuboyina_b5}
Let $f: \RR^n \to \RR$ be a convex function, and $\theta \in \RR^n$.
Define a vector $v_0$ as
\begin{equation}\label{eq:guntuboyina_b5_def_v0}
    v_0 := \argmin_{v \in \aff(\partial f(\theta))} \norm{v}_2,
\end{equation}
where $\aff(C)$ is the affine hull of the set $C \subseteq \RR^n$.
Suppose that $v_0 \neq 0$.
For any $z \in \RR^n$, define $\lambda(z) \geq 0$ as
\[
    \lambda(z) := \argmin_{\lambda \geq 0} \dist(z, \lambda \partial f(\theta)).
\]
Then, $\lambda(z)$ is well-defined, and has a finite expectation $\EE_{Z \sim N(0, I_n)}[\lambda(Z)] < \infty$.

Further, define $\lambda^*$ as
\[
    \lambda^* := \lambda^*(\theta)
    = \EE_{Z \sim N(0, I_n)}[\lambda(Z)] + \frac{2}{\norm{v_0}_2}.
\]
Then, for every $\lambda \geq \lambda^*$ and $v^* \in \partial f(\theta)$, we have
\begin{equation}\label{eq:lem_guntuboyina_b5}
    \mbD(\lambda \partial f(\theta))
    \leq 4 + \left(
        \sqrt{\mbD(\cone(\partial f(\theta)))} + \frac{4 \norm{v^*}_2}{\norm{v_0}_2}
        + 2 + (\lambda - \lambda^*) \norm{v^*}_2
    \right)^2.
\end{equation}
\end{lem}

Before proceeding, we introduce an additional terminology:
A convex function $f: \RR^n \to \RR$ is said to be \textit{weakly decomposable} if we have
\begin{equation}\label{eq:defi_weak_decomposable}
    \argmin_{v \in \aff(\partial f(\theta))} \norm{v}_2 \in \partial f(\theta)
\end{equation}
for every $\theta \in \RR^n$.
In other words, we can choose $v_0 \equiv v^*$ in \eqref{eq:lem_guntuboyina_b5} if $f$ is weakly decomposable.
%This notion was originaly introduced in the signal processing literature \citep{Foygel2014}.
Under the assumption that $f$ is weakly decomposable, the inequality \eqref{eq:lem_guntuboyina_b5} can be simplified as follows:

\begin{cor}\label{cor:lem_guntuboyina_b5}
Suppose that $f: \RR^n \to \RR$ is convex and weakly decomposable.
Under the same notation as in Lemma \ref{lem:guntuboyina_b5}, we have
\[
    \mbD(\lambda \partial f(\theta))
    \leq 3 \mbD(\cone(\partial f(\theta)))
    + 3 (\lambda - \lambda^*)^2 \norm{v_0}_2^2 + 112.
\]
\end{cor}

Now, we apply Lemma \ref{lem:guntuboyina_b5} to the case $f = \mcV_-$.
The following proposition provides the structural information of $\partial \mcV_-(\theta)$ that we need for evaluating the upper bound \eqref{eq:lem_guntuboyina_b5}.
The proof is postponed to Appendix \ref{sec:weak_decomp}.

\begin{prop}\label{prop:mcv_weak_decomp}
\begin{enumerate}[label=(\roman*)]
    \item $\theta \mapsto \mcV_-(\theta)$ is weakly decomposable.
    \item For any $\theta \in \RR^n$, let us define $v_0$ as \eqref{eq:guntuboyina_b5_def_v0}. Then, we have
    \begin{equation}\label{eq:mcv_v0_norm}
        \norm{v_0}_2^2 = \sum_{i=1}^k \frac{1}{|A_i|} 1_{w_{i} \neq w_{i + 1}}.
    \end{equation}
\end{enumerate}
\end{prop}

From Proposition \ref{prop:mcv_weak_decomp} and Corollary \ref{cor:lem_guntuboyina_b5}, $\mbD(\lambda \partial \mcV_-(\theta))$ is bounded from above by
\[
    C' n \left\{
        \frac{k(\theta)}{n} \log \frac{\ee n}{k(\theta)}
        + \frac{M(\theta)}{k(\theta)} \log \frac{\ee n}{k(\theta)}
    \right\}
    + C' (\lambda - \lambda^*)^2 \sum_{i=1}^k \frac{1}{|A_i|} 1_{w_{i} \neq w_{i + 1}}
\]
provided that $\lambda \geq \lambda^*$.
Here, $C' > 0$ is a universal constant.
Combining this bound with Lemma \ref{lem:proximal_general}, we proved the desired risk bound.

Lastly, we provide an upper bound for the optimal tuning parameter $\lambda^*$.
This is obtained from the following estimate of $\EE[\lambda(Z)]$.
% lambda^*
\begin{prop}
%Let $f: \RR^n \to \RR$ be a support function of a closed convex set $C \subseteq \RR^n$, that is, $f(\theta)$ is defined as $f(\theta) := \argmax_{z \in C} \langle z, \theta \rangle$ for any $\theta \in \RR^n$.
%Suppose that $\theta \in \RR^n$ satisfies $\mcV_-(\theta) \neq 0$.
Suppose that $\theta \in \RR^n$ and $\mcV_-(\theta) > 0$.
For any $z \in \RR^n$, define $\lambda(z)$ as
\[
    \lambda(z) := \argmin_{\lambda \geq 0} \dist(z, \lambda \partial \mcV_-(\theta)).
\]
Then, we have
\[
    \EE[\lambda(Z)]
    \leq \min \left\{
        \frac{\norm{\theta}_2}{\mcV_-(\theta)}, \
        \left(
            \sum_{i = 1}^k \frac{1_{\set{w_i \neq w_{i + 1}}}}{|A_i|}
        \right)^{-1/2}
    \right\}
        [\delta(T_{K_-(\mcV_-(\theta))}(\theta))]^{1/2},
\]
where $\EE$ is the expectation with respect to $Z \sim N(0, I_n)$.
\end{prop}

\begin{proof}
Let $C := \cone(\partial \mcV_-(\theta))$ be the conic hull of $\partial \mcV_-(\theta)$, and let $P_{C}$ denote the orthogonal projection map onto $C$.
By the definition of $\lambda(z)$, there exists a vector $v(z) \in \partial \mcV_-(\theta)$ such that $\lambda(z) v(z) = P_C(z)$.

First, we show a partial result
\[
    \EE[\lambda(Z)]
    \leq \frac{\norm{\theta}_2}{\mcV_-(\theta)} \sqrt{\delta(T_{K_-(\mcV_-(\theta))}(\theta))}.
\]
As we will see in Appendix \ref{sec:weak_decomp}, $\mcV_-$ is the support function for a certain convex set.
Then, by the fundamental fact for the support function that $\langle \theta, v \rangle = \mcV_-(\theta)$ for all $v \in \partial \mcV_-(\theta)$ (see Corollary 8.25 in \citet{RW}), we have
\begin{align*}
    \lambda(z) \mcV_-(\theta)
    & = \lambda(z) \langle \theta, v(z) \rangle
    \quad (\because v(z) \in \partial \mcV_-(\theta))\\
    & = \langle \theta, P_C(z) \rangle
    \quad (\because \lambda(z) v(z) = P_C(z))\\
    & = \langle \theta, z - P_{T}(z) \rangle.
\end{align*}
Here, in the last line, $T := T_{K_-(\mcV_-(\theta))}(\theta)$ is the polar cone of $C$ (see Proposition \ref{prop:conic_polar}), and we used the Moreau decomposition $z = P_C(z) + P_T(z)$.
Taking the expectation of both sides with respect to $z \sim N(0, I_n)$, we have
\begin{align*}
    \mcV_-(\theta) \EE [\lambda(z)]
    & = \underbrace{\EE [\langle \theta, z \rangle]}_{=0} - \EE [\langle \theta, P_{T}(z) \rangle] \\
    & \leq \norm{\theta}_2 \EE \norm{P_{T}(z)}_2 \\
    & \leq \norm{\theta}_2 (\EE \norm{P_T(z)}_2^2)^{1/2} \\
    & = \norm{\theta}_2 (\delta(T))^{1/2},
\end{align*}
which implies the desired result.
Here, we used the equality between the statistical dimension and the expected squared norm of projection: $\delta(T) = \EE_{Z \sim N(0, I_n)} \norm{P_T(Z)}_2^2$ (see Proposition 3.1 in \citet{Amelunxen2014}).

To prove the other inequality, we use the characterization of $\aff(\partial \mcV_-(\theta))$ given in \eqref{eq:lovasz_affine_hull} in Appendix \ref{sec:weak_decomp} below.
In particular, if we take $v^*$ as in \eqref{eq:mcv_v_star}, we have
\[
    \langle \lambda(z) v(z), v^* \rangle
    = \langle v^*, P_C(z) \rangle
    \leq \norm{v^*}_2 (\delta(T))^{1/2},
\]
and
\[
      \langle v(z), v^* \rangle = \norm{v^*}_2^2
      = \sum_{i = 1}^k \frac{1_{\set{w_i \neq w_{i + 1}}}}{|A_i|},
\]
and hence the result follows.
\end{proof}

\subsection{Proof of Corollary \ref{cor:moderate}}\label{sec:cor_moderate_proof}

First, we explain that a monotone vector satisfying the moderate growth condition is approximated by a piecewise-constant vector such that the segments at both ends have sufficient lengths.
To this end, we need the following lemma.
Here, the first two statements (i) and (ii) are shown in Lemma 2 in \citet{Bellec2015a}. The third statement (iii) ensures that the moderate growth conditions implies the minimal length condition \eqref{eq:min_length}.

\begin{lem}\label{lem:pc_approximation}
Let $\theta \in K^\uparrow_n$ be a monotone vector satisfying the moderate growth condition and $\theta_n - \theta_1 = \mcV$. Then, there exists another monotone vector $\theta^\prime \in K^\uparrow_n$ satisfying the following three conditions.
\begin{enumerate}[label=(\roman*)]
    \item $\theta^\prime$ is $k$-piecewise constant with
    \begin{equation}\label{eq:lem_pc_approximation_k}
        k = \max \left\{
            3, \
            \left \lceil
                \left(
                    \frac{\mcV^2 n}{\sigma^2 \log (\ee n)}
                \right)^{1/3}
            \right \rceil
        \right\}.
    \end{equation}
    Here, $\lceil t \rceil$ is the smallest integer that is not less than $t$.
    \item We have
    \begin{equation}
        \frac{1}{n} \norm{\theta - \theta^\prime}_2^2
        \leq \frac{1}{4} \max\left\{
            \left( \frac{\sigma^2 \mcV \log (\ee n)}{n} \right)^{2/3}, \
            \frac{3 \sigma^2 \log (\ee n)}{n}
        \right\}
    \end{equation}
    and
    \begin{equation}
        \frac{\sigma^2 k}{n} \log \frac{\ee n}{k}
        \leq 2 \max\left\{
            \left( \frac{\sigma^2 \mcV \log (\ee n)}{n} \right)^{2/3}, \
            \frac{3 \sigma^2 \log (\ee n)}{n}
        \right\}.
    \end{equation}
    \item Let $\Pi^\prime = \set{ A_1, A_2, \ldots, A_k }$ be the partition on which $\theta^\prime$ is constant. Then, we have $|A_1| \geq n / k$ and $|A_k| \geq n / k$.
\end{enumerate}
\end{lem}

\begin{proof}
Let $k$ be an integer defined in \eqref{eq:lem_pc_approximation_k}.
We construct a $k$-piecewise constant monotone vector $\theta' \in K_n^\uparrow$ as follows:
First, define an equi-spaced partition $I_1, I_2, \ldots, I_k$ of the interval $[\theta_1, \theta_n]$ as
\[
    I_j := \left[\theta_1 + \frac{j - 1}{k}\mcV, \ \theta_1 + \frac{j}{k}\mcV \right)
    \quad \text{for $j = 1, 2, \ldots, k-1$},
\]
and $I_k := [\theta_1 + \frac{k - 1}{k} \mcV, \theta_n]$.
Next, define a partition $\Pi = (A_1, A_2, \ldots, A_k)$ of $[n]$ as $A_j := \set{i \in [n]: \theta_i \in I_j}$ ($j = 1, 2, \ldots, k$).
Then, let $\theta'$ be a piecewise-constant vector such that $\theta'_i := \theta_1 + \frac{j - 1/2}{k}\mcV$ for $i \in A_j$.
See the right panel of Figure \ref{fig:moderate_growth} for an illustrative example for $\theta$ and its piecewise-constant approximation $\theta'$.
By a similar argument as Lemma 2 in \citet{Bellec2015a}, we can check (i) and (ii).

It remains to prove (iii) under the moderate growth condition.
Below, we will only check that the maximal element in $A_1$ is not less than $n / k$ because $|A_k| \geq n / k$ can be checked in a similar way.
Let $i^* := \lceil n / k \rceil$.
Note that we have $i^* \leq \lceil n / 2 \rceil$ since $k \geq 3$.
By the moderate growth condition, we have
\[
    \theta_{i^*} \leq \theta_1 + \frac{n / k - 1}{n - 1} \mcV \leq \theta_1 + \frac{\mcV}{k},
\]
which means $i^* \in A_1$ and hence $|A_1| \geq \lceil n / k \rceil$.
\end{proof}

Now, we are ready to prove Corollary \ref{cor:moderate}.
Applying Lemma \ref{lem:pc_approximation} for every segments $A_1, A_2, \ldots, A_m$, we have a $k$-piecewise constant and $m$-piecewise monotone vector $\theta' \in \RR^n$ such that
\[
    \frac{1}{n} \norm{\theta^* - \theta'}_2^2
    \leq \frac{1}{4} \max \left\{
        \left(
            \frac{\sigma^2 \mcV \log \frac{\ee n}{m}}{n}
        \right)^{2/3}, \
        \frac{3 m \sigma^2}{n} \log \frac{\ee n}{m}
    \right\}
\]
and
\[
    \frac{\sigma^2 k}{n} \log \frac{\ee n}{k}
    \leq 2 \max \left\{
            \left( \frac{\sigma^2 \mcV \log \frac{\ee n}{m}}{n} \right)^{2/3}, \
            \frac{3 m \sigma^2}{n} \log \frac{\ee n}{3 m}
        \right\}.
\]
Moreover, $\theta'$ satisfies the minimum length condition \eqref{eq:min_length} with $c = 1$. Therefore, we have $M(\theta') \leq 2(m - 1) k / n$ and
\[
    \frac{\sigma^2 M(\theta')}{k} \log \frac{\ee n}{k}
    \leq \frac{2 (m - 1) \sigma^2}{n} \log \frac{\ee n}{m},
\]
where we used an obvious inequality $m \leq k$.
Then, Theorem \ref{thm:penalized} implies that there exists $\lambda$ such that
\begin{align*}
    \frac{1}{n} \EE_{\theta^*} \norm{\hat{\theta}_\lambda - \theta^*}_2^2
    & \leq \frac{1}{n} \norm{\theta^* - \theta'}_2^2
    + C\frac{\sigma^2 k}{n} \log \frac{\ee n}{k}
    + C\frac{\sigma^2 M(\theta')}{k} \log \frac{\ee n}{k} \\
    & \leq C' \max \left\{
        \left(
            \frac{\sigma^2 \mcV \log \frac{\ee n}{m}}{n}
        \right)^{2/3}, \
        \frac{m \sigma^2}{n} \log \frac{\ee n}{m}
    \right\}
\end{align*}
for some universal constant $C' > 0$.
This is the desired conclusion.
Note that an upper bound for such $\lambda$ is suggested by Proposition \ref{prop:minimal_lambda}.

\subsection{Subdifferential and weak decomposability}\label{sec:weak_decomp}

In this subsection, we discuss the structure of the subdifferential of the nearly-isotonic type penalties.
The main purpose is to discuss the weak decomposability (defined in Appendix \ref{sec:penalized_risk_proof}) of $\mcV_-$.

\subsubsection{Characterization of the subdifferential}

First, we observe that $\mcV_-(\theta) = \sum_{i=1}^{n-1} (\theta_i - \theta_{i+1})_+$ can be written as a support function of a certain convex set.
In fact, by Theorem 8.24 in \citet{RW}, we can see that
\begin{equation}\label{eq:mcv_as_support_function}
    \mcV_-(\theta) = \max_{v \in \mcB} \langle v, \theta \rangle,
\end{equation}
where $\mcB$ is a closed convex set. Conversely, once we have a convex function $\mcV_-$, the set $\mcB$ is specified as
\[
    \mcB = \set{v \in \RR^n: \forall \theta \in \RR^n, \ \langle v, \theta \rangle \leq \mcV_-(\theta)}.
\]
Many properties of the support function can be understood through the structure of the set $\mcB$;
In particular, we can characterize the subdifferential and weak decomposability.
Below, we investigate the more detailed structure of the set $\mcB$ in terms of submodular functions.

Let $G = (V, E)$ be a directed graph equipped with positive edge weights $\set{c_{(i, j)}}$.
%Suppose that each edge $(i, j) \in E$ is equipped with a positive weight $c_{(i, j)} > 0$.
For any $\theta \in \RR^n$, we define a nearly-isotonic type penalty $\mcV_G(\theta)$ for the weighted graph $G$ as in \eqref{eq:neariso_penalty_general}.
%For any choices of $G$ and $c$, $\mcV_G$ becomes a convex function.
%Clearly, the lower total variation $\mcV_-$ is a special case where $E = \set{(i, i + 1) : i = 1, 2, \ldots, n - 1}$ and $c_{(i, i + 1)} \equiv 1$.
For any subset $A \subseteq [n]$, we also define $\kappa_G(A)$ by the total weights of outgoing edges:
\begin{equation}\label{eq:cut_general}
    \kappa_G(A) := \sum_{(i, j) \in E : \ i \in A, \ j \notin A} c_{(i, j)}.
\end{equation}
The function $A \mapsto \kappa_G(A)$ is called the cut function of the weighted graph $G$.

It is well known that the cut function is a submodular function.
Here, a function $F: 2^{[n]} \to \RR$ is called submodular if $F(\emptyset) = 0$ and
\[
    F(A) + F(B) \geq F(A \cap B) + F(A \cup B)
\]
holds for any subsets $A, B \subseteq [n]$.
We refer the reader to \citet{Bach13} for fundamental properties of submodular functions.
For any submodular function $F: 2^{[n]} \to \RR$, we define the base polyhedron $\mcB(F) \subseteq \RR^n$ as
\[
    \mcB(F) := \left\{
        v \in \RR^n: \sum_{i \in V} v_i = F(V)
        \ \text{and} \ \sum_{i \in A} v_i \leq F(A) \ \text{for all $A \subseteq V$}
    \right\}.
\]
The Lov\'{a}sz extension $f: \RR^n \to \RR$ of $F$ is defined as the support function of $\mcB(F)$, that is, for any $\theta \in \RR^n$,
$
    f(\theta) := \max_{v \in \mcB(F)} \langle v, \theta \rangle.
$

We see that the nearly-isotonic type penalty \eqref{eq:neariso_penalty_general} is actually the Lov\'{a}sz extension of the cut function \eqref{eq:cut_general}.

\begin{prop}
For any directed graph $G$ and edge weight $c_{(i, j)}$, the function $\mcV_G$ is the Lov\'{a}sz extension of the cut function $\kappa_G$.
\end{prop}

\begin{proof}
This is the consequence of the well-known result so-called the greedy algorithm; see e.g., Proposition 3.2 in \citet{Bach13}. In particular, we can find a derivation in Section 6.2 of \citet{Bach13}.
\if0
For the sake of completeness, we provide a detailed derivation. Let $\theta \in \RR^n$ be any vector, and let $\tau: [n] \to [n]$ be a permutation such that $\theta_{\tau(1)} \geq \cdots \theta_{\tau(n)}$. Define $A_$
Then, it is known that the Lov\'{a}sz extension of a submodular function $F: 2^{[n]} \to \RR$ is given as
\[
    f(\theta) =
    \sum_{i=1}^{n - 1} (\theta_{\tau(i)} - \theta_{\tau(i + 1)}) F(\set{\tau(1), \ldots, \tau(i)}).
\]
\fi
\end{proof}

Now, we have the following useful characterizations of the subdifferential.

\begin{prop}\label{prop:lovasz_subdifferential}
Define $F: 2^{[n]} \to \RR$ be a submodular function and $f: \RR^n \to \RR$ be its Lov\'{a}sz extension.
Suppose $\theta \in \RR^n$.
\begin{enumerate}[label=(\roman*)]
    \item The subdifferential $\partial f(\theta)$ coincides with a face of $\mcB(F)$ given as
    \[
        \partial f(\theta) = \argmax_{v \in \mcB(F)} \langle v, \theta \rangle = \set{v \in \mcB(F): \langle v, \theta \rangle = f(\theta)}.
    \]
    \item There is an (ordered) partition $(A_1, A_2, \ldots, A_k) \subseteq [n]$ such that
    \begin{equation}\label{eq:lovasz_affine_hull}
        \aff(\partial f(\theta))
        = \left\{
            v \in \RR^n: \sum_{j \in S_i} v_j = F(S_i) \ \text{for all $i = 1, 2, \ldots, k$}
        \right\},
    \end{equation}
    where $S_i := \bigcup_{j=1}^i A_j$ ($i = 1, 2, \ldots, k$).
    In particular, we have $\partial f(\theta) = \mcB(F) \cap \aff(\partial f(\theta))$.
    % rmk: converse holds if (A_1, \ldots, A_k) is inseparable
    \item Let $v$ be any point in the relative interior of $\partial f(\theta)$.
    Then, the normal cone of $\partial f(\theta)$ at $v$ is contained in the set of partition-wise constant vectors:
    \[
        N_{\partial f(\theta)}(v) \subseteq \Span \set{1_{A_1}, 1_{A_2}, \ldots, 1_{A_k}}.
    \]
\end{enumerate}
\end{prop}

\begin{proof}
The first statement is just a well-known property for the support function (Corollary 8.25 in \citet{RW}).
The second statement follows from the characterization of faces for the base polyhedron (see Proposition 4.7 in \citet{Bach13}).
The third statement follows from (ii) and the characterization of normal cones of polyhedra (see Theorem 6.46 in \citet{RW}).
\end{proof}

\if0
\begin{rmk}
We should note that Proposition \ref{prop:lovasz_subdifferential} holds for any submodular function $F$ and its Lov\'{a}sz extension since we do not use any specific properties of $\mcV_G$ in the proof.
\end{rmk}
\fi

\subsubsection{Weak decomposability}

Here, we discuss the weak decomposability of the Lov\'{a}sz extension.

Before describing the result, we introduce some terminology.
Let $F: 2^{[n]} \to \RR$ be a submodular function.
We say that a set $A \subseteq [n]$ is separable for $F$ if there is a non-empty proper subset $B$ of $A$ such that $F(A) = F(B) + F(A \setminus B)$.
We also say that $A$ is inseparable if it is not separable.
For example, if $F = \kappa_G$ is the cut function defined in \eqref{eq:cut_general}, $A$ is inseparable if and only if it is a connected component in the graph $G$.
Furthermore, we define the following \textit{agglomerative clustering condition}.

\begin{defi}\label{defi:agglomerative}
We say that a submodular function $F: 2^{[n]} \to \RR$ satisfies the agglomerative clustering (AC) condition if it has the following property:
Let $A, B \subseteq [n]$ be a any disjoint pair of subsets such that $A \neq \emptyset$ and $A$ is inseparable for the function $F_B^A: 2^A \to \RR$ defined by $F_B^A(C) := F(B \cup C) - F(B)$.
Then, for any $C \subset A$, we have
\begin{equation}\label{eq:defi_ac_condition}
    \frac{|C|}{|A|} (F(B \cup A) - F(B)) \leq F(B \cup C) - F(B).
\end{equation}
\end{defi}

Recall the definition of weak decomposability \eqref{eq:defi_weak_decomposable}.
The following proposition provides a sufficient condition for the weak decomposability of the Lov\'{a}sz extension.

\begin{prop}\label{prop:ac_implies_weak_decomp}
Let $F: 2^{[n]} \to \RR$ be a submodular function satisfying the AC condition in Definition \ref{defi:agglomerative}.
Then, the Lov\'{a}sz extension of $f$ of $F$ is weakly decomposable.
\end{prop}

\begin{proof}
Fix $\theta \in \RR^n$.
Since $f$ is the support function of the base polyhedron $\mcB(F)$, $\partial f(\theta)$ coincides with a face of $\mcB(F)$.
Let $A_1, A_2, \ldots, A_k$ be a partition of $[n]$ such that $\aff(\partial f(\theta))$ is represented as \eqref{eq:lovasz_affine_hull}.
For $i = 1, 2, \ldots, k$, we write $S_0 := \emptyset$ and $S_i := A_1 \cup A_2 \cup \cdots \cup A_i$.
We should note that the above partition can be chosen so that $A_i$ is inseparable for the function defined as
\[
    (A_i \supseteq) \  C \mapsto F(S_{i-1} \cup C) - F(S_{i-1}).
\]
In this case, $\partial f(\theta)$ is an $n - k$ dimensional subset.

Define a vector $v^*$ as
\begin{equation}\label{eq:lovasz_v_star}
    v^* := \sum_{i=1}^k \frac{F(S_i) - F(S_{i-1})}{|A_i|} 1_{A_i}.
\end{equation}
Since
\[
    \sum_{j \in S_i} v_j^* = \sum_{j = 1}^i (F(S_j) - F(S_{j-1})) = F(S_i)
\]
holds for any $i = 1, \ldots, k$, we have $v^* \in \aff(\partial f(\theta))$.
Moreover, $v^*$ is also contained in the normal cone of $\aff(\partial f(\theta))$.
Hence, if we prove $v^* \in \partial f(\theta)$, we have
\[
    \forall v \in \partial f(\theta),
    \quad \langle v^*, v - v^* \rangle = 0,
\]
which implies that $v^* \in \argmin_{v \in \partial f(\theta)} \norm{v}_2^2$.

Now, our goal is to prove $v^* \in \partial f(\theta)$ under the AC condition.
If $k = n$, then it is clear from \eqref{eq:lovasz_affine_hull} that $\partial f(\theta) = \set{v^*}$.
Below, we assume that $k < n$.
Since $v^* \in \aff(\partial f(\theta))$, it suffices to show that $\sum_{i \in S} v_i^* \leq F(S)$ holds for any $S \subseteq [n]$ that determines a relative boundary of $\partial f(\theta)$.
The relative boundary of $\partial f(\theta)$ can be written as the union of all $n - k - 1$ dimensional faces of $\mcB(F)$ that have non-empty intersection with $\partial f(\theta)$.
Such faces can be characterized as follows:
Let $\Pi = (A_1, A_2, \ldots, A_k)$ be the partition defined in the above, and choose $A_i$ with $|A_i| \geq 2$.
Let $A'_i$ be any non-empty proper subset of $A_i$.
We define a new ordered partition of $[n]$ by inserting $(A'_{i}, A_i \setminus A'_{i})$ instead of $A_i$:
\[
    \Pi' = (A_1, A_2, \ldots, A_{i-1}, A'_{i}, (A_i \setminus A'_i), A_{i+1}, \ldots, A_k).
\]
Then, $\Pi'$ defines an $n - k- 1$ dimensional affine subspace by \eqref{eq:lovasz_affine_hull}, which defines a part of the relative boundary of $\partial f(\theta)$.
Therefore, we have to show that $\sum_{i \in S} v_i^* \leq F(S)$ for any $S$ that can be written as $S = S_{i-1} \cup A'_i$ with $A'_i \subset A_i$.
From the AC condition, we have
\begin{align*}
    \sum_{i \in S} v_i^*
    & = \sum_{j=1}^k \frac{F(S_j) - F(S_{j-1})}{|A_j|} |A_j \cap S| \\
    & = \sum_{j = 1}^{i - 1} (F(S_j) - F(S_{j-1}))
    + \frac{F(S_{i - 1} \cup A'_i) - F(S_{i-1})}{|A_i|} |A'_i| \\
    & \leq F(S_{i - 1}) + (F(S_{i - 1} \cup A'_i) - F(S_{i-1})) \\
    & = F(S).
\end{align*}
This proves that $v^* \in \partial f(\theta)$, and hence $f$ is weakly decomposable.
\end{proof}

\begin{rmk}\label{rmk:on_agglomerative}
The AC condition was originally introduced in \citet{Bach11}.
In that paper, the author consider the proximal denoising estimators \eqref{eq:proximal_denoising} where $f$ is the Lov\'{a}sz extension of a submodular function $F$.
The name ``agglomerative clustering'' captures the following property:
Let us consider the \textit{solution path} of the minimization problem \eqref{eq:proximal_denoising} parametrized by $\lambda$,
that is, the solution path is the collection $\set{\hat{\theta}_\lambda}_{\lambda \geq 0}$ calculated for all $\lambda \geq 0$.
In general, the solution path starts with $\hat{\theta}_\lambda = y$ for $\lambda = 0$, and $\hat{\theta}_\lambda$ shrinks toward some piecewise constant vector as $\lambda$ increases.
Proposition 4 of \citet{Bach11} showed that the solution path is agglomerative if $F$ satisfies the AC condition.

We provide some examples of functions satisfying the AC condition:
\begin{itemize}
    \item Let $h: \RR \to \RR$ be a concave function with $h(0) = 0$.
    A submodular function defined as $F(A) := h(|A|)$ satisfies the AC condition.
    Examples of solutions paths for this class can be found in \citet{Bach11}.
    \item The one-dimensional fused lasso has an agglomerative solution path.
    The corresponding submodular function is the cut function of the undirected one-dimensional grid graph, which satisfies the AC condition.
    Hence, by Proposition \ref{prop:ac_implies_weak_decomp}, the penalty of the one-dimensional fused lasso is weakly decomposable.
    This provides an alternative proof for Lemma 2.7 in \citet{Guntuboyina2017a}.
    On the other hand, the fused lasso on the two-dimensional grid does not satisfy this condition.
    See \citet{Bach11} for details.
    \item The nearly-isotonic regression \eqref{eq:neariso} has an agglomerative solution path.
    %Figure \com{tbd} is an example of solution path of the nearly-isotonic regression with $n = 5$.
    A direct proof for this property is provided in Lemma 1 in \citet{Tibshirani2011}.
    Below, we prove that the cut function for directed one-dimensional grid graph satisfies the AC condition, which provides an alternative proof for this fact.
\end{itemize}
\end{rmk}

The following proposition provides a proof for Proposition \ref{prop:mcv_weak_decomp}.

\begin{prop}\label{prop:neariso_ac_condition}
The cut function $F$ associated with the nearly-isotonic regression satisfies the AC condition.
In particular, the lower total variation $\mcV_-(\theta)$ is weakly decomposable.
Moreover, for any $\theta \in \RR^n$, the minimum value of the $\ell_2$-norm in $\partial \mcV_-(\theta)$ is given by \eqref{eq:mcv_v0_norm}.
\end{prop}

\begin{proof}
For any $A \subseteq V := [n]$, $F(A)$ is given by the number of connected components in $A$ that does not contains the rightmost point $n$.
Let $A \subseteq [n]$ be a connected subset, and $B \subseteq [n] \setminus A$.
The value of $F(B \cup A) - F(B)$ depends on whether one or both of two endpoints of $A$ are adjacent to $B$.

We will check the AC condition by considering all patterns of adjacency as Table \ref{tb:adjacency}.
\begin{table}[h!]
\centering
\caption{The values of $F_B^A$ for the cut function $F$ of one-dimensional grid graph.}
\label{tb:adjacency}
\begin{tabular}{cccc}
    Node left to $A$ & Node right to $A$ & $F(B \cup A) - F(B)$ & $F(B \cup C) - F(B)$ \\ \hline
    None & None & 0 & $\geq  0$ \\
    None & $B$ & 0 & $\geq 0$ \\
    None & $V \setminus B$ & 1 & $\geq 1_{\set{C \neq \emptyset}}$ \\
    $B$ & None & -1 & $\geq 0$ \\
    $B$ & $B$ & -1 & $\geq 0$ \\
    $B$ & $V \setminus B$ & 0 & $\geq 0$ \\
    $V \setminus B$ & None & 0 & $\geq 0$ \\
    $V \setminus B$ & $B$ & 0 & $\geq 0$ \\
    $V \setminus B$ & $V \setminus B$ & 1 & $\geq 1_{\set{C \neq \emptyset}}$
\end{tabular}
\end{table}
Here, $C$ represents any proper subset of $A$, and ``None'' means that $A$ contains $1$ or $n$.
In each case, we can easily check that the inequality \eqref{eq:defi_ac_condition} is satisfied.
Hence, $F$ satisfies the AC condition.

The second statement is a consequence of Proposition \ref{prop:ac_implies_weak_decomp}.

The last statement follows from fact that the minimizer of $\norm{v}_2^2$ in $\partial f(\theta)$ coincides with that in $\aff(\partial f(\theta))$, which is given as \eqref{eq:lovasz_v_star}.
In this case, we can choose $A_1, A_2, \ldots, A_k$ as the constant partition of $\theta$ that is sorted by the values of $\theta$.
Thus, we have
\begin{equation}\label{eq:mcv_v_star}
    v^* = \sum_{i=1}^k \frac{F(S_i) - F(S_{i-1})}{|A_i|} 1_{A_i}
    = \sum_{i=1}^k \frac{1_{w_i \neq w_{i+1}}}{|A_i|} 1_{A_i}
\end{equation}
which proves the desired result.
\end{proof}

\begin{rmk}[Missing part in the proof of Proposition \ref{prop:mod_pava_validity}]\label{rmk:concavity_usage}
With a slight modification of the above argument, we can show the AC condition for the cut function of weighted graph
\[
    F(A) = \sum \set{c_j: j \in A, j+1 \notin A},
\]
where $c_j > 0$ ($j = 1, \ldots, n - 1$) are edge weights. As mentioned in Proposition \ref{prop:mod_pava_validity}, we need this result to prove the validity of the modified PAVA algorithm (Algorithm \ref{alg:modified_pava}). Here, we prove that \eqref{eq:mod_pava_validity_condition} provides a sufficient condition for the AC condition, and hence the solution path of the weighted nearly-isotonic regression \eqref{eq:weighted_neariso} is agglomerative.

Let $A \subseteq [n]$ be a non-empty connected subset, $B$ be a subset of $[n] \setminus A$, and $C$ be a proper subset of $A$. Recall that our goal is to check the inequality \eqref{eq:defi_ac_condition}. For clarity, we write $A = \set{j_L, j_L +1, \ldots, j_R}$. As in the proof of Proposition \ref{prop:neariso_ac_condition}, we consider all adjacency patterns of $A$, $B$ and $C$. Then, we can easily check the following case statement:
\begin{enumerate}
    \item Suppose that either ``$j_L = 1$ and $j_R + 1 \notin B$'' or ``$j_L - 1 \notin B$ and $j_R + 1 \notin B$'' holds. Then, we have $F(B \cup A) - F(B) = F(A) = c_{j_R}$ and $F(B \cup C) - F(B) = F(C)$. Now, we will check \eqref{eq:defi_ac_condition} under the concavity condition \eqref{eq:mod_pava_validity_condition}. First, \eqref{eq:defi_ac_condition} trivially holds when $j_R \in C$ because in this case $F(C) \geq c_{j_R} = F(A)$. Next, we assume $j_R \notin C$. Let $i$ be the largest element in $C$. Then, we have $F(C) \geq c_i$, $|C| \leq i - j_L + 1$. Under the assumption \eqref{eq:mod_pava_validity_condition}, we have
    \begin{align*}
        \frac{|C|}{|A|} F(A) & \leq \frac{i - j_L + 1}{j_R - j_L + 1} c_{j_R} \\
        & \leq \frac{i}{j_R} c_{j_R} \quad (\because j_L \leq i < j_R) \\
        & \leq c_i \quad (\because \eqref{eq:mod_pava_validity_condition}) \\
        & \leq F(C),
    \end{align*}
    which implies \eqref{eq:defi_ac_condition}.
    \item Suppose that $j_L - 1 \in B$ and $j_R + 1 \notin B$. Then, we have $F(B \cup A) - F(B) = c_{j_R} - c_{j_L - 1}$ and $F(B \cup C) - F(B) \geq F(C) - c_{j_L - 1}$. By a similar argument above, \eqref{eq:defi_ac_condition} trivially holds when $j_R \in C$. Let $j_R \notin C$ and let $i$ be the largest element in $C$. Then, under the assumption \eqref{eq:mod_pava_validity_condition}, we have
    \begin{align*}
        \frac{|C|}{|A|} (F(B \cup A) - F(B))
        & \leq \frac{i - j_L + 1}{j_R - j_L + 1} (c_{j_R} - c_{j_L - 1}) \\
        & \leq c_i - c_{j_L - 1} \quad (\because \eqref{eq:mod_pava_validity_condition}) \\
        & \leq F(C) - c_{j_L - 1} \\
        & \leq F(B \cup C) - F(B).
    \end{align*}
    \item For other case, we have $F(B \cup A) - F(B) \leq F(B \cup C) - F(B)$, which implies \eqref{eq:defi_ac_condition}.
\end{enumerate}

\if0
\begin{table}[h!]
\centering
\caption{The values of $F_B^A$ for the cut function $F$ of weighted one-dimensional grid graph.}
%\label{tb:adjacency}
\begin{tabular}{cccc}
    Node left to $A$ & Node right to $A$ & $F(B \cup A) - F(B)$ & $F(B \cup C) - F(B)$ \\ \hline
    None & None & 0 & $F(C)$ \\
    None & $B$ & 0 & $\geq 0$ \\
    None & $V \setminus B$ & $F(A)$ & $F(C)$ \\
    $B$ & None & $- F(B)$ & $\geq 0$ \\
    $B_\mathrm{left}$ & $B_\mathrm{right}$ & $- F(B_\mathrm{left})$ & $\geq 0$ \\
    $B$ & $V \setminus B$ & $c_a - c_{b}$ & $c_c - c_b$ \\
    $V \setminus B$ & None & 0 & $\geq 0$ \\
    $V \setminus B$ & $B$ & 0 & $\geq 0$ \\
    $V \setminus B$ & $V \setminus B$ & 1 & $\geq 1_{\set{C \neq \emptyset}}$
\end{tabular}
\end{table}
\fi
\end{rmk}

\section{Proofs in Section \ref{sec:model_selection}}

The goal of this section is to prove Theorem \ref{thm:model_selection}.
The outline of the proof is essentially the same as the framework of Theorem 4.18 in \citet{Massart2007}.
We explain this framework in Section \ref{sec:ms_proof_overview}.
To complete the proof, we have to control the maximum value of a certain normalized Gaussian process.
For this, we provide an upper bound in Section \ref{sec:ms_proof_peeling}.

\subsection{Proof overview}\label{sec:ms_proof_overview}

Let $(\hat{\Pi}, \hat{\mbV})$ be the selected pair in \eqref{eq:model_selection_estimator_sieve}.
Fix any connected partition $\Pi$ and $\mbV \in \mathscr{V}(|\Pi|)$.
By the definition of the estimator, we have
\begin{align*}
    \norm{y - \hat{\theta}_{\hat{\Pi}, \hat{\mbV}}}_2^2 + \mathrm{pen}(\hat{\Pi}, \hat{\mbV})
    & \leq \norm{y - \hat{\theta}_{\Pi', \mbV'}}_2^2 + \mathrm{pen}(\Pi', \mbV')\\
    & \leq \norm{y - \theta'}_2^2 + \mathrm{pen}(\Pi', \mbV')
\end{align*}
for any vector $\theta'$ that belongs to $K_{\Pi'}^\uparrow(\mbV')$.
In particular, we can choose $\theta'$ as
\[
    \theta' = \theta^*_{\Pi', \mbV'} := \argmin_{\theta' \in K_{\Pi'}^\uparrow(\mbV')} \norm{\theta' - \theta^*}_2.
\]
Substituting $y = \theta^* + \xi$, we can deduce that
\begin{equation}\label{eq:ms_proof_01}
    \norm{\theta^* - \hat{\theta}_{\hat{\Pi}, \hat{\mbV}}}_2^2
    \leq \norm{\theta^* - \theta^*_{\Pi', \mbV'}}_2^2
    - \mathrm{pen}(\hat{\Pi}, \hat{\mbV}) + \mathrm{pen}(\Pi', \mbV')
    + 2 \langle \hat{\theta}_{\hat{\Pi}, \hat{\mbV}} - \theta^*_{\Pi', \mbV'}, \ \xi \rangle.
\end{equation}
Here, recall that $\xi$ is a random variable drawn from $N(0, \sigma^2 I_n)$.

Let $z > 0$ be a positive number and $c \in (0, 1)$.
Suppose that an inequality
\begin{equation}\label{eq:ms_proof_02}
    \max_{\Pi} \sup_{\mbV \in \mathscr{V}(|\Pi|)}
    \sup_{\theta \in K_{\Pi}^\uparrow(\mbV)}
    \frac{\langle \theta - \theta^*_{\Pi', \mbV'}, \ \xi \rangle}
    {(\norm{\theta - \theta^*}_2+ \norm{\theta' - \theta^*}_2)^2 + \eta(\Pi, \mbV, z)}
    \leq \frac{c}{4}
\end{equation}
holds on some event $\Omega_z$ that occurs with probability at least $1 - \ee^{-z}$.
Here, $\eta(\Pi, \mbV, z) > 0$ is a positive constant that can depend on $\Pi, \mbV, z$.
Combining this inequality with \eqref{eq:ms_proof_01}, we have on the same event
\begin{equation}\label{eq:ms_proof_07}
    (1 - c) \norm{\theta^* - \hat{\theta}_{\hat{\Pi}, \hat{\mbV}}}_2^2
    \leq (1 + c) \norm{\theta^* - \theta^*_{\Pi', \mbV'}}_2^2
    - \mathrm{pen}(\hat{\Pi}, \hat{\mbV}) + \mathrm{pen}(\Pi', \mbV')
    + c \eta(\hat{\Pi}, \hat{\mbV}, z),
\end{equation}
where we used the elementary inequality $(a + b)^2 \leq 2(a^2 + b^2)$.
\if0
Then, if the penalty is chosen so that $\mathrm{pen}(\Pi, \mbV) \geq c \eta(\hat{\Pi}, \hat{\mbV}, z)$, we have
\[
    (1 - c) \norm{\theta^* - \hat{\theta}_{\hat{\Pi}, \hat{\mbV}}}_2^2
    \leq (1 + c) \norm{\theta^* - \theta^*_{\Pi', \mbV'}}_2^2
    + \mathrm{pen}(\Pi', \mbV').
\]
\fi

\subsection{Controlling the normalized process}\label{sec:ms_proof_peeling}

Now, our goal is to provide an inequality of the form \eqref{eq:ms_proof_02}.
Below, we fix $\theta' := \theta^*_{\Pi', \mbV'}$.

First, we fix a partition $\Pi$ and $\mbV \in \mathscr{V}(|\Pi|)$.
%We write $K := K_\Pi^\uparrow (\mbV)$ for notation simplicity.
For any $\theta \in K_\Pi^\uparrow(\mbV)$, we define
\[
    \omega (\theta) = \omega_{\Pi, \mbV}(\theta) :=
    (\norm{\theta - \theta^*}_2 + \norm{\theta' - \theta^*}_2)^2 + \eta,
\]
where $\eta > 0$ is a positive constant which will be specified later.
Define a random variable $Z_{\Pi, \mbV}$ as
\[
    Z_{\Pi, \mbV} := \sup_{\theta \in K_\Pi^\uparrow (\mbV)}
    \frac{\langle \theta - \theta', \xi \rangle}{\omega(\theta)}.
\]
Note that $Z_{\Pi, \mbV}$ is the supremum of a sample-continuous Gaussian process.
By the concentration inequality for Gaussian processes (Lemma \ref{lem:borel_tis}), we have
\begin{equation}\label{eq:ms_proof_03}
    \mathrm{Pr} \left\{
        Z_{\Pi, \mbV} - \EE[Z_{\Pi, \mbV}] \geq \sqrt{2v(x + z)}
    \right\} \leq \exp(-(x + z))
\end{equation}
for any $x > 0$ and $z > 0$.
Here, the variance $v$ is bounded as
\[
    v := \sup_{\theta \in K_\Pi^\uparrow (\mbV)} [Z_{\Pi, \mbV}^2] \leq \frac{\sigma^2}{4\eta}
\]
because $\omega(\theta) \geq \norm{\theta - \theta'}_2^2 + \eta \geq 2 \eta^{1/2} \norm{\theta - \theta'}_2$, and $\langle u, \xi \rangle$ is distributed according to $N(0, \sigma^2 \norm{u}_2^2)$ for any $u \in \RR^n$.

We will provide an upper bound for $\EE[Z_{\Pi, \mbV}]$.
Let $\theta^*_{\Pi, \mbV}$ be the orthogonal projection of $\theta^*$ onto $K_{\Pi}^\uparrow(\mbV)$.
Note that
\begin{equation}\label{eq:ms_proof_04}
    \EE[Z_{\Pi, \mbV}]
    \leq \underbrace{\EE \left[
        \sup_{\theta \in K_\Pi^\uparrow (\mbV)}
        \frac{\langle \theta - \theta^*_{\Pi, \mbV}, \ \xi \rangle}{\omega(\theta)}
    \right]}_{\text{(a)}}
    + \underbrace{\EE \left[
        \frac{|\langle \theta^*_{\Pi, \mbV} - \theta', \ \xi \rangle|}{\inf_{\theta \in K_{\Pi}^\uparrow(\mbV)} \omega(\theta)}
    \right]}_{\text{(b)}}.
\end{equation}
The second term (b) in the right-hand side of \eqref{eq:ms_proof_04} is bounded from above by $\sigma \eta^{-1/2}$.
Indeed, since
\[
    \inf_{\theta \in K_{\Pi}^\uparrow(\mbV)} \omega(\theta)
    = (\norm{\theta^*_{\Pi, \mbV} - \theta^*}_2 + \norm{\theta' - \theta^*}_2)^2 + \eta
    \geq 2 \eta^{1/2} \norm{\theta^*_{\Pi, \mbV} - \theta'}_2,
\]
we have
\[
    \text{(b)} \leq \frac{1}{2 \sqrt{\eta}} \EE_{u \sim N(0, \sigma^2)}[|u|] = \frac{\sigma}{\sqrt{2 \pi \eta}}.
\]
To bound the term (a) in \eqref{eq:ms_proof_04}, we use the following lemma:

\begin{lem}\label{lem:ms_proof_peeling}
Let $\Pi = (A_1, A_2, \ldots, A_m)$ be any partition and $\mbV = (\mcV_1, \mcV_2, \ldots, \mcV_m)$.
Fix any $\bar{\theta} \in K_\Pi^\uparrow(\mbV)$.
For any $t > 0$, we have
\begin{equation}\label{eq:ms_proof_peeling_1}
    \EE \left[
        \sup_{\theta \in K_\Pi^\uparrow(\mbV): \norm{\theta - \bar{\theta}}_2 \leq t}
        \langle \xi, \theta - \bar{\theta} \rangle
    \right]
    \leq
    C \sigma t^{1/2} \left(
        \sum_{i=1}^m |A_i|^{1/3} \mcV_i^{2/3}
    \right)^{3/4}
    + C \sigma t \sqrt{m \log \frac{\ee n}{m}},
\end{equation}
where $C > 0$ is a universal constant.
Futhermore, for any $\eta > 0$, we have
\begin{equation}\label{eq:ms_proof_peeling_2}
    \EE \left[
        \sup_{\theta \in K_\Pi^\uparrow(\mbV)}
        \frac{\langle \theta - \bar{\theta}, \ \xi \rangle}{\norm{\theta - \bar{\theta}}_2 + \eta}
    \right]
    \leq 4 C \sigma \left\{
        \eta^{- 3/4} \left(
            \sum_{i=1}^m |A_i|^{1/3} \mcV_i^{2/3}
        \right)^{3/4}
        + \eta^{- 1/2} \sqrt{m \log \frac{\ee n}{m}}
    \right\},
\end{equation}
where $C$ is the same constant as in \eqref{eq:ms_proof_peeling_1}.
\end{lem}

\begin{proof}
We will prove the first inequality \eqref{eq:ms_proof_peeling_1}.
Let $W := W(\Pi, \mbV)$ denote the left-hand side of \eqref{eq:ms_proof_peeling_1}.
We consider a collection of finitely many sets $S(\mathbf{q})$ as follows:
Let $\mcQ := \mcQ(m)$ be a collection of vectors $\mathbf{q} = (q_1, q_2, \ldots, q_m)$ that can be written as $\mathbf{q} = t^2 \mathbf{a}/m$ for some integer vector $\mathbf{a} = (a_1, a_2, \ldots, a_m)$ such that $1 \leq a_i \leq m$ and $\sum_{i=1}^m a_i \leq 2m$.
Note that, by Proposition \ref{prop:cardinality}, the cardinality of $\mcQ$ is bounded by $(2 \ee)^m$.
For any $\mathbf{q} \in \mcQ$, define the set
\[
    S(\mathbf{q}) := \left\{
        \theta  \in \RR^n: \norm{\theta_{A_i}}_2^2 \leq q_i, \ 
        \mcV^{A_i}(\theta_{A_i}) \leq 2 \mcV_i \
        \text{for all $A_i \in \Pi$}
    \right\}.
\]
Then, we can easily check that
\[
    K_\Pi^\uparrow(\mbV) \cap \set{\theta \in \RR^n: \norm{\theta - \bar{\theta}}_2 \leq t}
    \subseteq \bigcup_{\mathbf{q} \in \mcQ} S(\mathbf{q}).
\]

From Lemma \ref{lem:guntuboyina_b1} below, there exists a universal constant $C > 0$ such that
\begin{equation}\label{eq:ms_proof_peeling_3}
    \EE \left[
        \sup_{\theta \in S(\mathbf{q})} \langle \theta, \xi \rangle
    \right]
    \leq C\sigma \sum_{i=1}^m \left\{
        \sqrt{2} q_i^{1/4} |A_i|^{1/4} \mcV_i^{1/2}
        + q_i^{1/2} \sqrt{\log \ee |A_i|}
    \right\}.
\end{equation}
Here, by H\"{o}lder's inequality, we have
\[
    \sum_{i=1}^m q_i^{1/4} |A_i|^{1/4} \mcV_i^{1/2}
    \leq \left(
        \sum_{i=1}^m q_i
    \right)^{1/4} \left(
        \sum_{i=1}^m (|A_i|^{1/4} \mcV_i^{1/2})^{4/3}
    \right)^{3/4}
    \leq 2^{1/4} t^{1/2} \left(
        \sum_{i=1}^m |A_i|^{1/3} \mcV_i^{2/3}
    \right)^{3/4},
\]
and by the Cauchy-–Schwarz inequality, we also have
\[
    \sum_{i=1}^m 2 q_i^{1/2} \sqrt{\log \ee |A_i|}
    \leq 2 \sqrt{2} t\left(\sum_{i=1} \log \ee |A_i| \right)^{1/2}
    \leq 2 \sqrt{2} t\sqrt{m \log \frac{\ee n}{m}}.
\]
Then, by Lemma \ref{lem:guntuboyina_d1} below, we have
\begin{align*}
    W & \leq \max_{\mathbf{q} \in \mcQ} \EE \left[
        \sup_{v \in S(\mathbf{q})} \langle \xi, v \rangle
    \right]
    + 2t \sigma \left(
        \sqrt{2 \log |\mcQ|} + \sqrt{\frac{\pi}{2}}
    \right)\\
    & \leq
    C \sigma \left\{
        2^{3/4} t^{1/2} \left(
            \sum_{i=1}^m |A_i|^{1/3} \mcV_i^{2/3}
        \right)^{3/4}
        + 2 \sqrt{2} t\sqrt{m \log \frac{\ee n}{m}}
    \right\}
    + 2t \sigma \left(
        \sqrt{4 m \log 2\ee} + \sqrt{\frac{\pi}{2}}
    \right) \\
    & \leq C' \sigma \left\{
        t^{1/2} \left(
            \sum_{i=1}^m |A_i|^{1/3} \mcV_i^{2/3}
        \right)^{3/4}
        + t \sqrt{m \log \frac{\ee n}{m}}
    \right\}
\end{align*}
for some $C' > 0$.
Thus, \eqref{eq:ms_proof_peeling_1} has been proved.

The second inequality \eqref{eq:ms_proof_peeling_2} is a consequence of the peeling lemma (Lemma \ref{lem:peeling} below).
\end{proof}

Combining \eqref{eq:ms_proof_03}, \eqref{eq:ms_proof_04} and \eqref{eq:ms_proof_peeling_2}, we conclude that
\begin{align}\label{eq:ms_proof_05}
    Z_{\Pi, \mbV} & \leq
    4 C \sigma \eta^{- 3/4} \left(
            \sum_{i=1}^m |A_i|^{1/3} \mcV_i^{2/3}
        \right)^{3/4} \nonumber  \\
    & + \sigma \eta^{- 1/2} \left\{
        4 C \sqrt{m \log \frac{\ee n}{m}}
        + (2\pi)^{-1/2} + 2^{-1/2} \sqrt{x + z}
    \right\}
\end{align}
holds with probability at least $1 - \exp(- (x + z))$, where $C$ is the constant in \eqref{eq:ms_proof_peeling_2}.
Now, we choose the two constant $\eta := \eta(\Pi, \mcV, z)$ and $x := x(\Pi, \mcV)$ as
\[
    \eta(\Pi, \mcV, z)
    := 2^8 (4C + 1)^{4/3} \sum_{i=1}^m \sigma^{4/3}|A_i|^{1/3} \mcV_i^{2/3}
    + 2^8 (4C + 2)^2 \sigma^2 m \log \frac{\ee n}{m}
    + 2^8 \sigma^2 z
\]
and
\[
    x(\Pi, \mcV)
    := \sum_{i=1}^m \sigma^{- 2/3}|A_i|^{1/3} \mcV_i^{2/3}
    + 2 m \log \frac{\ee n}{m},
\]
respectively.
Then, it is elementary to check that the right-hand side of \eqref{eq:ms_proof_05} is not larger than $1/8$.

Applying the union bound over all pairs $(\Pi, \mbV)$, we have
\begin{align*}
    \mathrm{Pr} \left\{
        \max_{\Pi} \sup_{\mbV \in \mathscr{V}(|\Pi|)} Z_{\Pi, \mbV} > \frac{1}{8}
    \right\}
    & \leq
    \exp(-z) \sum_{\Pi} \sum_{\mbV} \exp(-x (\Pi, \mbV)).
\end{align*}
Here, we can show that
\begin{equation}\label{eq:ms_proof_06}
    \sum_{\Pi} \sum_{\mbV} \exp(-x (\Pi, \mbV)) \leq 1,
\end{equation}
and hence we conclude that \eqref{eq:ms_proof_02} holds with $c = 1/2$.
Indeed, \eqref{eq:ms_proof_06} follows from the fact that, for any $\Pi$, 
\begin{align*}
    \sum_{\mbV \in \mathscr{V}(\Pi)} \exp \left(
        - \sum_{i=1}^m \sigma^{- 2/3}|A_i|^{1/3} \mcV_i^{2/3}
    \right)
    & = \prod_{i=1}^m \exp \left(
        - \sigma^{- 2/3}|A_i|^{1/3}
    \right) \left(
        \sum_{j_i = 1}^\infty \ee^{-j_i}
    \right) \\
    & \leq \exp \left(
        - \sum_{i=1}^m \sigma^{- 2/3}|A_i|^{1/3}
    \right) \leq 1
\end{align*}
and
\begin{align*}
    \sum_{\Pi} \exp\left( - 2|\Pi| \log \frac{\ee n}{|\Pi|} \right)
    & = \sum_{m=1}^n \sum_{\Pi: |\Pi| = m} \exp\left( - 2m \log \frac{\ee n}{m} \right)\\
    & \leq \sum_{m=1}^n \sum_{\Pi: |\Pi| = m} \exp\left( - m - \log \binom{n-1}{m-1} \right) \\
    & = \sum_{m=1}^n \ee^{-m} \leq 1.
\end{align*}

\subsection{Proof of Theorem \ref{thm:model_selection}}\label{sec:ms_proof_complete}

Now, we are ready to complete the proof of Theorem \ref{thm:model_selection}.
Define $\mathrm{pen}(\Pi, \mbV)$ as
\[
    2^7 (4C + 1)^{4/3} \sum_{i=1}^m \sigma^{4/3}|A_i|^{1/3} \mcV_i^{2/3}
    + 2^7 (4C + 2)^2 \sigma^2 m \log \frac{\ee n}{m},
\]
where $C$ is the constant in \eqref{eq:ms_proof_peeling_2}.
Let $(\Pi', \mbV')$ be the pair that minimizes
\[
    (\Pi, \mbV) \mapsto \frac{3}{2} \norm{\theta^* - \theta^*_{\Pi, \mcV}}_2^2 + \mathrm{pen}(\Pi, \mbV)
\]
among all possible pairs.
Applying \eqref{eq:ms_proof_07} and \eqref{eq:ms_proof_02} for this choice of $(\Pi', \mbV')$, we conclude that
\[
    \norm{\hat{\theta}_{\hat{\Pi}, \hat{\mbV}} - \theta^*}_2^2
    \leq \min_{\Pi} \min_{\mbV \in \mathscr{V}(|\Pi|)} \left\{
        3 \dist^2(\theta^*, K_{\Pi}^\uparrow(\mbV))
        + 2 \mathrm{pen}(\Pi, \mbV)
    \right\}
    + 2^8 \sigma^2 z
\]
holds with probability at least $1 - \exp(-z)$.
Moreover, by integrating both sides with respect to $z$, we have
\[
    \EE_{\theta^*} \norm{\hat{\theta}_{\hat{\Pi}, \hat{\mbV}} - \theta^*}_2^2
    \leq \min_{\Pi} \min_{\mbV \in \mathscr{V}(|\Pi|)} \left\{
        3 \dist^2(\theta^*, K_{\Pi}^\uparrow(\mbV))
        + 2 \mathrm{pen}(\Pi, \mbV)
    \right\} + 2^8 \sigma^2.
\]

\section{Auxiliary lemmas}

Here, we present several auxiliary lemmas that are used in the proofs in the previous sections.

\begin{lem}[Borel--Tsirelson--Ibragimov--Sudakov inequality; see Proposition 3.19 in \citet{Massart2007}]\label{lem:borel_tis}
Suppose that $(X_t)_{t \in T}$ is a Gaussian process on a totally bounded metric space $(T, d)$ such that $\EE[X_t] = 0$ for any $t \in T$ and the sample path $t \mapsto X_t$ is almost surely continuous.
Let $v := \sup_{t \in T} \EE[X^2_t]$.
Then, for any $z > 0$, we have
\[
    \mathrm{Pr} \left\{
        \sup_{t \in T} X_t - \EE \left[ \sup_{t \in T} X_t \right]
        \geq \sqrt{2vz}
    \right\} \leq \exp(-z).
\]
\end{lem}

\begin{lem}[Peeling lemma; see e.g. Lemma 4.23 in \citet{Massart2007}]\label{lem:peeling}
Let $K$ be a set in $\RR^n$ and $\bar{\theta} \in K$.
Assume that there is a function $\psi: [0, \infty) \to \RR$ such that $\psi(t) / t$ is non-increasing and
\[
    \EE_{\xi \sim N(0, I_n)} \left[
        \sup_{\theta \in K: \norm{\theta - \bar{\theta}}_2 \leq t} \langle \xi, \theta - \bar{\theta} \rangle
    \right]
    \leq \psi(t)
\]
for any $t \geq \bar{t} \geq 0$.
Then, for any $x\geq \bar{t}$, we have
\[
    \EE_{\xi \sim N(0, I_n)} \left[
        \sup_{\theta \in K: \norm{\theta - \bar{\theta}}_2 \leq t} \frac{\langle \xi, \theta - \bar{\theta} \rangle}{\norm{\theta - \bar{\theta}}_2^2 + x^2}
    \right]
    \leq \frac{4\psi(x)}{x^2}.
\]
\end{lem}

\begin{lem}[\citet{Guntuboyina2017a}, Lemma B.1]\label{lem:guntuboyina_b1}
For any $t > 0$ and $\mcV > 0$, let
\[
    S(V, t) := \set{\theta \in \RR^n: \mcV(\theta) \leq \mcV \ \text{and} \ \norm{\theta}_2 \leq t}.
\]
There exists a universal constant $C > 0$ such that
\[
    \EE_{\xi \sim N(0, \sigma^2 I_n)} \left[
        \sup_{\theta \in S(V, t)} \langle \theta, \xi \rangle
    \right]
    \leq C \sigma t^{1/2} n^{1/4} \mcV^{1/2} + C \sigma t \sqrt{\log \ee n}.
\]
\end{lem}

\begin{lem}[\citet{Guntuboyina2017a}, Lemma D.1]\label{lem:guntuboyina_d1}
Suppose $p, n \geq 1$ and let $\Theta_1, \ldots, \Theta_p$ be subset of $\RR^n$ each containing the origin and each contained in the closed Euclidean ball of radius $D$ centered at the origin. Then, for $\xi \sim N(0, \sigma^2 I)$, we have
\begin{equation}
    \EE \left[
        \max_{1 \leq i \leq p} \sup_{\theta \in \Theta_i} \langle \xi, \theta \rangle
    \right]
    \leq \max_{1 \leq i \leq p}\EE \left[
        \sup_{\theta \in \Theta_i} \langle \xi, \theta \rangle
    \right]
    + D \sigma \left(
    \sqrt{2 \log p} + \sqrt{\frac{\pi}{2}}
    \right).
\end{equation}
\end{lem}

\section*{Acknowledgment}
This work was supported by JSPS KAKENHI Grant Number JP17J06640. The author would like to thank three anonymous reviewers for their valuable comments and suggestions. The author also thanks Hiromichi Nagao for suggesting the example of a seismological phenomenon, and Fumiyasu Komaki and Keisuke Yano for helpful discussions.

\bibliographystyle{plainnat}
\bibliography{main}

\end{document}